\documentclass[12pt, leqno]{article}
\usepackage{fullpage} 
\usepackage{amsmath}
\usepackage{amssymb}
\usepackage{natbib}
\usepackage{undertilde}
\usepackage{stmaryrd}
\usepackage{bbding}
\usepackage{enumerate}
\usepackage{graphicx}
\usepackage{xypic}
\xyoption{curve}
\usepackage{multicol}
\xyoption{color}
\usepackage{setspace}
\usepackage{boxedminipage}
\usepackage{nicefrac}
\usepackage{multicol}
\usepackage{amsthm}
  \usepackage[pdftex,
  bookmarks = false,
  pdfstartview = FitBH,
  linktocpage = true,
  pagebackref = true,
  pdfhighlight = /O,
  colorlinks=true,
  linkcolor=blue,
  citecolor=blue,
  filecolor = blue,
  urlcolor = blue,
  menucolor = blue,
]{hyperref}

\usepackage{xypic}
\xyoption{curve}
\usepackage{boxedminipage}
\usepackage{multicol}
\xyoption{color}
\newcommand{\id}{\mathrm{id}}

\usepackage[pagewise, displaymath, mathlines]{lineno}

\newtheorem{thm}{Theorem}[section]
\newtheorem{prop}[thm]{Proposition}
\newtheorem{definition}[thm]{Definition}

\numberwithin{equation}{section}
\numberwithin{enumi}{section}
\numberwithin{foo}{section}
\numberwithin{thm}{section}

\title{Realizability Semantics for Quantified Modal Logic: \\ Generalizing Flagg's 1985 Construction}

 \author{Benjamin G. Rin\footnote{Institute for Logic, Language and Computation at the University of Amsterdam, and Amsterdam University College. \emph{Email}: benjamin.rin@gmail.com} \; and Sean Walsh\footnote{Department of Logic and Philosophy of Science, University of California, Irvine. \emph{Email}: swalsh108@gmail.com or walsh108@uci.edu}
}

\begin{document}

\maketitle

\begin{abstract}
A semantics for quantified modal logic is presented that is based on Kleene's notion of realizability. This semantics generalizes Flagg's 1985 construction of a model of a modal version of Church's Thesis and first-order arithmetic. While the bulk of the paper is devoted to developing the details of the semantics, to illustrate the scope of this approach, we show that the construction produces (i) a model of a modal version of Church's Thesis and a variant of a modal set theory due to Goodman and Scedrov, (ii) a model of a modal version of Troelstra's generalized continuity principle together with a fragment of second-order arithmetic, and (iii) a model based on Scott's graph model (for  the untyped lambda calculus) which witnesses the failure of the stability of non-identity.
\end{abstract}


\newpage

\section{Introduction}\label{sec:intro}

Realizability is the method devised by \cite{Kleene1945aa} to provide a semantics for intuitionistic first-order number theory. In the course of its long history (cf. \cite{Oosten2002aa}), realizability was subsequently generalized to fragments of second-order number theory and set theory by  \cite{Troelstra1998aa}, \cite{McCarty1984aa,McCarty1986aa}, and others. The primary aim of this paper is to present a systematic way of transforming these semantics into an associated semantics for systems of classical modal logic, so that we obtain modal systems of first-order arithmetic, fragments of second-order arithmetic, and set theory. The resulting systems validate not intuitionistic logic but rather classical logic, so that while intuitionistic logic is used in the construction of these modal systems, the logic of these modal systems is thoroughly classical.

The resulting semantics generalize the important but little-understood construction of \cite{Flagg1985aa}, whose goal was to provide a consistency proof of Epistemic Church's Thesis together with epistemic arithmetic, a modal rendition of first-order arithmetic. \emph{Epistemic Church's Thesis} ($\mathsf{ECT}$) is the following statement:
\begin{equation}\label{eqn:whatwegotnow}
[\Box (\forall \; n \; \exists \; m \;  \Box \varphi(n,m))]\Rightarrow  [\exists \; e \; \Box \; \forall \; n \; \exists \; m \; \exists \; q \; (T(e,n,q) \wedge U(q,m) \wedge \Box \varphi(n,m))]
\end{equation}
In this, the quantifiers are understood to range over natural numbers, and $T(e,n,q)$ is Kleene's $T$-predicate, which says that program~$e$ on input $n$ halts and outputs a code for the halting computation $q$, while $U(q,m)$ says that the computation $q$ has output value~$m$. The modal operator takes on an epistemic interpretation due to \cite{Shapiro1985ac}, whereby $\Box\varphi$ represents ``$\varphi$ is knowable''. $\mathsf{ECT}$ then expresses the computability of any number-theoretic function which can be known to be total in the admittedly strong sense that it's knowable that the value of this function is knowable for each input.\footnote{There is some regrettable clash in terminology which ought to be mentioned at the outset. The modal principle $\mathsf{ECT}$ as defined in~(\ref{eqn:whatwegotnow}) ought not be confused with the non-modal principle known as \emph{extended Church's Thesis}  sometimes abbreviated  similarly. For a statement of extended Church's Thesis, see \citet[volume 1 p. 199]{Troelstra1988aa}.}

The extension of Flagg's construction sits well with some of the original philosophical motivations for $\mathsf{ECT}$ and epistemic arithmetic. Shapiro's motivation was to have an object-language in which certain ``pragmatic'' properties of computable functions could be expressed, where a property is ``pragmatic'' if an object has it or lacks it ``in virtue of human abilities, achievements, or knowledge, often idealized''  (\citet[p. 41]{Shapiro1985ac}, \citet[pp. 61-62]{Shapiro1993ab}). Reinhardt's interest stemmed from the observation that $\mathsf{ECT}$ together with $\Box \; \forall \; n \; (\Box \; \theta(n) \vee \Box \; \neg \theta(n))$ implies that~$\theta(n)$ is recursive, thereby expressing the idea that epistemically decidable predicates are as rare as recursive predicates \cite[ p. 185]{Reinhardt1985ad}. Given either of these motivations, one would obviously want to know what happens when the mathematical background is changed from arithmetic to other domains, and thus it's natural to seek to understand the extent to which Flagg's construction may be generalized.

Our semantics generalizes Flagg's work by distinguishing between two roles played by arithmetic in his original construction: on the one hand, arithmetic is used to formalize the standard notion of computation on the natural numbers, and on the other hand arithmetic is used to provide the domain and the interpretation of the non-logical arithmetic primitives. On our approach, the notion of computation is generalized to the setting of partial combinatory algebras (cf. \S~\ref{sec:heytingfrompca}), which roughly are algebraic structures capable of formalizing the elementary parts of computability theory such as the $s\mbox{-}m\mbox{-}n$ theorem, the enumeration theorem, the recursion theorem, etc. These algebras are used to construct a space of truth-values, and our semantics then maps modal sentences to elements of this space. Thus the order of explanation in our work is the reverse of that found in earlier work on $\mathsf{ECT}$ and epistemic arithmetic: whereas the earlier work attempted to use modal logic to explicate the notion of computability, our work uses certain notions from computability to aid in the explication of modality, at least in the minimal sense of providing a semantics for it.

While there are a number of different axiomatic systems for quantified modal logic, in what follows we only need to work with a small number of them. Let's define $Q^{\circ}.\mathsf{K}$ to be the Hilbert-style deductive system for the basic modal predicate system $\mathsf{K}$, as set out in \citet[pp. 133-134]{Fitting1998aa}  and \citet[p. 1487]{Corsi2002aa}. See the proof of Theorem~\ref{prop:S4soundness} below for an explicit listing of the axioms of $Q^{\circ}.\mathsf{K}$. The system $Q^{\circ}.\mathsf{K}+\mathsf{CBF}$ is then simply $Q^{\circ}.\mathsf{K}$ plus the Converse Barcan Formula~$\mathsf{CBF}$:
\begin{equation}\label{eqn:CBF}
\Box \; \forall \; x \; \varphi(x) \; \Rightarrow \; \forall \; x \; \Box \; \varphi(x)
\end{equation}
The system $Q.\mathsf{K}$ (with no $\circ$ superscript) results from the system $Q^{\circ}.\mathsf{K}$ by replacing the \textsc{Universal Instantiation Axiom} 
\begin{equation}\label{eqn:UniversalInstantiationAxiom}
\forall \; y \;  ((\forall \; x \; \varphi(x)) \; \Rightarrow \; \varphi(y))
\end{equation}
with its free-variable variant:
\begin{equation}\label{eqn:UniversalInstantiationAxiom2}
(\forall \; x \; \varphi(x)) \Rightarrow \varphi(y)
\end{equation}
As is well known, $Q^{\circ}.\mathsf{K}$ does not prove $\mathsf{CBF}$, but $Q.\mathsf{K}$ does \cite[pp. 245-246]{Hughes1996aa}. Finally, if $\mathsf{L}$ is any set of propositional modal axioms, then let $Q^{\circ}.\mathsf{L}$ be the system $Q^{\circ}.\mathsf{K}$ plus the $\mathsf{L}$-axioms, and similarly for $Q.\mathsf{L}$. In what follows, we'll work almost exclusively with $Q^{\circ}.\mathsf{S4}+\mathsf{CBF}$ and $Q.\mathsf{S4}$, where $\mathsf{S4}$ refers as usual to the $\mathsf{T}$-axiom $\Box \varphi \Rightarrow \varphi$ and the $\mathsf{4}$-axiom $\Box \varphi \Rightarrow \Box \Box \varphi$. In particular, we appeal repeatedly to the common theorem of $Q^{\circ}.\mathsf{S4}+\mathsf{CBF}$ and $Q.\mathsf{S4}$ that $\Box \; \forall \; x \; \varphi(x)$ and $\Box \; \forall \; x \; \Box \; \varphi(x)$ are equivalent.

The final modal notation that we need pertains to identity, existence, and stability. We expand the system $Q^{\circ}.K$ (resp. $Q.K$) to the system $Q^{\circ}_{eq}.K$ (resp. $Q_{eq}.K$) by adding the following axioms, which respectively express the reflexivity of identity and the indiscernibility of identicals, and wherein $s,t$ range over terms in the signature:
\begin{align}
t & = t \label{eqn:necid} \\
s = t & \Rightarrow (\varphi(t) \Rightarrow \varphi(s)) \label{eqn:subs}
\end{align}
Further, if $x$ is a variable, then let the existence predicate $\mathsf{E}(x)$ be defined by $\exists \; y \; y=x$, where $y$ is a variable chosen distinct from $x$. Finally, if $\Gamma$ is a class of formulas, then \emph{$\Gamma$-stability} is the following schema, wherein $\varphi$ is from $\Gamma$:
\begin{equation}
\Box \; [\forall \; \overline{x} \; (\varphi(\overline{x}) \Rightarrow \Box \;\varphi(\overline{x}))]
\end{equation}
Most often in what follows we'll focus on the stability of atomic formulas and their negations. Finally, it's perhaps worth mentioning that our notation for the systems $Q^{\circ}_{eq}.K$ and $Q_{eq}.K$ is patterned after \cite{Corsi2002aa}, who additionally assumes the stability of the negation of identity, which in general fails on our semantics. See the systems $Q^{\circ}_{=}.K$ and $Q_{=}.K$ on \citet[p. 1498]{Corsi2002aa}.

The systems of first-order arithmetic with which we shall work are Heyting arithmetic $\mathsf{HA}$, Peano arithmetic $\mathsf{PA}$, and two versions of epistemic arithmetic which we call $\mathsf{EA}^{\circ}$ and $\mathsf{EA}$. The first two are sufficiently well known from standard references such as \citet[volume 1 p. 126]{Troelstra1988aa} and \citet[p. 28]{Hajek1998}, although it's perhaps useful to stipulate that in what follows the signature of both shall be taken to be $L_{0}=\{0,S, f_1, f_2,\ldots\}$, wherein $0$ and $S$ stand respectively for zero and successor and $f_i$ is an enumeration of function symbols for the primitive recursive functions. Thus we take $x<y$ to be an abbreviation for $f(x,y)=0$ for a suitable primitive recursive function \cite[volume 1 p. 124 Definition 2.7]{Troelstra1988aa}. Then we define:
\begin{definition}\label{defn:EA}
The system $\mathsf{EA}^{\circ}$ is $Q^{\circ}_{eq}.\mathsf{S4}$ plus $\mathsf{PA}$, and the system $\mathsf{EA}$ is $Q_{eq}.\mathsf{S4}$ plus $\mathsf{PA}$, where in both cases we include induction axioms for all formulas in the extended modal language.
\end{definition}
\noindent These definitions in place, we can now state what we take our generalization of Flagg's construction to establish in the setting of first-order arithmetic:
\begin{thm}\label{thm:Flaggthm}
The theory consisting of $\mathsf{EA}^{\circ}$, $\mathsf{CBF}$, and $\mathsf{ECT}$ is consistent with both the failure of $\mathsf{EA}$ and the failure of the Barcan Formula~$\mathsf{BF}$:
\begin{equation}\label{eqn:BF}
[\forall \; x \; \Box \; \varphi(x)]\Rightarrow [\Box \; \forall \; x \;\varphi(x)]
\end{equation}
\end{thm}
\noindent The proof of this theorem is presented in \S\ref{sec:arithmetic}, and follows immediately from Theorem~\ref{thm:EA} and Theorem~\ref{thm:ECTisvalid} and Proposition~\ref{prop:failureofbarcan}. Further, as we note in Theorem~\ref{thm:EA}, the existence predicate is a counterexample to~(\ref{eqn:UniversalInstantiationAxiom2}). This admittedly clashes with Flagg's own statement of his results, as he takes himself to be working with $\mathsf{EA}$ instead of $\mathsf{EA}^{\circ}$. We discuss this issue more extensively immediately following Theorem~\ref{thm:ECTisvalid}. Finally, despite the centrality of the Barcan Formula~$\mathsf{BF}$~(\ref{eqn:BF}) to quantified modal logic, to our knowledge neither Flagg nor any of the commentators on his results had indicated whether his model validated this formula, and so it seems fitting to record this information in the statement of the above theorem. However, see \citet[p. 20]{Shapiro1985ac} for a discussion of the status of the Barcan Formula~$\mathsf{BF}$ in $\mathsf{EA}$ itself.

One distinctive feature of our Flagg-like construction is that it always produces the stability of atomics (cf. Proposition~\ref{prop:atomics}~(i)). In the set-theoretic case, the membership relation is atomic and so any set-theoretic construction \emph{\`a la} Flagg will induce the stability of the membership relation. This places strong constraints on the types of set theory that can be shown to be consistent with $\mathsf{ECT}$ using these methods. For, consider any non-computable set of natural numbers such as the halting set $X=\emptyset^{\prime}$, and let $Y=\omega\setminus X$ be its relative complement in the natural numbers. Then of course we have $\forall \; n\in \omega \; (n\in X \vee n\in Y)$. Then the stability of the membership relation implies that $\forall \; n\in \omega \; \exists \; y\in \omega \; \Box \; ( (n\in X\wedge y=1) \vee (n\in Y \wedge y=0))$. But then $\mathsf{ECT}$ implies that $X$'s characteristic function is computable. The moral of this elementary observation is that if one wants to show the consistency of $\mathsf{ECT}$ with a modal set theory, and one proceeds by a construction which forces atomics to be stable, then the modal set theory has to be something far different from just the usual set-theoretic axioms placed on top of the quantified modal logic.

  In what follows, we rather focus on a variant~$\mathsf{eZF}$ of Goodman's epistemic theory~$\mathsf{EZF}$ \cite[pp. 155-156]{Goodman1990aa}. Goodman's theory includes a version of the axiom of choice, and without this axiom, it is equivalent to Scedrov's epistemic set theory $\mathsf{ZFE}_C$ \cite[p. 749-750]{Scedrov1986aa}. This system is the successor to other versions of intensional or epistemic set theories proposed by Myhill, Goodman, and Scedrov in their contributions to the volume \cite{Shapiro1985ad} which also contains Flagg's original paper. Since $\mathsf{EZF}$ is less familiar, let's begin by briefly reviewing this system. This system is built from $Q_{eq}.\mathsf{S4}$ by the addition of the following axioms:

\vspace{2mm}

\noindent I. \emph{Modal Extensionality}: $\forall x \; \forall \; y \; (\Box (\forall \; z \; (z\in x \Leftrightarrow z\in y)) \Rightarrow x=y)$

\vspace{2mm}

\noindent II. \emph{Induction Schema}: $[\forall \; x \; ((\forall \; y\in x \; \varphi(y)) \Rightarrow \varphi(x))] \Rightarrow [\forall \; x \; \varphi(x)]$

\vspace{2mm}

\noindent III. \emph{Scedrov's Modal Foundation}: $[\Box \; \forall \; x \; (\Box (\forall \; y\in x \; \varphi(y)) \Rightarrow \varphi(x))]\Rightarrow [\Box \; \forall \; x \; \varphi(x)]$

\vspace{2mm}

\noindent IV. \emph{Pairing}: $\forall \; x,y\; \exists \; z \; \Box \;(\forall \; u\; (u\in z \Leftrightarrow (u=x \vee u=y)))$

\vspace{2mm}

\noindent V. \emph{Unions}: $\forall \; x \; \exists \; y \; \Box \; (\forall \; z \; (z\in y \Leftrightarrow \exists \; v \; (v\in x \wedge z\in v)))$

\vspace{2mm}

\noindent VI. \emph{Comprehension}: $\forall \; x \; \exists \; y \; \Box (\forall \; z \; (z\in y \Leftrightarrow (z\in x \; \wedge \; \varphi(z)))$

\vspace{2mm}

\noindent VII. \emph{Modal Power Set}: $\forall \; x \; \exists \; y \; \Box \; \forall \; z \; (z\in y \Leftrightarrow \Box \; (\forall \; u \; (u\in z\Rightarrow u\in x)))$

\vspace{2mm}

\noindent VIII. \emph{Infinity}: $\exists \; u \; \Box \; [\exists \; y \in u \; \Box \; (\forall \; z \; z\notin y)$

\hspace{30mm} $\wedge (\forall \; y\in u \; \exists \; z \in u \; \Box \; \forall \; v \; (v\in z \Leftrightarrow (v\in y \vee v=y)))]$

\vspace{2mm}

\noindent IX. \emph{Modal Collection}: 

\vspace{2mm}

\hspace{5mm} $[\Box \; (\forall \; y\in u \; \exists \; z \; \varphi(y,z))]\Rightarrow [\exists \; x \; \Box \; \forall \; y\in u \; \exists \; z \; (\Box (z\in x \wedge \varphi(y,z)))] $

\vspace{2mm}

\noindent X. \emph{Collection}: $[\forall \; y\in u \; \exists \; z \; \varphi(y,z)]\Rightarrow [\exists \; x \; \forall \; y\in u \; \exists \; z\in x \; \varphi(y,x)]$ 

\vspace{2mm}

\noindent In Goodman's earlier paper, ``A Genuinely Intensional Set Theory''~\cite[see][p. 63]{Goodman1985aa}, the Modal Extensionality Axiom (Axiom~I) is explained by noting that the objects in the theory's intended domain of discourse are not sets in the conventional sense, but should be understood rather as ``properties.'' Accordingly, one identifies $x$ and $y$ not when they have the same elements, but when they \emph{knowably} have the same elements. In the schemas II, III, VI, IX, X, the formulas can be any modal formulas in the signature of the membership relation, which may contain free parameter variables. Further, in this enumeration of the axioms, we've omitted Goodman's version of the axiom of choice, since its statement is complicated in the absence of full extensionality and since Scedrov's axiomatization contains no similar such axiom.

The only atomic in $\mathsf{EZF}$ besides identity is the membership relation, and one can check that Goodman's construction does not in general satisfy the stability of the membership relation. However, in this system there is the following entailment:
\begin{prop}
$\mathsf{EZF}$ plus the stability of the membership relation implies the stability of the negation of the membership relation.
\end{prop}
\begin{proof}
Consider the following instance of the Comprehension Axiom~VI of $\mathsf{EZF}$:
\begin{equation}\label{eqn:previousequation3234124399}
\forall \; p \; \forall \; x \; \exists \; y \; \Box \; (\forall \; z \; (z\in y \Leftrightarrow (z\in x \; \wedge \; z\notin p))
\end{equation}
We may argue that this implies the following:
\begin{equation}\label{eqn:previousequation323412433399}
\forall \; p,x,z \; (( z\in x \; \wedge \; z\notin p ) \Rightarrow \Box (z\notin p))
\end{equation}
For, let $p,x,z$ satisfy the antecedent. Let $y$ be the witness from~(\ref{eqn:previousequation3234124399}) with respect to this $x$ and $p$. From $z\in x \; \wedge \; z\notin p$ we can infer to $z\in y$, and from this and the stability of the membership relation we can infer to $\Box (z\in y)$. But then from~(\ref{eqn:previousequation3234124399}) we have $\Box (z\notin p)$. So indeed~(\ref{eqn:previousequation323412433399}) holds. But since $\mathsf{EZF}$ proves $\forall \; z \; \exists \; x \; z\in x$ (for instance by the Pairing Axiom~IV), we then have that $\mathsf{EZF}$ proves
\begin{equation}\label{eqn:previousequation323412433333399}
\forall \; p,z \; ((z\notin p ) \Rightarrow \Box (z\notin p))
\end{equation}
Then by necessitation, we obtain the stability of the negation of the membership relation. 
\end{proof}

As we will see in \S\ref{sec:setheory}, Flagg's construction applied to the set-theoretic case satisfies the stability of atomics but \emph{not} the stability of negated atomics (cf. Proposition~\ref{prop:counterstabilityatomics}). Thus the previous proposition implies that the type of set theory modeled by this construction will be slightly different from $\mathsf{EZF}$. Hence, we introduce the following modification:
\begin{definition}\label{defn:eZF}
The theory $\mathsf{eZF}$ consists of $Q_{eq}.\mathsf{S4}$ plus (i)~the following axioms of $\mathsf{EZF}$: Axiom~I (Modal Extensionality), Axiom~III (Scedrov's Modal Foundation), Axiom~IV (Pairing), Axiom~V (Unions), Axiom~VII (Modal Power set), Axiom~VIII (Infinity), and (ii)~the following axioms:

\vspace{2mm}

\noindent VI$^{\Box}$. \emph{Comprehension$^{\Box}$}: $\forall \; x \; \exists \; y \; \Box (\forall \; z \; (z\in y \Leftrightarrow (z\in x \; \wedge \; \Box \; \varphi(z))))$

\vspace{2mm}

\noindent X$^{\Box}$. \emph{Collection}$^{\Box}$: $[\forall \; y\in u \; \exists \; z \; \Box \; \varphi(y,z)]\Rightarrow [\exists \; x \; \forall \; y\in u \; \exists \; z\in x \; \Box \; \varphi(y,x)]$ 

\vspace{2mm}

\end{definition}

\noindent Hence, outside of the comprehension schema, the primary difference between the two systems is that $\mathsf{eZF}$ lacks Axiom~II (Induction Schema) and Axioms~IX-X (Modal Collection and Collection) of the system $\mathsf{EZF}$. It compensates by including versions of comprehension and collection for formulas which begin with a box. Our result is that when we apply Flagg's construction to a set-theoretic setting, we obtain the following:
\begin{thm}\label{thm:main}
The theory consisting of $\mathsf{eZF}$, $\mathsf{ECT}$, and the stability of atomics is consistent with the failure of the stability of negated atomics.
\end{thm}
\noindent This theorem is proven in \S\ref{sec:setheory} below, and in particular is an immediate consequence of Proposition~\ref{prop:eZF}, Proposition~\ref{prop:counterstabilityatomics}, and Proposition~\ref{eqn:finallyECTsettheory}. In Theorem~\ref{thm:main}, like in all set-theoretic formalizations of Epistemic Church's Thesis, one understands $\mathsf{ECT}$ to be rendered by binding all the variables to the set-theoretic surrogate for the natural numbers guaranteed by the Infinity Axiom (Axiom VIII).

The bulk of this paper is devoted to developing the details of our realizability semantics for quantified modal logic. In \S\ref{sec:heyting} we describe a general class of prealgebras, called Heyting prealgebras, which serve as the basis for the truth-values of the many-valued structures which we shall build in later sections. Further, in this section we emphasize a generalized double-negation operator which in our view is critical to the overall construction. In \S\ref{sec:heytingfrompca} we describe how to build Heyting prealgebras out of partial combinatory algebras. We there describe three important examples of partial combinatory algebras: Kleene's first model (the model used in ordinary computation on the natural numbers), Kleene's second model (related to oracle computations), and Scott's graph model (used to build models of the untyped lambda calculus).

Then in \S\ref{sec:heytingtobooleanmodal} we describe how to transform Heyting prealgebras based on partial combinatory algebras into Boolean prealgebras with a modal operator, where the relevant distinction is that Boolean prealgebras validate the law of the excluded middle whereas Heyting prealgebras do not in general. Prior to proceeding with the construction proper, in \S\ref{secS5isfalse} we focus on delimiting the scope of the modal propositional logic validated in these modal Boolean prealgebras. Then, in \S\ref{sec:intuitiontomodal}, we show how to associate a modal Boolean-valued structure to certain kinds of Heyting-valued structures. In \S\ref{sec:GTplusCB} we describe two general results on this semantics, namely the G\"odel translation and Flagg's Change of Basis Theorem. 

In \S\ref{sec:arithmetic} we apply this construction to the arithmetic case and obtain Theorem~\ref{thm:Flaggthm} and in \S\ref{sec:setheory} we apply the construction to the set-theoretic case and establish Theorem~\ref{thm:main}. These two constructions use the partial combinatory algebra associated to Kleene's first model, namely the model used in ordinary computation on the natural numbers. To illustrate the generality of the semantics constructed here, in \S\ref{sec:Troelstrakleene} we build a model of a fragment of second-order arithmetic which employs Kleene's second model and which validates a modal version of Troelstra's generalized continuity scheme. This modal version then stands to second-order arithmetic with low levels of comprehension roughly as $\mathsf{ECT}$ stands to first-order arithmetic (cf. Theorem~\ref{thm:genalizedcont}). Finally, in \S\ref{sec:untypedlambda}, we employ Scott's graph model and use it to build a simple example wherein the stability of the negation of identity fails (cf. Proposition~\ref{prop:failsurenegstable}).

Before proceeding, it ought to be explicitly mentioned that ours is not the first attempt to revisit Flagg's important construction. In particular, Flagg's advisor Goodman did so in 1986, remarking that Flagg's original presentation was ``not very perspicuous'' (cf. \citet[ p. 387]{Goodman1986aa}). In addition to generalizing from the arithmetic to the set-theoretic realm, one difference between our approach and that of Goodman is that he worked only over a single partial combinatory algebra, namely that of the ordinary model of computation on the natural numbers. Further, Goodman's proof proceeds via a series of syntactic translations between intuitionsitic arithmetic, epistemic arithmetic, and a modal intuitionistic system. By contrast, our methods are entirely semantic in nature, and carry with it all the benefits and limitations of a semantic approach. For instance, while we get lots of information about independence from the axioms, we get little information about the strength of the axioms. See the concluding section \S\ref{sec:conclusions} for some more specific questions about deductive features of the axiom systems modeled in this paper.


\section{Heyting Prealgebras and the Generalized Double-Negation Operator}\label{sec:heyting}

In this section, we focus on Heyting prealgebras and a certain operator $\ominus$ which is formally analogous to  double negation. A \emph{Heyting prealgebra} $\mathbb{H}$ is given by a reflexive, transitive order $\leq$ on $\mathbb{H}$ such that there is an infimum $x\wedge y$ and supremum $x\vee y$ of any two element subset $\{x,y\}\subseteq \mathbb{H}$, and a bottom element $\bot$, a top element $\top$, and such that there is a binary map $\Rightarrow$ satisfying $x\wedge y\leq z \mbox{ iff } y\leq x\Rightarrow z$ for all $x,y,z\in \mathbb{H}$. Hence, overall, a Heyting prealgebra is given by the signature $(\mathbb{H}, \leq, \wedge, \vee, \bot, \top, \Rightarrow)$. Hereafter, we will use the notation $\neg x$ to denote $x \Rightarrow \bot$, and $x\equiv y$ as an abbreviation for $x\leq y$ and $y\leq x$. If $t(\overline{x}), s(\overline{x})$ are terms in the signature of Heyting prealgebras, then atomic formulas of the form $t(\overline{x})\leq s(\overline{x})$ will be called \emph{reductions} or sometimes \emph{inequalities}, while atomics of the form $t(\overline{x})\equiv s(\overline{x})$ will be called \emph{equivalences}. Further, given the close connection between reductions and the conditional, in the context $t(\overline{x})\leq s(\overline{x})$, the term $t(\overline{x})$ will be called the \emph{antecedent} and the term $s(\overline{x})$ will be called the \emph{consequent}. As one can easily verify, the following are true on any Heyting prealgebra (cf. Proposition~\ref{prop:moreelementary} in Appendix~\ref{sec:appendixheyting}):
\begin{align}
& x \wedge (x\Rightarrow z)\leq z \label{help1}  \\
& x\leq y \mbox{ implies } y\Rightarrow z \; \leq \; x\Rightarrow z\label{help2} \\
& (x\Rightarrow y) \wedge (y \Rightarrow z)\; \leq\; x\Rightarrow z \label{help3} 
\end{align}

Suppose that $\mathbb{H}$ is a Heyting prealgebra and ${d}$ is in $\mathbb{H}$. Then we define the map $\ominus_{d}:\mathbb{H}\rightarrow \mathbb{H}$ by $\ominus_{d}(x) = (x\Rightarrow {d})\Rightarrow {d}$. Obviously if $d=\bot$, then $\ominus_d(x)=\neg\neg x$, and so we can think of $\ominus_d$ as a kind of generalization of the double-negation operator. This map $\ominus_{d}:\mathbb{H}\rightarrow \mathbb{H}$ then has the following properties (cf. Proposition~\ref{prop:thediamondprop22} in Appendix~\ref{sec:appendixheyting}):
\vspace{-8mm}
\begin{multicols}{2}
\begin{align}
& \hspace{-5mm} x\leq y \mbox{ implies } \ominus_d(x)\leq \ominus_{d}(y) \label{prop:diamonds1}\\
& \hspace{-5mm}x\leq \ominus_{d}(x)\label{prop:diamonds2} \\
& \hspace{-5mm}\ominus_{d}(\ominus_{d}(x))\leq \ominus_d (x) \label{prop:diamonds3}\\
& \hspace{-5mm}\ominus_{d}(\ominus_{d}(x)) \equiv \ominus_{d} (x) \label{prop:diamonds5} \\
& \hspace{-5mm}\ominus_{d}(x\wedge y) \equiv \ominus_{d}(x)\wedge \ominus_{d}(y)\label{prop:diamonds4} \\
& \hspace{-5mm} \ominus_{d}(x\Rightarrow y) \; \leq \;  \ominus_{d}(x)\Rightarrow \ominus_{d}(y)\label{prop:diamonds6} 
\end{align}
\columnbreak
\begin{align}
& \mbox{\;} \notag \\
& \ominus_{d}(x\vee y) \geq \ominus_{d}(x)\vee \ominus_{d}(y)\label{prop:diamonds9}\\
& \ominus_d(d)\leq d \label{prop:diamonds12} \\
& d\leq \ominus_d (x)\label{prop:diamonds8} \\
& \ominus_d (\top) \equiv \top \label{prop:diamonds7} \\
& \ominus_d (x) \Rightarrow d \; \equiv \; x\Rightarrow d \label{prop:diamonds10}
\end{align}
\end{multicols}
\vspace{-5mm}
\noindent Notationally, we may write $\ominus_d x$ for $\ominus_d (x)$. Further, we stipulate that $\ominus_d$ binds tightly, so that, e.g., writing $\ominus_d x \vee y$ will refer not to $\ominus_d (x \vee y)$ but to $(\ominus_d x) \vee y$. The map $\ominus_d$ is denoted by $\diamond_d$ in \citet[Example 8.3 p. 133]{Flagg1985aa}, but we avoid this terminology due to the received use of the diamond symbol for the modal operator.

There is a strong connection between Heyting prealgebras and intuitionistic logic. Following the tradition in many-valued logics, let's say that a formula $\varphi$ of propositional logic is \emph{valid} in a Heyting prealgebra if the homophonic translation of this formula into an element of the prealgebra is equivalent to the top element of the prealgebra. Further, let $\mathsf{IPC}$ designate the intuitionistic propositional calculus \cite[Definition 2.1.9, volume 1 p. 48]{Troelstra1988aa}. One can then show that if  $\mathsf{IPC}\vdash \varphi$, then $\varphi$ is valid in any Heyting prealgebra \cite[Theorem 13.5.3 volume 2 p. 705]{Troelstra1988aa}. Intuitionistic logic is also a natural logic to reason about Heyting prealgebras, and in Appendix~\ref{sec:appendixheyting} we verify that~(\ref{prop:diamonds1})-(\ref{prop:diamonds10}) are provable in a weak intuitionsitic logic, a fact which will prove useful in the subsequent sections (cf. in particular Proposition~\ref{prop:formaljustificationunfirm}, as well as the discussion of uniformity in \S\ref{sec:heytingfrompca}).

There is similarly a natural connection between Boolean prealgebras and classical logic. If a Heyting prealgebra satisfies $x\vee \neg x\equiv \top$ (equivalently, $x \equiv \neg \neg x$) for all elements $x$, then it is called a \emph{Boolean prealgebra}. Since the classical propositional calculus $\mathsf{CPC}$ simply extends the intuitionistic propositional calculus $\mathsf{IPC}$ by the law of the excluded middle, $\mathsf{CPC}\vdash \varphi$ implies that $\varphi$ is valid in any Boolean prealgebra. The purpose of the operator $\ominus_d$ is to map a Heyting prealgebra $\mathbb{H}$ with element~$d$ to a Boolean prealgebra $\mathbb{H}_d$ wherein $d$ is equivalent to the bottom element. This is the substance of the following proposition with which we close the section. This proposition is stated and a proof of it is sketched in \citet[Theorem 8.4 p. 134]{Flagg1985aa}.
\begin{prop}\label{prop:diamondsB}
Suppose that $\mathbb{H}$ is a Heyting prealgebra and ${d}\in \mathbb{H}$. Define the set 
\begin{equation}\label{defn:H_d}
\mathbb{H}_{d} = \{x\in \mathbb{H}: \ominus_{d} x \equiv x\}
\end{equation}
Then $\mathbb{H}_{d}$ has the structure of a Heyting prealgebra, where the operations are defined by restriction from $\mathbb{H}$, except in the case of join and falsum, which we define as follows:
\begin{equation}\label{structure0}
f \vee^{d} g  =  \ominus_d (f \vee g), \hspace{10mm} \bot^{d} = \ominus_d(\bot)
\end{equation}
Further, $\mathbb{H}_{d}$ is a Boolean prealgebra and $\bot^d \equiv d$ and $\neg_d \neg_d x \equiv \ominus_d x$.
\end{prop}
\begin{proof}
For the proof, let us explicitly name the components of the structure on $\mathbb{H}_d$:
\begin{equation}
(\mathbb{H}_d, \leq^{d}, \wedge^{d}, \vee^{d}, \bot^{d}, \top^{d}, \Rightarrow^{d})
\end{equation}
where these are defined by restriction of $\mathbb{H}$ to $\mathbb{H}_d$, except in the case of join and falsum which are defined as in equation~(\ref{structure0}). Since $\leq^{d}$ is the restriction of the preorder $\leq$ on $\mathbb{H}$ to $\mathbb{H}_d$, of course $\leq^d$ is a preorder on $\mathbb{H}_d$. Since they are defined by restriction, technically we have $\wedge^d, \vee^d, \Rightarrow^d :\mathbb{H}_d \times \mathbb{H}_d \rightarrow \mathbb{H}$. Let us argue that the codomain of these operations is in fact $\mathbb{H}_d$. So suppose that $x,y\in \mathbb{H}_d$. First note that $x\wedge^d y \in \mathbb{H}_d$. For, it follows from equation~(\ref{prop:diamonds4}) that $\ominus_d (x\wedge y) \equiv  \ominus_d x \wedge \ominus_d y \equiv x\wedge y$. Second note that 
\begin{equation}\label{eqn:seqasdfdas}
\ominus_{d}(x\Rightarrow y) \ \equiv \ \ominus_{d}x\Rightarrow \ominus_{d}y
\end{equation}
For, one half of this follows from equation~(\ref{prop:diamonds6}). For the other half, note that since $x,y\in \mathbb{H}_d$, we can appeal to equation~(\ref{prop:diamonds2}) to obtain $\ominus_d x \Rightarrow \ominus_d y \equiv  x\Rightarrow y \leq \ominus_d(x\Rightarrow y)$. From equation~(\ref{eqn:seqasdfdas}), it follows that $x\Rightarrow^d y$ is in $\mathbb{H}_d$, since $x\Rightarrow y \ \equiv \ \ominus_d x \Rightarrow \ominus_d y \ \equiv \ \ominus_d(x\Rightarrow y)$. Third, note that $x\vee^d y \in \mathbb{H}_d$. For equation~(\ref{prop:diamonds5}) yields the second equivalence in the following: $\ominus_d (x\vee^d y) \equiv \ominus_d (\ominus_d (x \vee y) ) \equiv \ominus_d (x \vee y) \equiv x\vee^d y$.

Since $\wedge^d$ is defined by restriction, it follows immediately that $x\wedge^d y$ is the infimum of $\{x,y\}$ in $\mathbb{H}_d$. To see that $x\vee^d y$ is the supremum of $\{x,y\}$ in $\mathbb{H}_d$, note that $x \equiv  \ominus_d x \leq \ominus_d (x\vee y) \equiv x\vee^d y$ and $y \equiv  \ominus_d y \leq \ominus_d (x\vee y) \equiv x\vee^d y$, where the inequality comes from equation~(\ref{prop:diamonds1}) applied to $x\leq x\vee y$ and $y\leq x\vee y$. Suppose now that $z\in \mathbb{H}_d$ with $x,y\leq z$. Then $x\vee y \leq z$. Then by applying equation~(\ref{prop:diamonds1}), we obtain $x\vee^d y \equiv \ominus_d (x\vee y) \leq \ominus_d z \equiv z$.  Hence, this is why $x\vee^d y$ is the supremum of $\{x,y\}$ in $\mathbb{H}_d$. Since we have that $\wedge^d$ and $\Rightarrow^d$ are the restrictions of $\wedge$ and $\Rightarrow$ to $\mathbb{H}_d$, it follows that we automatically have the axiom for the $\Rightarrow$-map.  Finally, since $\leq^d$ is defined by restriction and $\top$ is in $\mathbb{H}_d$ by equation~(\ref{prop:diamonds7}), we have that $\top$ is the top element of $\mathbb{H}_d$. As for the bottom element of $\mathbb{H}_d$, suppose that $x$ is an arbitrary element of $\mathbb{H}_d$. Then $\bot \leq x$, and so $\ominus_d \bot \leq \ominus_d x \equiv x$ by equation~(\ref{prop:diamonds1}). Thus it follows that $\mathbb{H}_d$ is a Heyting prealgebra. 

To see that $\bot^d = \ominus_d(\bot)\equiv d$, first note that $\bot \leq d$ and so $\ominus_d \bot \leq \ominus_d d\leq d$, where the first inequality comes from~(\ref{prop:diamonds1}) and the second from (\ref{prop:diamonds12}). On the other hand, one has that $d\leq \ominus_d(\bot)$ by equation~(\ref{prop:diamonds8}). So in fact we have $\bot^d \equiv \ominus_d(\bot)\equiv d$. 

Now we argue that $\mathbb{H}_d$ is a Boolean prealgebra. In our Heyting prealgebra $\mathbb{H}_d$, negation $\neg_d$ is defined by $\neg_d x \equiv (x \Rightarrow \bot^d) \equiv (x \Rightarrow \ominus_d \bot) \equiv (x\Rightarrow d )$. Note that it follows from this that $\neg_d \neg_d x \equiv \ominus_d x$. To establish that $\mathbb{H}_d$ is a Boolean prealgebra, we must show that  if $x\in \mathbb{H}_d$ then $\neg_d \neg_d x \equiv x$, or what is the same we must show $\ominus_d x \equiv x$. But this is precisely the condition to be an element of $\mathbb{H}_d$ in the first place.
\end{proof}


\section{Heyting Prealgebras from Partial Combinatory Algebras}\label{sec:heytingfrompca}

  The goal of this section is to recall the definition of a partial combinatory algebra and to describe some examples which will be relevant for our subsequent work. Our treatment relies heavily on \citet[Chapter 1]{Oosten2008aa}. The only way in which this differs from his treatment is that we need to pay attention to certain uniformity properties delivered by the standard construction of a Heyting prealgebra from a partial combinatory algebra. So without further ado, here is the definition of a partial combinatory algebra (cf. \citet[p. 3]{Oosten2008aa}, \citet[p. 100]{Beeson1985aa}):

\begin{definition}\label{defn:pca}
A \emph{partial combinatory algebra}, or \emph{pca}, is a set $\mathcal{A}$  with a partial binary function $(a,b)\mapsto ab$ and distinguished elements $k$ and $s$ such that (i) $sab\hspace{-1mm}\downarrow$ for all $a,b$ from $\mathcal{A}$, and (ii) $kab=a$ for all $a,b$ from $\mathcal{A}$, and (iii) $sabc = ac(bc)$ for all $a,b,c$ from $\mathcal{A}$.
\end{definition}

The convention in pca's is that one associates to the left, so that $abc=(ab)c$. Further, ``downarrow'' notation in condition~(i) of pca's is borrowed from computability theory, so that in general in a pca, $ab\hspace{-1mm}\downarrow$ indicates that the partial binary function is defined on input $(a,b)$. In any pca $\mathcal{A}$, there are some elements beyond $k$ and $s$ to which it is useful to call attention. First, the element $skk$ serves as an identity element, in that one has $skka=a$ for all $a$ from $\mathcal{A}$. Second, there is the pairing function $p$ and the two projection functions $p_0, p_1$ such that $p_0(pab)=a$ and $p_1(pab)=b$. Third, there is an element $\breve{k}$ such that $\breve{k}ab=b$. Sometimes we use $k$ and $\breve{k}$ to code case splits, much like we might use $0$ and $1$ in other contexts. It's worth noting that \citet[p. 5]{Oosten2008aa} uses the notation $\overline{k}$ instead of $\breve{k}$, but we opt for the latter because the former potentially clashes with notation for the numerals introduced in equation (\ref{eqn:mccartynumbers}). Finally, one has the following very helpful proposition, which is sometimes used as an alternative characterization of a pca (cf. \citet[p. 3]{Oosten2008aa}, \citet{Beeson1985aa} p. 101):
\begin{prop}\label{prop:onpcas}
Let $\mathcal{A}$ be a pca and let $t(x_1, \ldots, x_n)$ be a term. Then there is an element $f$ of $\mathcal{A}$ such that for all $a_1, \ldots, a_n$ from $\mathcal{A}$ one has $fa_1\cdots a_n\hspace{-1mm}\downarrow=t(a_1, \ldots, a_n)$.
\end{prop}
Sometimes in what follows we'll have occasion to treat a pca as a first-order structure, and there arises then the question of the appropriate signature. It's helpful then to distinguish between a few different options. The \emph{sparse} signature takes merely $s,k$ as constants, while the \emph{expansive} signature invokes the above proposition to recursively introduce a new constant $f_t$ for any term $t$. Further, sometimes we'll work with pca's in which the binary operation is total. The \emph{relational} signature then takes the binary operation to be rendered as a ternary relation, while the \emph{functional} signature takes the binary operation as a binary function symbol.

The paradigmatic example of a pca is \emph{Kleene's first model} $\mathcal{K}_1$, which has as its underlying set the natural numbers and its operator $(e,n)\mapsto \{e\}(n)$, where this denotes the output, if any, of the $e$-th partial recursive function run on input~$n$.  In Kleene's first model $\mathcal{K}_1$, the element $k$ from condition~(ii) is given by a program~$k$ which on input~$a$ calls up a program for the constant function~$a$. For the element~$s$ from conditions~(i) and (iii), choose a total recursive function $s^{\prime\prime}$ such that $s^{\prime\prime}(a,b,c)$ is an index for the program which first runs input $c$ on index $a$ to obtain~$e$ and second runs input $c$ on index~$b$ to obtain index~$n$, and then runs input $n$ on index $e$. Then choose total recursive function~$s^{\prime}$ such that $s^{\prime}(a,b)$ is an index for a function which on input~$c$ returns $s^{\prime\prime}(a,b,c)$; and finally choose $s$ to be a total recursive function such that $s(a)$ is an index for a function which on input~$b$ returns~$s^{\prime}(a,b)$. The functions $k,s^{\prime\prime}, s^{\prime}, s$ all come from the $s\emph{-}m\emph{-}n$ Theorem and so may be chosen to be total, thus ensuring that condition~(i) holds.

Two further pcas which will be of use in \S\S\ref{sec:Troelstrakleene}-\ref{sec:untypedlambda} are Kleene's second model $\mathcal{K}_2$ and Scott's graph model, and before describing these pcas it's perhaps worth stressing that the reader who is uninterested in these specific sections can safely skip our treatment of these other pcas. Kleene's second model $\mathcal{K}_2$ (cf. \citet[pp. 15 ff]{Oosten2008aa} \S{1.4.3}, \citet[\S{VI.7.4} p. 132]{Beeson1985aa}) has Baire space $\omega^{\omega}=\{\alpha: \omega\rightarrow \omega\}$ as its underlying set. Recall that the topology on Baire space has the basis of clopens $[\sigma] =\{\alpha: \forall \; i<\left|\sigma\right| \; \alpha(i)=\sigma(i)\}$ wherein $\sigma$ is an element of $\omega^{<\omega}$, namely, the set of all strings $\sigma:[0,\ell)\rightarrow \omega$ where $\ell\geq 0$ and we define \emph{length of $\sigma$} as $\left|\sigma\right|=\ell$. Strings of small length are sometimes explicitly written out as, e.g., $(0,5,3)$ or $(57, 3,25, 100)$, and a degenerate case is length~$1$ strings written out as, e.g., $(0)$ or $(57)$. Each $\sigma$ from $\omega^{<\omega}$ is assumed to be identified with a natural number so that all the basic operations on strings such as length $\sigma\mapsto \left|\sigma\right|$ and concatenation $(\sigma,\tau)\mapsto \sigma^{\frown}\tau$ are computable. Finally, each $\alpha$ from Baire space and each length $\ell\geq 0$ naturally determines an element $\alpha\upharpoonright \ell$ of $\omega^{<\omega}$ of length $\ell$ by restriction. Likewise, the concatenation $(\sigma,\alpha)\mapsto \sigma^{\frown}\alpha$ is naturally defined to be the element of Baire space which begins with $\sigma$ along its length and then followed by $\alpha$. Baire space is of course a basic object of descriptive set theory, but the only part of that theory that we need to recall is that the $G_{\delta}$ subsets are the countable intersections of open sets, and that the subsets of Baire space whose relative topology is completely metrizable are precisely the $G_{\delta}$ sets (\citet{Kechris1995aa} p. 17). This hopefully motivates the attention paid to $G_{\delta}$ sets in the context of Kleene's second model, since it is these on which the relative topology is most similar to that of the usual topology on Baire space.

In Kleene's second model $\mathcal{K}_2$, the application function is defined in terms of oracle computations. The relevant notation here is: $\{e\}^{\gamma}(n)$ denotes the $e$-th Turing program run on input $n$ and with the graph of the function~$\gamma$ on the oracle tape, and $\alpha\oplus \beta$ is the effective union of $\alpha$ and $\beta$ which follows $\alpha$ on the evens and $\beta$ on the odds. Then the application function in $\mathcal{K}_2$ is defined as follows: 
\begin{equation}\label{eqn:appinkleenesecond}
(\alpha \beta)(n) = \{e_0\}^{\alpha\oplus \beta}(n)
\end{equation}
\noindent where program~$e_0$ on input~$n$, searches for the least $\ell$ such that $\alpha(((n)^{\frown} \beta)\upharpoonright \ell)$ has non-zero value $m+1$, and return $m$ if such is found. It is very difficult to work directly with this characterization of the application function, and so practically one proceeds by using the following proposition, whose proof we relegate to Appendix~\ref{appendix:specificpcas} since while we use it frequently in \S\ref{sec:Troelstrakleene}, we only use it in that section:
\begin{prop}\label{prop:oostenfacts}
(I)~For every continuous function $G:D\rightarrow \omega^{\omega}$ with $G_{\delta}$ domain $D\subseteq (\omega^{\omega})^n$ there is $\gamma$ such that $G(\alpha_1,\ldots \alpha_n)=\gamma\alpha_1 \cdots \alpha_n$ for all $(\alpha_1, \ldots, \alpha_n)$ in~$D$. (II)~For every $\gamma$ and each $n\in \{1,2\}$, the partial map $(\alpha_1, \ldots, \alpha_n)\mapsto \gamma \alpha_1 \cdots \alpha_n$ has $G_{\delta}$ domain and is continuous on this domain.
\end{prop}
Applying the Proposition~\ref{prop:oostenfacts}.I to the first projection allows one to satisfy condition~(ii) of a pca, and with a little further ingenuity one can likewise obtain conditions~(i) and (iii) of a pca (cf. \cite[pp. 17-18]{Oosten2008aa}). Further, Proposition~\ref{prop:oostenfacts}.I implies that for every index~$e\geq 0$ there is $\alpha_e$ such that $(\alpha_e \beta)(n) = \{e\}^{\beta}(n)$. Hence, Kleene's second model $\mathcal{K}_2$ can be understood as an ``oracle computation'' version of Kleene's first model.

The final pca which is used in this paper is Scott's graph model $\mathcal{S}$ (cf. \citet[\S{1.4.6} pp. 20-21]{Oosten2008aa}, \citet[pp. 157-165]{Scott1975ab}). The pca $\mathcal{S}$ has domain $P(\omega)$, namely all subsets of natural numbers. Often the pca is itself denoted as $P(\omega)$, but we will use $\mathcal{S}$ instead since the semantics we develop involves us using a good deal of powerset notation. In the context of Scott's graph model $\mathcal{S}$, the lower case Roman letters $a,b,c,x,y,z$ and subscripted versions thereof are reserved for elements of $\mathcal{S}$, while we reserve $n,m,\ell$ and subscripted versions thereof for natural numbers. Further, the letters $d_1, d_2, \ldots$ are reserved for a standard and fixed enumeration of finite subsets of the natural numbers.  For any $x\in \mathcal{S}$, we define $[x]=\{y\in \mathcal{S}: x\subseteq y\}$. Then the topology on $\mathcal{S}$ has basis $[d_n]$, and the topology on $\mathcal{S}^k$ is the product topology. When the exponent $k\geq 1$ is understood from context, we abbreviate e.g. $\overline{x}=(x_1, \ldots, x_k)$ and $\overline{y}=(y_1, \ldots, y_k)$ and $\overline{n} = (n_1, \ldots, n_k)$ and $d_{\overline{n}} = (d_{n_1}, \ldots, d_{n_k})$ and $[d_{\overline{n}}] = [d_{n_1}]\times \cdots\times [d_{n_k}]$. Finally, we say that  $\overline{x}\subseteq \overline{y}$ if $x_1\subseteq y_1, \ldots, x_k\subseteq y_k$, and we say that  a function $f:\mathcal{S}^k\rightarrow \mathcal{S}$ is \emph{monotonic} if $\overline{x}\subseteq \overline{y}$ implies $f(\overline{x})\subseteq f(\overline{y})$.

Scott then proved the following helpful characterizations of the open and closed sets and the continuous functions \cite[pp. 158-159]{Scott1975ab}, from which it immediately follows that all continuous functions are monotonic:
\begin{prop}\label{prop:openclosed} Suppose that $\mathcal{U}, \mathcal{C}\subseteq \mathcal{S}^k$. Then (i) $\mathcal{U}$ is open iff for all $\overline{x} \in \mathcal{S}^k$, one has $\overline{x}\in \mathcal{U}$ iff $\exists \; \overline{n} \; (d_{\overline{n}}\subseteq \overline{x} \; \wedge \; d_{\overline{n}} \in \mathcal{U})$. Further, (ii) $\mathcal{C}$ is closed iff for all $\overline{x} \in \mathcal{S}^k$, one has $\overline{x}\in \mathcal{C}$ iff $\forall \; \overline{n} \; (d_{\overline{n}}\subseteq \overline{x} \mbox{ implies } d_{\overline{n}} \in \mathcal{C})$.
\end{prop}

\begin{prop}\label{eqn:helperprop}
Suppose that $f:\mathcal{S}^k\rightarrow \mathcal{S}$. Then the following conditions are equivalent:
\begin{enumerate}
\item[] (i) For all $\overline{x}\in \mathcal{S}^k$, one has $f(\overline{x})=\bigcup_{d_{\overline{n}}\subseteq \overline{x}} f(d_{\overline{n}})$.
\item[] (ii) The function $f:\mathcal{S}^k\rightarrow \mathcal{S}$ is continuous.
\item[] (iii) For all $\overline{x}\in \mathcal{S}^k$, one has $m\in f(\overline{x})$ iff $\exists \; \overline{n} \; (d_{\overline{n}}\subseteq \overline{x} \; \wedge \; m\in f(d_{\overline{n}}))$.
\item[] (iv) For all $\overline{x}\in \mathcal{S}^k$, one has $d_m\subseteq f(\overline{x})$ iff $\exists \; \overline{n} \; (d_{\overline{n}}\subseteq \overline{x} \; \; \wedge \; d_m \subseteq f(d_{\overline{n}}))$.
\end{enumerate}
\end{prop}

Using these propositions, we can define the application function and show that the axioms of a pca are satisfied. In this, we fix a computable \emph{bijection} $\langle \cdot, \cdot\rangle: \omega\times \omega\rightarrow \omega$, and we iterate in the natural way, e.g. $\langle n,m,\ell\rangle = \langle \langle n,m\rangle, \ell\rangle$. Then we define the application function as follows:
\begin{equation}\label{eqn:appREpca}
ab =  \{m: \exists \; n\geq 0 \; (\langle n,m\rangle \in a \; \wedge \; d_n\subseteq b)\}
\end{equation}
When $a$ is recursively enumerable, the thought behind this operation is that $ab=c$ means that ``we can effectively enumerate $c$ whenever we are given any enumeration of $b$'' \cite[p. 827]{Odifreddi1999aa}, with the variables changed in the quotation to match~(\ref{eqn:appREpca}), cf. also \citet[\S{9.7} pp. 145 ff]{Rogers1987aa}, \citet[p. 155, p. 160]{Scott1975ab}. From~(\ref{eqn:appREpca}), an induction on $k\geq 1$ shows that for all $\overline{b}=(b_1, \ldots, b_k)\in \mathcal{S}^k$ one has
\begin{equation}\label{eqn:appREpca2}
a b_1 \cdots b_k =  \{m: \exists \; \overline{n} \; (\langle \overline{n},m\rangle \in a \; \wedge \; d_{\overline{n}}\subseteq \overline{b})\}
\end{equation}
Note that unlike the other paradigmatic examples of pcas, the application operation~(\ref{eqn:appREpca}) here is total.


Like with Proposition~\ref{prop:oostenfacts}, it's helpful to articulate the connection between the continuous functions and the application operation. However, in this setting the connection is much tighter, due to the fact that the application operation~(\ref{eqn:appREpca}) is total. In particular, we have the following proposition due to \cite[p. 160]{Scott1975ab}.
\begin{prop}\label{prop:veryhlpeful112}
(I)~For any continuous function $f:\mathcal{S}^k \rightarrow \mathcal{S}$ there is $a=\mathrm{graph}(f)=\{\langle \overline{n}, m\rangle: m\in f(d_{\overline{n}})\}$ such that for all $\overline{b}=(b_1, \ldots, b_k)\in \mathcal{S}^k$, one has $f(\overline{b})=a b_1 \cdots b_k$. (II)~For any $a\in \mathcal{S}$, the function $\mathrm{fun}(a):\mathcal{S}^k\rightarrow \mathcal{S}$ given by $\mathrm{fun}(a)(\overline{b})=a b_1 \cdots b_k$ is continuous. (III)~For any continuous function $f:\mathcal{S} \rightarrow \mathcal{S}$, one has $\mathrm{fun}(\mathrm{graph}(f)) =f$. (IV) The function $\mathrm{graph} \circ \mathrm{fun}: \mathcal{S}\rightarrow \mathcal{S}$ is continuous and for all $a\in \mathcal{S}$ one has
\begin{equation}\label{eqn:helpfulchar}
\langle n,m\rangle \in \mathrm{graph}(\mathrm{fun}(a)) \mbox{ iff } \exists \; \ell\; (\langle \ell,m\rangle \in a \; \wedge \; d_{\ell}\subseteq d_n)
\end{equation}
\end{prop}

\noindent This proposition allows us to quickly see that $\mathcal{S}$ is a pca. For~$k$, simply apply the Proposition~\ref{prop:veryhlpeful112}.I to the continuous map $(a,b)\mapsto a$. For~$s$, the function $G(a,b)=a\cdot b$ is continuous by Proposition~\ref{prop:veryhlpeful112}.II. Then the map $H(a,b,c) = G(G(a,c), G(b,c))$ is continuous and so by Proposition~\ref{prop:veryhlpeful112}.I there is $s$ such that $sabc = (ac)(bc)$. Hence, $\mathcal{S}$ is a pca. 

However, the original interest in $\mathcal{S}$ is related to the untyped lambda calculus, and Scott and Meyer produced a series of axioms describing when a pca yields a model of the untyped lambda-calculus. These are the axioms at issue in the following proposition, where in a pca we define $1=1_1=s(k(skk))$ and $1_{n+1}=s(k1_1)(s(k1_n))$:
\begin{prop}\label{prop:scottmeyer} 
The pca $\mathcal{S}$ satisfies the axioms $\forall \; x,y \; kxy=x$ and $\forall \; x,y,z \; sxyz = (xz)(yz)$ and $1_2 k = k$ and $1_3 s = s$ and the Meyer-Scott axiom $(\forall \; x \; ax=bx)\Rightarrow 1a=1b$. 
\end{prop}
For a proof of this proposition and the result that this constitutes a characterization of models of the untyped lambda calculus, see \citet[Corollary 18.1.8 p. 473]{Barendregt1981aa} and \citet[Theorem 5.6.3 p. 117]{Barendregt1981aa}. In our context, it will be useful to have a simpler description of the action of~$1=s(k(skk))$. By Proposition~\ref{prop:veryhlpeful112}.I \& III, there is $c$ such that $cx=\mathrm{graph}(\mathrm{fun}(x))$. It can then be shown that $1x=cx$ \cite[Proposition 18.1.9.i p. 473, Lemma 5.2.8.ii p. 95]{Barendregt1981aa}, and so on $\mathcal{S}$, the consequent of the Meyer-Scott axiom can be expressed as $\mathrm{graph}(\mathrm{fun}(a))=\mathrm{graph}(\mathrm{fun}(b))$. By considering the case of $a=\emptyset$ and $b=\omega$, one can deduce from this and~(\ref{eqn:helpfulchar}) the following elementary observation, which we shall use later in \S\ref{sec:untypedlambda} to provide an example of the non-stability of negated atomics (cf. Proposition~\ref{prop:failsurenegstable}):
\begin{prop}\label{prop:counterextension}
There are $a,b\in \mathcal{S}$ such that $1a\neq 1b$ and for all $x\in \mathcal{S}$ one has $ax\neq bx$.
\end{prop}

Before turning to a discussion of how pca's give rise to Heyting prealgebras, let's adopt the following abbreviation: if $e$ is from $\mathcal{A}$ and $X,Y\subseteq \mathcal{A}$, then we say that 
\begin{equation}\label{eqn:handynotation}
e:X\leadsto Y \mbox{ iff } \forall \; a\in X \; ea\hspace{-1mm}\downarrow\in Y
\end{equation}
In other standard treatments of pca's, this is sometimes described simply by saying that $e$ \emph{realizes} $X\leq Y$ \cite[p. 6]{Oosten2008aa}. But since this notion will figure heavily in the modal semantics in \S\ref{sec:intuitiontomodal}, it is useful to have some explicit notation for it. With this in place, we can state two variations of the well-known result that pca's allow one to generate Heyting algebras. In the case of $\mathcal{A}=\mathcal{K}_1$, the first of these is denoted by $R$ in Flagg's original paper (cf. \cite[Theorem 9.1 p. 135]{Flagg1985aa}).
\begin{definition}\label{recHeytingpre}
Let $\mathcal{A}$ be a pca. Then the Heyting prealgebra structure on the powerset $P(\mathcal{A})$ of $\mathcal{A}$ is given by the following:
\begin{align}
& X\leq Y \mbox{ iff } \exists \; e\in \mathcal{A} \; e : X\leadsto Y\notag \\
& X \wedge Y = \{pab : a\in X \; \wedge \; b\in Y\} \notag \\
& X \vee Y =  \{pka \ : a\in X\}\cup \{p\breve{k}b : b\in Y\} \notag \\
& \bot = \emptyset \notag \\
& \top = \mathcal{A} \notag \\
& X\Rightarrow Y = \{e\in \mathcal{A} : e:X\leadsto Y\} \notag
\end{align}
\end{definition}
In this, recall from the discussion at the outset of this section that $k$, $\breve{k}$ code case breaks and that $p$ is the pairing function with inverses $p_0$ and $p_1$. Finally, note that for every $X\in P(\mathcal{A})$, one has that $X\equiv \bot$ or $X\equiv \top$. Indeed, if $X$ is empty, then of course $X=\bot$. So suppose that $X$ is non-empty with element $a$. By Proposition~\ref{prop:onpcas}, choose $e$ in $\mathcal{A}$ such that $ex=a$ for all $x$ from $\mathcal{A}$, so that $e: \top \leadsto X$ and hence $\top\leq X$. And of course the identity function witnesses $X\leq \top$. Hence, up to equivalence, the Heyting prealgebra $P(\mathcal{A})$ is only two-valued, and if one took its quotient one would then obtain a Boolean algebra. But the Heyting prealgebra $P(\mathcal{A})$ is still interesting because in \S\ref{sec:intuitiontomodal} we will assign formulas $\varphi(y)$ to elements $\|\varphi(y)\|$ of this Heyting prealgebra, and we will often be looking at how $\|\varphi(y)\|$ varies with $y$ uniformly, where $y$ comes from some  domain $N$ (cf. Definition~\ref{DefValuationR}). And just because for all $X\in P(\mathcal{A})$ one has that $X\equiv \bot$ or $X\equiv \top$, does not mean that for all sequences $\{X_y\in P(\mathcal{A}): y\in N\}$ it is the case that $(\forall \; y\in N \; X_y\equiv \bot)$ or $(\forall \; y\in N \; X_y\equiv \top)$.

By consulting the proof of Proposition 1.2.1 from~\cite[p. 6]{Oosten2008aa}, which verifies that $P(\mathcal{A})$ is actually a Heyting prealgebra, one can see that the Heyting prealgebra structure on $P(\mathcal{A})$ is highly uniform, in that there are $e_k$ from $\mathcal{A}$ for $1\leq k \leq 12$ such that for all $X,Y,Z\subseteq \mathcal{A}$ and all $a,a^{\prime}$ from $\mathcal{A}$ one has
\begin{eqnarray}
e_1:X\leadsto X\vee Y, & \hspace{10mm} & e_2:Y\leadsto X\vee Y  \label{hrec1}\notag\\
a:X\leadsto Z  \; \wedge \; a^{\prime}:Y\leadsto Z  & \mbox{ implies } &  e_3(paa^{\prime}):X\vee Y\leadsto Z \label{hrec1a}\notag\\
a:Z\leadsto X  \; \wedge \; a^{\prime}:Z\leadsto Y  & \mbox{ implies } &  e_4(paa^{\prime}):Z\leadsto X\wedge Y \label{hrec1b}\notag\\
e_5:X\wedge Y\leadsto X, & \hspace{10mm} & e_6:X\wedge Y\leadsto Y\label{hrec2} \notag\\
e_7:X\leadsto \top, & \hspace{5mm} & e_8:\bot \leadsto X, \hspace{10mm} e_9:X\leadsto X\label{hrec3}\notag \\ 
a:X\leadsto Y \; \wedge \; a^{\prime}:Y\leadsto Z & \mbox{ implies } & e_{10}(paa^{\prime}):X\leadsto Z\label{hrec4} \notag\\
a^{\prime}:X\wedge Y\leadsto Z & \mbox{ implies } & e_{11}a^{\prime}:Y\leadsto (X\Rightarrow Z) \label{hrec5}\notag\\
a^{\prime}:Y\leadsto (X\Rightarrow Z) & \mbox{ implies } & e_{12}a^{\prime}:X\wedge Y\leadsto Z\label{hrec6} \notag
\end{eqnarray}
Further, all of $e_1, \ldots, e_{12}$ can be taken to be terms in the ample relational signature of pcas (cf. immediately after Proposition~\ref{prop:onpcas}) and indeed the same terms will work for any pca.

From this uniformity, we also obtain the following associated Heyting prealgebra structure. In Flagg's original paper, this was denoted by $R(\mathcal{X})$ (cf. \citet[p. 122, Theorem 9.3 p. 136]{Flagg1985aa}).
\begin{definition}\label{FlaggHeyting}
Let $\mathcal{A}$ be a pca and let $\mathcal{X}$ be a non-empty set. Let $\mathbb{F}(\mathcal{X})$ or \emph{the Flagg Heyting prelagebra over $\mathcal{X}$} be the set of functions $\{f:\mathcal{X}\rightarrow P(\mathcal{A})\}$. Then the Heyting prealgebra structure on $\mathbb{F}(\mathcal{X})$ is given by the following:
\vspace{-8mm}
\begin{multicols}{2}
\begin{align}
& f \leq g \mbox{ iff } \exists \; e\in \mathcal{A} \; \forall \; x\in \mathcal{X} \; e:f(x)\leadsto g(x)  \notag \\
& (f \wedge g)(x) = f(x)\wedge g(x)   \notag \\
& (f \vee g)(x) = f(x)\vee g(x) \notag 
\end{align}
\columnbreak
\begin{align}
& & \mbox{\;} \notag \\
& \bot(x) = \bot \ \  = \ \ \emptyset \notag \\
& \top(x) = \top \ \ = \ \ \mathcal{A}\notag  \\
& (f \Rightarrow g)(x) = f(x)\Rightarrow g(x) \notag 
\end{align}
\end{multicols}
\end{definition}
\vspace{-8mm}

\noindent Usually in what follows, the pca $\mathcal{A}$ will be clear from context, and so failing to display the dependence of $\mathbb{F}(\mathcal{X})$ on $\mathcal{A}$ will not cause any confusion. As another natural piece of notation, we extend the notation $e:X\leadsto Y$ where $X,Y\subseteq \mathcal{A}$ to $e:f\leadsto g$ where $f,g$ from $\mathbb{F}(\mathcal{X})$ by stipulating that this happens iff for all $x$ from $\mathcal{X}$ one has $e:f(x)\leadsto g(x)$. With this notation, we have $f\leq g$ iff there is $e$ from $\mathcal{A}$ such that $e:f\leadsto g$.

In what follows, considerations pertaining to uniformity will play an important role in the proofs. The equivalences and reductions (\ref{prop:diamonds1})-(\ref{prop:diamonds12}) can all be done in weak intuitionistic theory, and this implies that there are uniform witnesses to these equivalences and reductions in $\mathbb{F}(\mathcal{X})$. For example, line~(\ref{prop:diamonds2}) not only gives us $f \leq \ominus_g f$ for all $f,g \in \mathbb{F}(\mathcal{X})$, but also gives us that there is a single witness $e$ from $\mathcal{A}$ such that  $e\colon f \leadsto \ominus_{g}f$ for all $f,g \in \mathbb{F}(\mathcal{X})$. For a formal justification of this, see Appendix~\ref{sec:appendixheyting} and in particular Proposition~\ref{prop:formaljustificationunfirm}.

In what follows, we'll say that an equivalence or inequality in $\mathbb{F}(\mathcal{X})$ holds \emph{uniformly} when there are such uniform witnesses. A similar locution is used when we have sequences of elements from $\mathbb{F}(\mathcal{X})$. For instance, suppose we know that $f_i \leq g$ uniformly in $i\in I$, in~that there is a single witness $e$ from $\mathcal{A}$ with $e:f_i\leadsto g$ for all $i\in I$.  Then we likewise know by~(\ref{prop:diamonds1}) that there is a single witness $e^{\prime}$ such that $e^{\prime}: \ominus_D f_i \leadsto \ominus_D g$ uniformly in $i\in I$ and $D\in P(\mathcal{A})$, and we express this more compactly by saying that $\ominus_D f_i \leq \ominus_D g$ holds \emph{uniformly} in $i\in I$ and $D\in P(\mathcal{A})$.  Finally, given a sequence of elements $\{f_i : i\in I\}$ of $\mathbb{F}(\mathcal{X})$, we define its union and intersection pointwise by $(\bigcup_{i\in I} f_i)(x)=\bigcup_{i\in I} f_i(x)$ and $(\bigcap_{i\in I} f_i)(x)=\bigcap_{i\in I} f_i(x)$. Then uniform reductions on sequences are sufficient for reductions concerning their intersection and union. In particular, supposing that $f_i\leq h$ uniformly, then one has $\bigcap_{i\in I} f_i\leq \bigcup_{i\in I} f_i\leq h$. Similarly, supposing that $h\leq f_i$ uniformly, then one has $h\leq \bigcap_{i\in I} f_i\leq \bigcup_{i\in I} f_i$.

We close this section with two elementary propositions on  $\mathbb{F}(\mathcal{X})$ that will prove useful in what follows, and that illustrate these uniformity considerations:
\begin{prop}\label{mylittlehelper} For any two sequences $\{g_i\in \mathbb{F}(\mathcal{X}): i\in I\}$ and $\{f_i\in \mathbb{F}(\mathcal{X}): i\in I\}$ one has uniformly in $D\in P(\mathcal{A})$ that $\ominus_D \bigcup_{i\in I} (g_i \wedge f_i) \equiv  \ominus_D \bigcup_{i\in I} (g_i \wedge \ominus_D f_i)$.
\end{prop}
\begin{proof}
The identity function witnesses that $(g_i \wedge f_i)\leq \bigcup_i (g_i \wedge f_i)$ uniformly in $i\in I$. Then $(g_i \wedge \ominus_D f_i) \leq \ominus_D (g_i \wedge f_i) \leq \ominus_D \bigcup_{i\in I} (g_i \wedge f_i)$ uniformly in $i \in I, D\in P(\mathcal{A})$ by~(\ref{prop:diamonds2}) in conjunction with~(\ref{prop:diamonds4}), and then by (\ref{prop:diamonds1}). Then $\bigcup_{i\in I}  (g_i \wedge \ominus_D f_i) \leq \ominus_D \bigcup_{i\in I} (g_i \wedge f_i)$ uniformly in $D\in P(\mathcal{A})$. Then uniformly in $D\in P(\mathcal{A})$ we have $\ominus_D \bigcup_{i\in I}  (g_i \wedge \ominus_D f_i) \leq \ominus_D \ominus_D \bigcup_{i\in I} (g_i \wedge f_i) \equiv \ominus_D \bigcup_{i\in I} (g_i \wedge f_i)$ by (\ref{prop:diamonds1}) and (\ref{prop:diamonds5}). Conversely, uniformly in $i\in I$ and $D\in P(\mathcal{A})$ we have $(g_i \wedge f_i)\leq (g_i \wedge \ominus_D f_i)$ by~(\ref{prop:diamonds2}). Then $\bigcup_{i\in I} (g_i \wedge f_i)\leq \bigcup_{i\in I} (g_i \wedge \ominus_D  f_i)$ uniformly in $D\in P(\mathcal{A})$. Then uniformly in $D\in P(\mathcal{A})$ we have $\ominus_D \bigcup_{i\in I} (g_i \wedge f_i)\leq \ominus_D \bigcup_{i\in I} (g_i \wedge \ominus_D f_i)$.
\end{proof}


\begin{prop}\label{prop:proponintersection}
For any sequence $\{f_i\in \mathbb{F}(\mathcal{X}): i\in I\}$, one has $\bigcap_{i\in I} f_i\equiv \bigcap_{D\in P(\mathcal{A})} \bigcap_{i\in I} \ominus_D f_i$.
\end{prop}
\begin{proof}
To see that $\bigcap_{i\in I} f_i\leq \bigcap_{D\in P(\mathcal{A})} \bigcap_{i\in I} \ominus_D f_i$, simply note that we have $f_i(x)\leq \ominus_D f_i(x)$ uniformly in $D\in P(\mathcal{A})$ and $x\in \mathcal{X}$ by line~(\ref{prop:diamonds2}). Conversely, let $e$ be a witness to~(\ref{prop:diamonds12}) in $P(\mathcal{A})$, so that $e:\ominus_Z Z\leadsto Z$ for all $Z\in P(\mathcal{A})$. Then suppose that $n$ is in the intersection $\bigcap_{D\in P(\mathcal{A})} \bigcap_{i\in I} \ominus_D f_i(x)$. To see that $en$ is in intersection $\bigcap_{i\in I} f_i(x)$, let $i\in I$ be given and consider $D=f_i(x)$. Then since $n$ is in $\ominus_D f_i(x)=\ominus_D D$, we then have that $en$ is in $D=f_i(x)$, which is what we wanted to show.
\end{proof}

\section{Construction of Boolean Algebras with Modal Operator}\label{sec:heytingtobooleanmodal}

In this section, we fix a pca $\mathcal{A}$ and present a general construction of a series of Boolean prealgebras associated to arbitrary non-empty sets $\mathcal{X}$ using Proposition~\ref{prop:diamondsB}. In what follows, we adopt the convention introduced in the last section that if $\mathcal{X}$ is an non-empty set, then $\mathbb{F}(\mathcal{X})$ denotes the set of functions $\{f:\mathcal{X}\rightarrow P(\mathcal{A})\}$, and we suppress the dependence of $\mathbb{F}(\mathcal{X})$ on the underlying pca $\mathcal{A}$. To get a sense of the notation, note that if $\mathcal{X}, \mathcal{Y}$ are two non-empty sets, then under our conventions $\mathbb{F}(\mathcal{X}\times \mathcal{Y})$ denotes the set of functions $\{f:\mathcal{X}\times \mathcal{Y} \rightarrow P(\mathcal{A})\}$. Often in what follows, we will be taking $\mathcal{Y}=P(\mathcal{A})$, and so it is important to take explicit note of the relevant conventions, since we're proceeding by suppressing the role $P(\mathcal{A})$ plays as the codomain of all the functions in $\mathbb{F}(\mathcal{X}\times P(\mathcal{A}))$, but we're explicitly displaying the role that $P(\mathcal{A})$ plays as a component of the domain of the functions in $\mathbb{F}(\mathcal{X}\times P(\mathcal{A}))$.

So let $\mathcal{X}$ be given and let $\pi:\mathcal{X}\times P(\mathcal{A})\rightarrow P(\mathcal{A})$ be the projection map $\pi(x,Z)=Z$, so that $\pi$ is a member of $\mathbb{F}(\mathcal{X}\times P(\mathcal{A}))$. Then $\ominus_{\pi}:\mathbb{F}(\mathcal{X}\times P(\mathcal{A}))\rightarrow \mathbb{F}(\mathcal{X}\times P(\mathcal{A}))$ is the map defined by $\ominus_{\pi}(f) = (f\Rightarrow \pi)\Rightarrow \pi$. Then define as in line~(\ref{defn:H_d}) the set
\begin{equation}\label{def B}
\mathbb{B}(\mathcal{X})=\{f\in \mathbb{F}(\mathcal{X}\times P(\mathcal{A})): \ominus_{\pi} f \equiv f\}
\end{equation}
Then as a consequence of Proposition~\ref{prop:diamondsB},  one has that $\mathbb{B}(\mathcal{X})$ possesses the structure of a Boolean prealgebra. In Flagg's original paper, this was denoted by $B(\mathcal{X})$ and presented only in the proof but not the statement of \citet[Theorem 10.1 p. 138]{Flagg1985aa}. See in particular the discussion of the projection map as the ``preferred element'' on the top of \citet[p. 139]{Flagg1985aa}.

Often in what follows we will focus on a special case of the above construction wherein $\mathcal{X}$ is a singleton. Since it will be the subject of special focus, and since choosing a specific singleton with which to work would be awkward, we present the following definition. Let $\id:P(\mathcal{A})\rightarrow P(\mathcal{A})$ be the identity map $\id(Z)=Z$, which is an element of $\mathbb{F}(P(\mathcal{A}))$. Then recall that $\ominus_{\id}:\mathbb{F}(P(\mathcal{A}))\rightarrow \mathbb{F}(P(\mathcal{A}))$ is the map defined by $\ominus_{\id}(f) = (f\Rightarrow \id)\Rightarrow \id$. Then define as in line~(\ref{defn:H_d}) the set
\begin{equation}\label{def Breal3}
\mathbb{B}=\{f\in \mathbb{F}(P(\mathcal{A})): \ominus_{\id} f \equiv f\}
\end{equation}
Then $\mathbb{B}$ too has the structure of a Boolean prealgebra.

Before proceeding to the construction on the modal operator, let's briefly note two elementary propositions pertaining to closure conditions in these Boolean prealgebras:

\begin{prop}\label{prop:uniformclosure} (Proposition on $\mathbb{B}(\mathcal{X})$ being closed under uniform intersections). Suppose that a sequence $\{g_i\in \mathbb{B}(\mathcal{X}): i\in I\}$ is in $\mathbb{B}(\mathcal{X})$ uniformly, in that there is $e$ from $\mathcal{A}$ such that all for all $i\in I$ and $D\in P(\mathcal{A})$ and $x\in \mathcal{X}$ one has $e: \ominus_D g_i(x,D)\leadsto g_i(x,D)$. Then the intersection $\bigcap_{i\in I} g_i$ is also in $\mathbb{B}(\mathcal{X})$.
\end{prop}
\begin{proof}
Let $h(x,D)=\bigcap_{i\in I} g_i(x,D)$. Then the identity map is a witness to $h(x,D)\leq g_i(x,D)$ for all $i\in I$ and $D\in P(\mathcal{A})$ and $x\in \mathcal{X}$. Then by (\ref{prop:diamonds1}) there is a witness $e^{\prime}$ such that $e^{\prime}:\ominus_D h(x,D)\leadsto \ominus_D g_i(x,D)$ for all $i\in I$ and $D\in P(\mathcal{A})$ and $x\in \mathcal{X}$. Then composing $e^{\prime}$ with the postulated witness $e$ yields a uniform witness $e^{\prime\prime}$ such that $e^{\prime\prime}: \ominus_D h(x,D)\leadsto g_i(x,D)$ for all $i\in I$ and $D\in P(\mathcal{A})$ and $x\in \mathcal{X}$. Then since $h(x,D)=\bigcap_{i\in I} g_i(x,D)$, we also have that $e^{\prime\prime}: \ominus_D h(x,D) \leadsto h(x,D)$ for all $D\in P(\mathcal{A})$ and $x\in \mathcal{X}$.
\end{proof}

The following proposition provides a closure condition which we appeal to repeatedly in what follows. The name we give to this proposition reflects the fact that it's \emph{not} assumed that the antecedent is a member of the Boolean prealgebra:
\begin{prop}\label{prop:mixing} (Proposition on Freedom in the Antecedent)
Suppose that $f\in \mathbb{B}(\mathcal{X})$ and $Q: \mathcal{X}\times P(\mathcal{A}) \rightarrow P(\mathcal{A})$. Let $h: \mathcal{X}\times P(\mathcal{A})\rightarrow P(\mathcal{A})$ be defined by $h(x,D) = (Q(x,D)\Rightarrow f(x,D))$. Then $h\in \mathbb{B}(\mathcal{X})$. 
\end{prop}
\begin{proof}
We must show $\ominus_D h(x,D)\leq h(x,D)$ uniformly in $x$ and $D$. By~(\ref{prop:diamonds6}) and~(\ref{prop:diamonds2}), 
\begin{equation}
(\ominus_D h(x,D)) \wedge Q(x,D) \leq ((\ominus_D Q(x,D))\Rightarrow \ominus_D f(x,D)) \wedge \ominus_D Q(x,D) \leq \ominus_D f(x,D)
\end{equation}
which is $\leq f(x,D)$ since $f\in \mathbb{B}(\mathcal{X})$. Then we have that $\ominus_D h(x,D)\leq (Q(x,D)\Rightarrow f(x,D))$ and so we are done.
\end{proof}

Now we proceed to the construction of the $\Box$ operator. This is defined in terms of the following embedding, which we dub ``$\mu$'' since it's a helpful mnemonic in this context for ``modal'' and since it's unlikely to be confused with any of our other notation. In Flagg's original paper, this map was denoted by $j_{\mathcal{X}}^{\ast}$ and mentioned in the proof of Theorem 10.1 on \citet[p. 139]{Flagg1985aa}.
\begin{definition}\label{jdef}
Let $\mu\colon\mathbb{F}(\mathcal{X})\rightarrow \mathbb{B}(\mathcal{X})$ be defined by $(\mu(f))(x,D)=\ominus_{D} (f(x))$.
\end{definition}
\noindent Sometimes we also write the action of $\mu$ as $f\mapsto \mu_{f}$ instead of $f\mapsto \mu(f)$. Note that $\mu$ varies with $\mathcal{X}$, but since $\mu$'s definition is uniform in the underlying set~$\mathcal{X}$, we don't mark this dependence explicitly in the notation for the map. Further, for any $f \in \mathbb F(\mathcal X)$, let $\overline f \in \mathbb F(\mathcal X \times P(\mathcal{A}))$ be given by $\overline f(x,D) = f(x)$. This allows us to present an equivalent definition of $\mu$ as $\mu(f) = \ominus_\pi \overline f$. The map $f\mapsto \overline{f}$ also helps to keep track of the typing, as one can see by consulting Figure~\ref{figure:1} which records the relation between these various maps. Finally, before proceeding, let's note that $\mu$ does indeed have codomain $\mathbb{B}(\mathcal{X})$. It suffices to show that $\ominus_{\pi}(\mu(f))\leq \mu(f)$. But we may calculate that
\begin{equation}
\ominus_{\pi}(\mu(f))(x,D) =[(\mu(f))(x,D)\Rightarrow \pi(x,D)]\Rightarrow \pi(x,D)= \ominus_D \ominus_D f(x) 
\end{equation}
which is $\leq (\mu(f))(x,D)$ uniformly in $D$ by (\ref{prop:diamonds3}).

The following proposition records some elementary properties of the $\mu$ map. In this, recall that the join and falsum on the codomain $\mathbb{B}(\mathcal{X})$ of $\mu$ are described in line~(\ref{structure0}) of Proposition~\ref{prop:diamondsB}.
\begin{prop}\label{eqn:embedthn}
Let $f,g \in \mathbb F(\mathcal X).$ Then we have:
\vspace{-8mm}
\begin{multicols}{2}
\begin{align}
& f\leq g \mbox{ iff } \mu(f)\leq \mu(g) \label{j1}\\
& \mu(f\wedge g) \equiv \mu(f) \wedge \mu(g) \label{j2}\\
& \mu(f\vee g) \equiv \mu(f) \vee^{\pi} \mu(g)\label{j2.5}
\end{align}
\columnbreak
\begin{align}
& & \mbox{\;} \notag \\
& \mu(\top) \equiv \top \label{j3} \\
& \mu(\bot) \equiv \bot^{\pi} \equiv \pi \label{j4}
\end{align}
\end{multicols}
\vspace{-10mm}
\end{prop}
\begin{proof}
For (\ref{j1}), first suppose that $f\leq g$. Then uniformly in $x\in \mathcal{X}$ we have that $f(x)\leq g(x)$. Then by (\ref{prop:diamonds1}), uniformly in $x\in \mathcal{X}$ and $D\in P(\mathcal{A})$, we have that $\ominus_D f(x)\leq \ominus_D g(x)$, which is just to say that $\mu(f)\leq \mu(g)$.

For the other direction, suppose $\mu(f)\leq \mu(g)$. Then uniformly for $x \in \mathcal{X}$ and $D \in P(\mathcal{A})$, we have $\ominus_D f(x)\leq \ominus_D g(x)$. In particular, let $D = g(x)$. So $\ominus_{g(x)} f(x) \leq \ominus_{g(x)} g(x)$. Hence by line~(\ref{prop:diamonds2}) and~(\ref{prop:diamonds12}), we have $f(x) \leq \ominus_{g(x)} f(x) \leq \ominus_{g(x)} g(x) \leq g(x)$. 

For conjunction, we want to show for all $(x,D) \in \mathcal X \times P(\mathcal{A})$ that $\mu(f \wedge g)(x,D) = \mu(f)(x,D) \wedge \mu(g)(x,D)$. But this holds because by (\ref{prop:diamonds4}), we have that $\mu(f\wedge g)(x,D) = \ominus_D ((f\wedge g)(x)) = \ominus_D(f(x) \wedge g(x)) \equiv \ominus_D f(x) \wedge \ominus_D g(x)$. 

For disjunction, first note by line~(\ref{prop:diamonds2}) and the basic properties of join that $f(x)\vee g(x)\leq \ominus_{D}f(x) \vee \ominus_D g(x)$, so by line~(\ref{prop:diamonds1}) we have $\ominus_D (f(x) \vee g(x)) \leq \ominus_D(\ominus_{D}f(x) \vee \ominus_D g(x))$, which is just to say that $\mu(f\vee g)\leq \mu(f)\vee^{\pi} \mu(g)$. For the other direction, begin by observing that $f,g\leq f\vee g$ implies that $\mu(f), \mu(g)\leq \mu(f\vee g)$ by (\ref{j1}), so that since $\vee^{\pi}$ denotes the join in $\mathbb{B}(\mathcal{X})$ we have that $\mu(f)\vee^{\pi} \mu(g)\leq \mu(f\vee g)$. 

For top and bottom, simply note that $\mu(\top)(x,D)=\ominus_D \top \equiv \top$ uniformly in $D$ by~(\ref{prop:diamonds7}), and $\mu(\bot)(x,D)=\ominus_D \bot = (\bot^{\pi})(x,D)$, since $\bot^\pi = \ominus_{\pi} \bot$.
\end{proof}

For the definition of the $\Box$ operator, we need only one further preliminary definition, namely the $\mathrm{inf}$ map. This map is designated as $j_{\mathcal{X}^{\ast}}$ in the proof of Theorem 10.1 on \citet[p. 139]{Flagg1985aa}.
\begin{definition}\label{inf(f)}
Let $f \in  \mathbb B (\mathcal X)$ and let $x\in \mathcal{X}$. Define the map $\mathrm{inf}\colon \mathbb{B}(\mathcal{X})\rightarrow \mathbb{F}(\mathcal{X})$ by $(\inf(f))(x) = \bigcap_{Z\in P(\mathcal{A})} f(x,Z)$.
\end{definition}
\noindent Finally, we define the $\Box$ operator as follows (cf. the definition in Theorem 10.2 of \citet[p. 140]{Flagg1985aa}):
\begin{definition}\label{defn:boxremark}
 Define the mapping $\Box\colon \mathbb{B}(\mathcal{X}) \rightarrow \mathbb{B}(\mathcal{X})$ by $\Box = \mu \circ \inf$, or equivalently $(\Box f)(x,D) = (\ominus_\pi \overline{\inf f})(x,D) = \ominus_D (\bigcap_{Z\in P(\mathcal{A})}f(x,Z))$.
\end{definition}
\noindent Again, one might consider consulting Figure~\ref{figure:1} to aid in keeping track of the relations between these various maps.

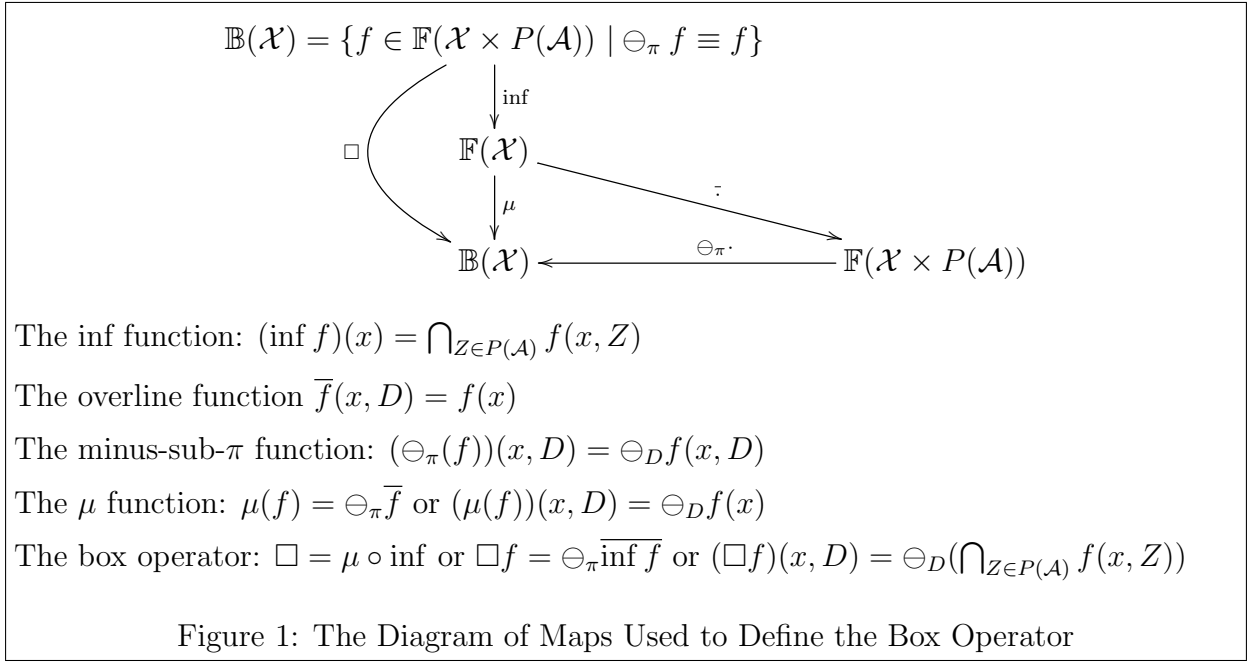
\begin{figure}
\begin{boxedminipage}{\textwidth}
\begin{displaymath}
    \xymatrix{ \mathbb{B}(\mathcal{X})=\{f\in \mathbb{F}(\mathcal{X}\times P(\mathcal{A})) \ \vert \ominus_{\pi} f\equiv f\}\ar[d]^{\mathrm{inf}}\ar@/_4pc/[dd]_{\Box} &  \\
               \mathbb{F}(\mathcal{X})\ar[d]^{\mu} \ar[dr]^{\overline{\cdot}} &  \\
               \mathbb{B}(\mathcal{X}) & \mathbb{F}(\mathcal{X}\times P(\mathcal{A}))\ar[l]_{\ominus_{\pi} \cdot} }
\end{displaymath}

\vspace{2mm}

The inf function: $(\inf f)(x) = \bigcap_{Z\in P(\mathcal{A})} f(x,Z)$ 

\vspace{2mm}

The overline function $\overline{f}(x,D) = f(x)$

\vspace{2mm}

The minus-sub-$\pi$ function: $(\ominus_{\pi}(f))(x,D) = \ominus_D f(x,D)$

\vspace{2mm}

The $\mu$ function: $\mu(f)=\ominus_{\pi}\overline{f}$ or $(\mu(f))(x,D)= \ominus_D f(x)$ 

\vspace{2mm}

The box operator: $\Box = \mu\circ \inf$ or $\Box f =\ominus_{\pi} \overline{\inf f}$  or $(\Box f)(x,D) = \ominus_D (\bigcap_{Z\in P(\mathcal{A})} f(x,Z))$

\vspace{2mm}

\caption{The Diagram of Maps Used to Define the Box Operator}\label{figure:1}
\end{boxedminipage}
\end{figure}

\mbox{\;}

The following proposition gives us some formal indication that the $\Box$ operator acts like a modal operator. Recall that \emph{valid} is a synonym for having top value in the relevant prealgebra. A schema of propositional modal logic can be represented by a formula $\varphi(f_1, \ldots, f_n)$ of propositional modal logic wherein $f_1, \ldots, f_n$ are the basic propositional letters. Let's then say that a formula $\varphi(f_1, \ldots, f_n)$ is \emph{uniformly valid} in $\mathbb{B}(\mathcal{X})$ if there is a single element $e$ of $\mathcal{A}$ such that for all $f_1, \ldots, f_n$ from $\mathbb{B}(\mathcal{X})$, one has that $e$ is a witness to $\varphi(f_1, \ldots, f_n)$ having top value.
\begin{prop}\label{S4}
All of the axioms and theorems of the propositional modal system $\mathsf{S4}$ are uniformly valid on $\mathbb B(\mathcal X)$. Further, (i) one has $f\leq g$ implies $\Box f\leq \Box g$ uniformly in $f, g$. Finally, (ii) one has $g \equiv \Box g$ iff there is $h\in \mathbb F(\mathcal{X})$ such that $g \equiv \mu(h)$.
\end{prop}
\begin{proof}
First we show that the modal axioms $\mathsf{K}, \mathsf{T}, \mathsf{4}$ are uniformly valid, and then we verify (i)-(ii), and then we finish the proof by showing that all the theorems of the propositional modal system $\mathsf{S4}$ are uniformly valid on $\mathbb B(\mathcal X)$. For $\mathsf{K}$, it suffices to show that $\Box f \wedge \Box (f\Rightarrow g)  \leq \Box g$.  First note that  $\bigcap_{Z\in P(\mathcal{A})}(f(x,Z) \wedge (f\Rightarrow g)(x,Z)) \leq  \bigcap_{Z\in P(\mathcal{A})} g(x,Z)$. Thus one has the following, where the equivalence follows from line~(\ref{prop:diamonds4}) and the inequality follows from line~(\ref{prop:diamonds1}):
\begin{equation} 
\Box f \wedge \Box (f\Rightarrow g) \equiv \ominus_\pi(\overline{\inf f} \wedge \overline{\inf (f\Rightarrow g)}) \leq \ominus_\pi \overline{\inf g} =\Box g
\end{equation}
For the $\mathsf{T}$-axiom, it suffices to show $\Box f \leq  f$. So suppose $f\in  \mathbb{B}(\mathcal{X})$ and note that $\overline{\inf f} \leq f$ by means of the identity function. Then by lines~(\ref{prop:diamonds1}) and~(\ref{def B}) we have $\Box f = \ominus_{\pi} \overline{\inf f} \leq  \ominus_{\pi} f \equiv f$. Finally, for the $\mathsf{4}$-axiom, it suffice to show that $\Box f \leq \Box \Box f$. By~(\ref{prop:diamonds2}), choose a uniform witness $e$ for all $I\leq \ominus_Z I$ when $I,Z\in P(\mathcal{A})$. Then by setting $I= (\inf f)(x)$, we have that $e$ is also a uniform witness over all $Z \in P(\mathcal{A})$ and $x\in \mathcal{X}$ for $\inf f(x) \leq \ominus_Z \inf f(x)$. Hence we have:
\begin{equation}
(\inf f)(x) \leq \bigcap_{Z\in P(\mathcal{A})} \ominus_Z (\inf f)(x) =  \bigcap_{Z\in P(\mathcal{A})} (\Box f) (x,Z) = (\inf \Box f)(x) 
\end{equation}
Thus by line~(\ref {prop:diamonds1}) one has $\Box f = \ominus_{\pi} \overline{\inf f} \leq \ominus_{\pi} \overline{\inf \Box f} = \Box \Box f$. 

Turning to~(i), suppose first that $f\leq g$. Then $f(x,Z)\leq g(x,Z)$ uniformly in~$x,Z$, so that $\mathrm{inf}(f)\leq \mathrm{inf}(g)$ and hence $\overline {\inf f}\leq \overline{\inf g}$. From this it follows by line~(\ref{prop:diamonds1}) that we uniformly have $\ominus_\pi \overline {\inf f} \leq \ominus_\pi \overline {\inf g}$, hence $\Box f\leq \Box g$. 

For (ii), the first direction is trivial: if $\Box g \equiv g$, then let $h =\inf(g)$. For the converse, suppose that $g\equiv \mu_h$ for some $h \in \mathbb F(\mathcal X)$. Then it suffices to show $\Box \mu_h \equiv \mu_h$. By the $\mathsf{T}$-axiom we may focus on showing that $\mu_h\leq \Box \mu_h$, which by definition is the claim that $\ominus_\pi \overline h \leq \ominus_\pi \overline {\inf \mu_h}$. For this it suffices by~(\ref{prop:diamonds1}) to show that $h(x) \leq (\inf \mu_h)(x) = \bigcap_{Z \in P(\mathcal{A})} \ominus_Z h(x)$ for all $x \in \mathcal X$. But this follows from the fact that~(\ref{prop:diamonds2}) holds uniformly.

Finally, we show by induction on length of proof that all of the theorems of the propositional modal system $\mathsf{S4}$ are uniformly valid on $\mathbb B(\mathcal X)$. The base cases correspond to the $\mathsf{K}, \mathsf{T}, \mathsf{4}$ axioms and the classical propositional tautologies; the earlier parts of this proof handle the former and the prealgebra being Boolean takes care of the latter. The induction steps amount to showing that the inference rules, namely modus ponens and necessitation, preserve uniform validity. The case of modus ponens follows from the axioms governing the conditional in a Heyting prealgebra. For the necessitation rule, suppose that $\varphi(f_1, \ldots, f_n)$ has top value in $\mathbb B(\mathcal X)$ uniformly in $f_1, \ldots, f_n$. We must then show that $\Box \varphi(f_1, \ldots, f_n)$ has top value uniformly in $f_1, \ldots, f_n$. But since $\varphi(f_1, \ldots, f_n)$ has top value, $\top \leq \varphi(f_1, \ldots, f_n)$. Then by (i), one has $\Box \top \leq \Box \varphi(f_1, \ldots, f_n)$ uniformly in $f_1, \ldots, f_n$. But by (ii), one has that $\Box \top\equiv \top$ since $\top \equiv \mu(\top)$ by~(\ref{j3}).
\end{proof}

Finally, we close this section by noting how the $\Box$-operator permits us to describe the behavior of the $\mu$-embedding on conditionals and intersections.
\begin{prop}\label{prop:actionconiditionals} (Proposition on Action of Embedding on Conditionals)
\begin{equation}\label{corthebobm}
\mu(f\Rightarrow g) \equiv \Box (\mu(f)\Rightarrow \mu(g))
\end{equation}
\end{prop}
\begin{proof}
Since $f \wedge (f\Rightarrow g)\leq g$, we have that $\overline f \wedge \overline{f\Rightarrow g}\leq \overline g$. By an application of lines~(\ref{prop:diamonds1}) and~(\ref{prop:diamonds4}) we have $\ominus_\pi \overline f \wedge \ominus_\pi \overline{f\Rightarrow g}\leq \ominus_\pi \overline g$, i.e., $\mu_f\wedge \mu_{f\Rightarrow g}\leq \mu_g$, which amounts to $\mu_{f\Rightarrow g}\leq \mu_f \Rightarrow \mu_g$. Then by parts~(i)-(ii) of the previous proposition we have $\mu_{f\Rightarrow g}\equiv \Box \mu_{f\Rightarrow g} \leq \Box (\mu_f \Rightarrow \mu_g)$. For the converse, note that $f(x)\leq \ominus_{g(x)} f(x)$ by line~(\ref{prop:diamonds2}), and that this holds uniformly, which implies that $(\ominus_{g(x)} f(x) \Rightarrow g(x)) \leq (f(x)\Rightarrow g(x))$. Since $g(x)\equiv \ominus_{g(x)} g(x)$ uniformly by~(\ref{prop:diamonds12}), we have by substitution that $(\ominus_{g(x)} f(x) \Rightarrow \ominus_{g(x)} g(x)) \leq (f(x)\Rightarrow g(x))$. By considering $Z=g(x)$, we thus see that  $\bigcap_{Z\in P(\mathcal{A})}(\ominus_Z f(x)\Rightarrow \ominus_Z g(x)) \leq (f(x)\Rightarrow g(x))$, i.e., $\bigcap_{Z\in P(\mathcal{A})} (\ominus_\pi \overline f (x,Z) \Rightarrow \ominus_\pi\overline g (x,Z)) \leq (f(x) \Rightarrow g(x))$. So $\mathrm{inf} (\ominus_\pi \overline f(x)\Rightarrow \ominus_\pi \overline g(x)) \leq (f(x)\Rightarrow g(x))$. From this it follows that  ${\inf (\mu_f\Rightarrow \mu_g)}\leq (f\Rightarrow g)$ and hence that $\Box (\mu_f\Rightarrow \mu_g) = \ominus_\pi \overline{\inf (\mu_f\Rightarrow \mu_g)}\leq \ominus_\pi  \overline{f\Rightarrow g} = \mu_{f\Rightarrow g}$.
\end{proof}

\begin{prop}\label{prop:proponintersectionembed} (Proposition on Action of Embedding on Intersections). Suppose that the sequence $\{f_i\in \mathbb{F}(\mathcal{X}): i\in I\}$ is such that $\bigcap_{i\in I} \mu(f_i)$ is in $\mathbb{B}(\mathcal{X})$. Then $\mu(\bigcap_{i\in I} f_i) \equiv \Box \bigcap_{i\in I} \mu(f_i)$.
\end{prop}
\begin{proof}
First note that by the monotonicity of $\mu$ (cf. line~(\ref{j1})) and the definition of $\Box=\mu\circ \mathrm{inf}$, the desired equivalence follows from the equivalence $\bigcap_{i\in I} f_i\equiv \mathrm{inf} (\bigcap_{i\in I} \mu(f_i))$ in $\mathbb{F}(\mathcal{X})$. But this equivalence just is $\bigcap_{i\in I} f_i\equiv \bigcap_{D\in P(\mathcal{A})} \bigcap_{i\in I} \ominus_D f_i$, which follows directly from Proposition~\ref{prop:proponintersection}.
\end{proof}


\section{The Status of S5}\label{secS5isfalse}

In Proposition~\ref{S4}, we showed that $\mathsf{S4}$ was uniformly valid on the structure $\mathbb B(\mathcal X)$. In this section, we show that $\mathsf{S5}$ is not uniformly valid on $\mathbb B(\mathcal X)$, but that there is a special case in which it is valid in a non-uniform sense. Recall in general that in the setting of $\mathsf{S4}$, the $\mathsf{S5}$ schema is equivalent to $\varphi \Rightarrow \Box \Diamond \varphi$. Before beginning, let us define the following element of $P(\mathcal A)$: \begin{equation}\label{defn:IamdefnofQ}
Q = \bigcap_{Z\in P(\mathcal A)} [(\top \Rightarrow Z) \Rightarrow (\top \Rightarrow Z)]
\end{equation}
Note that $Q$ is nonempty (e.g., it contains at least the identity element $skk$).

\begin{prop}\label{S5 lemma}
Given a function $f \in \mathbb B (\mathcal X)$, let us define $M_f\colon \mathcal X \to P(\mathcal A)$ to be $\inf(f \Rightarrow \bot^\pi)$. Then one has that $M_f(x) = \bigcap_{Z\in P(\mathcal A)} (f(x,Z) \Rightarrow (\top \Rightarrow Z))$ and 
\begin{equation}
   (\Box \Diamond f) (x, D) = \left\{
     \begin{array}{lr}
       \ominus_D Q & : M_f(x) = \bot \\
       \ominus_D \bot & : M_f(x) \neq \bot 
     \end{array}
   \right.
\end{equation}
\end{prop}

\begin{proof}
That the two definitions of $M_f$ are equivalent uses the fact that $\bot^\pi (x,D) = \ominus_D \bot = ((\bot \Rightarrow D) \Rightarrow D) = (\top \Rightarrow D)$. We will use this fact again repeatedly in what follows. Recall that in general the symbol $\Diamond$ is short for $\neg \Box \neg$ and that $\neg f$ in the structure $\mathbb B(\mathcal X)$ is shorthand for $f \Rightarrow \bot^\pi$. Thus one has $\Box \Diamond f = \Box[\Box(f \Rightarrow \bot^\pi) \Rightarrow \bot^\pi]$ and hence
\begin{equation}
\Box \Diamond f = \ominus_\pi \overline{\inf(\ominus_\pi \overline{\inf (f\Rightarrow \bot^\pi)} \Rightarrow \bot^\pi}) = \ominus_\pi \overline{\inf(\ominus_\pi \overline{M_f} \Rightarrow \bot^\pi})
\end{equation}
Now fix $x \in \mathcal X$. Then $\Box \Diamond f (x,D) = \ominus_D \bigcap_{Z\in P(\mathcal A)} (\ominus_Z M_f(x) \Rightarrow (\top \Rightarrow Z))$ for all $D$. If $M_f(x) = \bot$, then $\Box \Diamond f(x,D) = \ominus_D \bigcap_{Z\in P(\mathcal A)} ((\top \Rightarrow Z) \Rightarrow (\top \Rightarrow Z)) = \ominus_D Q$. If  $M_f(x) \neq \bot$, then note that when $Z = \bot$, we have 
\begin{equation}
(\ominus_Z M_f(x) \Rightarrow (\top \Rightarrow Z)) =  ( (\bot \Rightarrow \bot) \Rightarrow (\top \Rightarrow \bot)) = (\top \Rightarrow (\top \Rightarrow \bot)) = \bot
\end{equation}
Hence, $\Box\Diamond f(x,D)$ simplifies to $\ominus_D\bot$.
\end{proof}

Then let's show:
\begin{prop}
$\mathsf {S5}$ is not uniformly valid in $\mathbb B(\mathcal X)$.
\end{prop}
\begin{proof}
We define two distinct elements $f_0,f_1$ of $\mathbb{B}(\mathcal{X})$, namely
\begin{equation}\label{eqn:arewedoneyet}
f_0(x,D) = \top, \hspace{10mm} f_1(x,D) = (\top \Rightarrow D)
\end{equation}
First note that $f_0,f_1 \in \mathbb B(\mathcal X)$. In the case of $f_0$, any element of $\mathcal A$ trivially witnesses $\ominus_D \top \leq \top$. For the case of $f_1$, we have that
\begin{equation}
(\ominus_\pi f_1)(x,D) = \ominus_D (\top \Rightarrow D) = \ominus_D ((\bot \Rightarrow D) \Rightarrow D) = \ominus_D\ominus_D \bot \equiv \ominus_D\bot = f_1(x,D)
\end{equation}
Now, using the notation of Proposition~\ref{S5 lemma}, we have that 
\begin{equation}
M_{f_0}(x) =  \bigcap_{Z\in P(\mathcal A)} (\top \Rightarrow (\top \Rightarrow Z)) = \bot 
\end{equation}
and  \begin{equation} M_{f_1}(x) = \bigcap_{Z\in P(\mathcal A)}((\top \Rightarrow Z) \Rightarrow (\top \Rightarrow Z)) = Q \neq \bot,
\end{equation}
where $Q$ is defined as in line~(\ref{defn:IamdefnofQ}). Hence by Proposition~\ref{S5 lemma}, we have  $\Box \Diamond {f_0}(x,D) = \ominus_D Q$ and $\Box \Diamond {f_1}(x,D) = (\top \Rightarrow D)$.

Now suppose for the sake of contradiction that $\mathsf{S5}$ were uniformly valid in $\mathbb{B}(\mathcal{X})$. Then there would be an $e$ from $\mathcal{A}$ such that both $e:f_0\leadsto \Box \Diamond f_0$ and $e:f_1\leadsto \Box \Diamond f_1$. Thus for all $D \in P(\mathcal A)$ and all $x\in \mathcal{X}$, we have both $e : f_0(x,D) \leadsto \ominus_D Q$ and  $e: f_1(x,D) \leadsto (\top \Rightarrow D)$. The first of these means we have $e \colon \top \leadsto \ominus_D Q$ for every $D \in P(\mathcal A)$. The second of these means that we have $e\colon (\top \Rightarrow D) \leadsto (\top \Rightarrow D)$ for every $D \in P(\mathcal A)$.

Now let $a$ and $b$ be two distinct members of $\mathcal A$. By Proposition~\ref{prop:onpcas}, choose $e^\prime \in \mathcal A$ such that $e^\prime n = a$ for all $n \in \mathcal A$. Then $e^{\prime}$ is in $(\top \Rightarrow \{a\})$, so by the above we have both $e  e^{\prime}\in (\top\Rightarrow \{a\})$ and $e e^{\prime}\in \ominus_{\{b\}} Q = ((Q\Rightarrow \{b\}) \Rightarrow \{b\})$. This is a contradiction, because for any $m\in (Q\Rightarrow \{b\})$ we obtain both $e e^{\prime} m =b$ and $e e^{\prime} m=a$.
\end{proof}

Now, there's a natural enough non-uniform version of validity for $\mathsf{S5}$, and it turns out that it holds when $\mathcal{X}$ has exactly one element and fails whenever $\mathcal{X}$ has more than one element:
\begin{prop}\label{prop:S5degncase}
(I) Suppose that $\mathcal{X}$ has more than one element. Then there is $f\in \mathbb{B}(\mathcal{X})$ such that there is no $e$ from $\mathcal{A}$ such that $e:f\leadsto \Box \Diamond f$.

(II) Suppose that $\mathcal{X}$ has only one element. Then for each $f\in \mathbb{B}(\mathcal{X})$ there is $e$ from $\mathcal{A}$ such that $e: f\leadsto \Box \Diamond f$. 
\end{prop}
\begin{proof}
For (I), let $x_0, x_1$ be two distinct elements of $\mathcal{X}$ and define \emph{\`a la} equation~(\ref{eqn:arewedoneyet}) the following element $f$:
\begin{equation}
f(x,D)=\begin{cases}
\top      & \text{if $x=x_0$}, \\
\top \Rightarrow D      & \text{otherwise}.
\end{cases}
\end{equation}
Then the proof proceeds exactly as in the proof of the previous proposition. 

For (II), $\mathcal X$ is a singleton set $\{x\}$. Let $f \in \mathbb B(\mathcal X)$. Using the notation of Proposition~\ref{S5 lemma}, in case $f \equiv \bot^\pi$, we easily have $M_f(x) \neq \bot$. So by Proposition~\ref{S5 lemma}, we then have $f(x,D) \leq (\Box \Diamond f)(x,D)$. In the alternate case, wherein $f \not \equiv \bot^\pi$, there must be some $Z \in P(\mathcal A)$ such that $f(x,Z) \not \leq \bot^\pi (Z)$. But then $M_f(x) = \bot$, so Proposition~\ref{S5 lemma} gives $(\Box \Diamond f)(x,D) = \ominus_D Q$. Let $j^{\prime}$ be any  uniform witness to line~(\ref{prop:diamonds2}) and $j$ be any member of $Q$. By Proposition~\ref{prop:onpcas}, choose $j^{\prime\prime}$ in $\mathcal{A}$ such that $j^{\prime\prime} n = j^{\prime} j$. Hence $j^{\prime\prime}$ witnesses $f \leq \Box \Diamond f$.
\end{proof}

It's worth emphasizing that case~II of the above proposition is exactly the situation of $\mathbb B$ from~(\ref{def Breal3}). As we will find in the next section, the semantics for closed sentences will produce functions precisely in $\mathbb{B}$. Hence, we will have that $\mathsf{S5}$ is valid for sentences in the non-uniform sense.


\section{The Modal Semantics}\label{sec:intuitiontomodal}

\noindent In this section, we provide a semantics for modal formulas $\varphi(\overline x)$ by defining for each such formula  a corresponding function $\|\varphi(\overline x)\|_\mu$ from one of the Boolean prealgebras $\mathbb B (\mathcal X)$. We begin with the following definition, which can be found in~\cite[volume~2 Definition 6.2-3 pp. 709-711]{Troelstra1988aa}, although they work with Heyting algebras rather than Heyting prealgebras, and they do not include the information about the quantifiers. The information about the quantifiers only comes into play in the last clauses of Definition~\ref{DefValuationR} and constraints on the quantifiers are only put in place later in the section (cf. Definition~\ref{defn:Qproperties}).  In this section, we'll be working with arbitrary first-order signatures~$L$, which as usual are just given by a collection of constant, relation, and function symbols along with a specification of the arities of the relation and function symbols.

\begin{definition}\label{defn:heytinggeneral}
Let $\mathbb{H}$ be a Heyting prealgebra and let $L$ be a signature. Then an \emph{$\mathbb{H}$-valued $L$-structure} $\mathcal{N}$ \emph{with quantifier $\mathbb{Q}$} is given by an underlying set $N$, a map $\| \cdot=\cdot \|: N^2\rightarrow \mathbb{H}$, a map $\mathbb{Q}: N\rightarrow \mathbb{H}$, and a distinguished element of $N$ for each constant symbol, a map $\|R(\cdot, \ldots, \cdot)\|:N^n\rightarrow \mathbb{H}$ for each $n$-ary relation symbol $R$, and a map $f:N^n\rightarrow N$ for every $n$-ary function symbol~$f$, such that for each $n$-ary relation symbol $R$ and each $n$-ary function symbol $f$, and all $a,b,c, a_1, \ldots, a_n, b_1, \ldots, b_n$ in $N$  one has
\begin{align}
&  \| a = a \|  =  \top, \hspace{5mm}  \|a=b\| = \|b=a\|, \hspace{5mm} \|a=b\| \wedge \|b=c\|\leq \|a=c\| \notag \\
& \| a_1=b_1 \| \wedge \cdots \wedge \|a_n=b_n\| \wedge \|R(a_1, \ldots, a_n)\| \leq  \|R(b_1, \ldots, b_n)\|  \notag \\
 & \| a_1=b_1 \| \wedge \cdots \wedge \|a_n=b_n\|  \leq  \|f(a_1, \ldots, a_n)=f(b_1, \ldots, b_n)\| \notag 
\end{align}
\end{definition}

Let $\mathcal{A}$ be a pca. We are interested in the case of $\mathbb{H}=P(\mathcal{A})$. In this setting, there is a natural strengthening of the notion in Definition~\ref{defn:heytinggeneral} wherein we require that there be uniform witnesses to the above conditions. So we define:
\begin{definition}\label{defn:uniform}
Let $L$ be a signature and let $\mathcal{A}$ be a pca. A \emph{uniform $P(\mathcal{A})$-valued $L$-structure $\mathcal{N}$ with quantifier $\mathbb{Q}$} is given by an underlying set $N$, a map $\| \cdot=\cdot \|: N^2\rightarrow P(\mathcal{A})$, a map $\mathbb{Q}: N\rightarrow P(\mathcal{A})$, and elements $e_{\mathrm{ref1}}, e_{\mathrm{ref2}}, e_{\mathrm{sym}}, e_{\mathrm{tran}}$ from $\mathcal{A}$, and a distinguished element of $N$ for each constant symbol, a map $\|R(\cdot, \ldots, \cdot)\|:N^n\rightarrow P(\mathcal{A})$ and an element $e_R$ from $\mathcal{A}$ for each $n$-ary relation symbol $R$, and a map $f:N^n\rightarrow N$ and element $e_f$ from $\mathcal{A}$ for every $n$-ary function symbol~$f$, such that for each $n$-ary relation symbol $R$ and each $n$-ary function symbol $f$, and all $a,b,c, a_1, \ldots, a_{n}, b_1, \ldots, b_{n}$ in $N$  one has
\begin{align}
& e_{\mathrm{ref1}}:\| a = a \|  \leadsto  \top, \hspace{10mm} e_{\mathrm{ref2}}: \top \leadsto \| a = a \| \notag \\
& e_{\mathrm{sym}}:\|a=b\| \leadsto \|b=a\|, \hspace{5mm} e_{\mathrm{tran}}: \|a=b\| \wedge \|b=c\|  \leadsto  \|a=c\| \notag \\
& e_R: \| a_1=b_1 \| \wedge \cdots \wedge \|a_n=b_n\| \wedge \|R(a_1, \ldots, a_n)\| \leadsto  \|R(b_1, \ldots, b_n)\|\notag  \\
& e_f:\| a_1=b_1 \| \wedge \cdots \wedge \|a_{n}=b_{n}\| \leadsto  \| f(a_1, \ldots, a_n)=f(b_1, \ldots, b_n)\| \notag 
\end{align}
\end{definition}

The following definition contains the semantics for uniform $P(\mathcal{A})$-valued structures. We follow the usual conventions in assuming that each of our languages has a constant symbol for each of the elements in the model under consideration. As one can see, the quantifier $\mathbb{Q}$ is providing the semantics for the existential and universal quantifiers.
\begin{definition}\label{DefValuationR}
Let $\mathcal{N}$ be a uniform $P(\mathcal{A})$-valued $L$-structure with quantifier~$\mathbb{Q}$. For every $L$-formula $\varphi(\overline{x})\equiv \varphi(x_1, \ldots, x_n)$ with all free variables displayed, define the map $\|\varphi(\overline{x})\| \colon N^n\rightarrow P(\mathcal{A})$, a member of $\mathbb{F}(N^n)$, inductively as follows, wherein the base cases for atomics come from Definition~\ref{defn:uniform}:
\vspace{-8mm}
\begin{multicols}{2}
\begin{align}
& \| \bot\| = \bot \notag \\
& \| (\varphi \wedge \psi)(\overline{a})\|  = \| \varphi(\overline{a})\| \wedge \| \psi(\overline{a})\|  \notag \\
& \| (\varphi \vee \psi)(\overline{a})\|  =  \| \varphi(\overline{a})\| \vee \| \psi(\overline{a})\| \notag 
\end{align}
\columnbreak
\begin{align}
& & \mbox{\;} \notag \\
& \| (\varphi \Rightarrow \psi)(\overline{a})\| \; = \; \| \varphi(\overline{a})\| \Rightarrow \| \psi(\overline{a})\|  \notag \\
& \| \exists  z \; \varphi(\overline{a}, z)\| = \bigcup_{c\in N} ( \mathbb{Q}(c) \wedge \|\varphi(\overline{a},c)\| ) \notag \\
& \| \forall  z \; \varphi(\overline{a}, z)\| = \bigcap_{c\in N} (\mathbb{Q}(c) \Rightarrow  \| \varphi(\overline{a},z)\|) \notag 
\end{align}
\end{multicols}
\vspace{-8mm}
\end{definition}

Then we argue that there are witnesses from the underlying pca for all the substitutions:
\begin{prop}\label{prop:uniform valuation}
Let $\mathcal{N}$ be a uniform $P(\mathcal{A})$-valued $L$-structure with quantifier~$\mathbb{Q}$. For every formula $\varphi(\overline{x})\equiv \varphi(x_1, \ldots, x_n)$ there is an element $e_{\varphi}$ of $\mathcal{A}$ such that for all $a_1, \ldots, a_n, b_1, \ldots, b_n$ from $N$ one has
\begin{equation}\label{helper 110}
e_{\varphi}: \| a_1=b_1 \| \wedge \cdots \wedge \|a_n=b_n\| \wedge \|\varphi(a_1, \ldots, a_n)\| \leadsto  \|\varphi(b_1, \ldots, b_n)\|
\end{equation}
\end{prop}
\begin{proof}
The base step for atomic formulas $R(\overline a)$ follows directly from $e_R$ in Definition~\ref{defn:uniform}. For $\bot$ the assertion is trivial. For atomics $t(\overline{x})=s(\overline{x})$, first by an induction on complexity of terms, one shows that for each term $t(\overline{x})$ one has a uniform witness
\begin{equation*}
e_t:\| a_1=b_1 \| \wedge \cdots \wedge \|a_n=b_n\|  \leadsto  \|t(a_1, \ldots, a_n)=t(b_1, \ldots, b_n)\|
\end{equation*}
For instance, if $t(x)=f(g(x))$, then one obtains $e_t$ which witnesses the reduction $\|a=b\|\leq \|g(a)=g(b)\|\leq \|f(g(a))=f(g(b))\|$, by composing $e_f$ and $e_g$ from Definition~\ref{defn:uniform} in the pca. Second, suppose that $t(x), s(x)$ are terms. Then by using $e_t, e_s$ in conjunction with $e_{\mathrm{sym}}, e_{\mathrm{trans}}$ one obtains a uniform witness to the reduction
\begin{align*}
& \|a=b\| \wedge \|t(a)=s(a)\| \leq  \|a=b\|  \wedge \|t(a)=t(b)\|\wedge \|t(a)=s(a)\| \wedge \|s(a)=s(b)\|
 \notag \\
& \leq \|a=b\|  \wedge \|t(b)=t(a)\|\wedge \|t(a)=s(a)\| \wedge \|s(a)=s(b)\| \leq \|t(b)=s(b)\|
\end{align*}

For the steps corresponding to the propositional connectives, let us first abbreviate $E_{a,b}=\|a=b\|$ and $\Phi_a=\|\varphi(a)\|$, $\Phi_b=\|\varphi(b)\|$, $\Psi_a=\|\psi(a)\|$, and $\Psi_b=\|\psi(b)\|$. For conjunction, note that $E_{a,b} \wedge \Phi_a\leq \Phi_b$ and $E_{a,b}\wedge \Psi_a\leq \Psi_b$ implies $E_{a,b}\wedge \Phi_a \wedge \Psi_a \leq \Phi_b\wedge \Psi_b$. For disjunction, note that $E_{a,b} \wedge \Phi_a\leq \Phi_b$ and $E_{a,b}\wedge \Psi_a\leq \Psi_b$ implies the following: $E_{a,b}\wedge (\Phi_a\vee \Psi_a) \leq (E_{a,b} \wedge \Phi_a)\vee (E_{a,b}\wedge \Phi_b)\leq \Phi_b\vee \Psi_b$. For the conditional, one notes first that the inductive hypothesis also gives $E_{a,b}\wedge \Phi_b \leq \Phi_a$. Then one has $E_{a,b} \wedge (\Phi_a\Rightarrow \Psi_a) \wedge \Phi_b \leq E_{a,b} \wedge (\Phi_a\Rightarrow \Psi_a) \wedge \Phi_a \leq E_{a,b} \wedge \Psi_a \leq \Psi_b$, which of course implies that $E_{a,b} \wedge (\Phi_a\Rightarrow \Psi_a)\leq (\Phi_b\Rightarrow \Psi_b)$.

For the quantifiers, first consider the existential quantifier, and suppose that the pair $pmn$ is in $\|a=b\|\wedge \|\exists \; x \; \varphi(a,x)\|$. Then for some $c\in N$ one has that $n$ is in $\mathbb{Q}(c) \wedge \|\varphi(a,c)\|$. Then $e_{\varphi}(pm(p_1n))$ is in $\|\varphi(b,c)\|$ by induction hypothesis and so $p((p_0n)(e_{\varphi}(pm(p_1n))))$ is in $\|\exists \; x \; \varphi(b,x)\|$. For the universal quantifier, suppose that $pmn$ is in $\|a=b\|\wedge \|\forall \; x \; \varphi(a,x)\|$. Then for all $c\in N$ one has that $n$ is in $(\mathbb{Q}(c)\Rightarrow \|\varphi(a,c)\|)$. Let $t(x_1, x_2, x_3, x_4)\equiv x_1(px_2(x_3x_4))$, and by Proposition~\ref{prop:onpcas} choose $f$ such that $f{x_1}{x_2}{x_3}{x_4} =t(x_1, x_2, x_3, x_4)$. Then one has $f{e_{\varphi}}mn{\ell} = e_{\varphi}(p m (n\ell))$, so that by induction hypothesis for all $c\in N$ one has that $f{e_{\varphi}}mn$ is in  $(\mathbb{Q}(c)\Rightarrow \|\varphi(b,c)\|)$.
\end{proof}

 Now we finally come to the modal semantics:
\begin{definition}\label{thetheorem}
Let $\mathcal{N}$ be a uniform $P(\mathcal{A})$-valued $L$-structure with quantifier $\mathbb{Q}$ and underlying set $N$. Then we define \emph{the modal $\mathbb{B}$-valued $L$-structure $\mu[\mathcal{N}]$ with quantifier~$\mathbb{Q}$} by defining a valuation map $\|\varphi(x_1, \ldots, x_n)\|_{\mu}\colon N^n \times P(\mathcal{A}) \to P(\mathcal{A})$ for each modal $L$-formula $\varphi(\overline{x})\equiv \varphi(x_1, \ldots, x_n)$ as follows, wherein $t,s$ are $L$-terms and $R$ is an $L$-relation symbol, and where we write the action as $(\overline{a},D)\mapsto \|\varphi(\overline{a})\|_{\mu}(D)$:
\begin{align}
& \hspace{10mm} \| t(\overline{a})=s(\overline{a})\|_\mu(D) = \ominus_D (\|t(\overline{a}) = s(\overline{a}) \|) \notag \\ 
& \hspace{10mm}\| R(\overline{a})\|_{\mu}(D) = \ominus_D (\|R(\overline a)\|), \hspace{5mm} \|\bot \|_{\mu}(D)  = \ominus_D \bot \notag  \\
& \hspace{10mm} \| (\varphi \wedge \psi)(\overline{a}) \|_{\mu}(D) = (\|\varphi(\overline{a})\|_{\mu}(D) \wedge \|\psi(\overline{a})\|_{\mu}(D))\notag \\
&\hspace{10mm}  \| (\varphi \vee \psi)(\overline{a}) \|_{\mu}(D)  = (\| \varphi(\overline{a})\|_{\mu}(D)  \vee^\pi \|\psi(\overline{a})\|_{\mu}(D)) \notag \\
& \hspace{10mm} \| (\varphi \Rightarrow \psi)(\overline{a}) \|_{\mu}(D)  = (\| \varphi(\overline{a})\|_{\mu}(D) \Rightarrow \|\psi(\overline{a})\|_{\mu}(D)) \notag  \\ 
& \hspace{10mm}\| \exists z \; \varphi(\overline{a}, z)\|_{\mu}(D) =  \ominus_D \bigcup_{c\in N} [\mathbb{Q}(c) \wedge \|\varphi(\overline{a},c)\|_{\mu}(D)] \notag \\
& \hspace{10mm}\| \forall z \; \varphi(\overline{a}, z)\|_{\mu}(D) =  \bigcap_{c\in N} [\mathbb{Q}(c) \Rightarrow \|\varphi(\overline{a},c)\|_{\mu}(D)]\notag  \\ 
&\hspace{10mm} \| \Box \varphi(\overline{a})\|_{\mu}(D) = \Box (\|\varphi(\overline{a})\|_{\mu}(D)) \notag
\end{align}
\end{definition}

The following proposition says that, as the name suggests, the valuation maps lie in the Boolean algebras $\mathbb{B}(\mathcal{X})$ and $\mathbb{B}$~(defined in~(\ref{def B}) and (\ref{def Breal3}) from \S\ref{sec:heytingtobooleanmodal}):
\begin{prop}\label{prop:uniformme}
(I) For each modal $L$-formula $\varphi(x_1, \ldots, x_n)$ with all free variables displayed, the function $\|\varphi(x_1, \ldots, x_n)\|_{\mu}\colon N^n \times P(\mathcal{A}) \to P(\mathcal{A})$ is uniformly in $\mathbb B(N^n)$, in that for each $\varphi(\overline{x})$ there is an element~$i_{\varphi}$ of $\mathcal{A}$ such that for all $\overline{a}$ from $N$ and all $D$ from $P(\mathcal{A})$ one has that  $i_{\varphi}: \ominus_D \|\varphi(\overline{a})\|_{\mu} (D)\leadsto \|\varphi(\overline{a})\|_{\mu} (D)$.

(II) For each modal $L$-sentence $\varphi$, one has that $\|\varphi\|_{\mu}$ is an element of $\mathbb{B}$.
\end{prop}
\begin{proof}
Let's begin with (I). This is obvious in the case of the atomics, as well as in the case of the disjunctions, existentials, and the box since the $\|\cdot\|_{\mu}$-valuations all begin with $\ominus_D$ in these cases. For the universal quantifier, one simply appeals to the induction hypothesis and the proposition on $\mathbb{B}(\mathcal{X})$ being closed under uniform intersections (Proposition~\ref{prop:uniformclosure}) and the Proposition on Freedom in the Antecedent (Proposition~\ref{prop:mixing}). The inductive steps for the conjunctions and conditionals follow easily from lines (\ref{prop:diamonds4}) and (\ref{eqn:seqasdfdas}) along with the induction hypothesis. This finishes the argument for (I). 

For (II), any $L$-sentence $\varphi$ can be written as $\psi(c)$ for an $L$-formula $\psi(x)$ and a constant symbol $c$. By (I), $\|\psi(x)\|_{\mu}: N\times P(\mathcal{A}) \rightarrow P(\mathcal{A})$ and is an element of $\mathbb{B}(N)$, i.e., one has $\ominus_D \|\psi(a)\|_{\mu}(D)\leq \|\psi(a)\|_{\mu}(D)$ uniformly in $a$ from $N$ and $D$ from $P(\mathcal{A})$. Then by evaluating at constant~$c$, we obtain $\|\varphi\|_{\mu}:P(\mathcal{A}) \rightarrow P(\mathcal{A})$ with $\ominus_D \|\varphi\|_{\mu}(D)\leq \|\varphi\|_{\mu}(D)$ uniformly in $D$ from $P(\mathcal{A})$. But this is precisely the condition to be an element of $\mathbb{B}$, as defined in~(\ref{def Breal3}). 
\end{proof}

Note that Part~(II) of this proposition, in conjunction with Proposition~\ref{prop:S5degncase}, implies that if $\varphi$ is a modal $L$-sentence, then $\|\varphi\Rightarrow \Box \Diamond \varphi\|_{\mu}$ has top value in $\mathbb{B}$, so that $\mathsf{S5}$ holds for each sentence taken one by one.

Further, we can do substitution in the modal structure just as in the original structure:
\begin{prop}\label{prop:subsnecwow}
For every modal $L$-formula $\varphi(\overline{x})\equiv \varphi(x_1, \ldots, x_n)$ there is an element $e_{\varphi}$ of $\mathcal{A}$ such that for all $\overline{a}, \overline{b}$ from $N$ and all $D \in P(\mathcal{A})$, one has:
\begin{equation}\label{eqn:whatwewanttoverifynow}
e_\varphi\colon \bigwedge_{1\leq i \leq n} \|a_i = b_i\|_{\mu}(D) \wedge \|\varphi(\overline a)\|_{\mu}(\overline D) \leadsto  \|\varphi(\overline b)\|_{\mu}(\overline D)
\end{equation}
\end{prop}
\begin{proof}
For atomic $\varphi$, the result follows from Proposition~\ref{prop:uniform valuation} and lines~(\ref{prop:diamonds1}) and (\ref{prop:diamonds4}), since for atomic $\varphi$ one has that Definition~\ref{defn:uniform} says that $\|\varphi(\overline{a})\|_{\mu}(D)=\ominus_D \|\varphi(\overline{a})\|$. The conjunction, conditional, and universal quantifier cases follow as in the proof of Proposition~\ref{prop:uniform valuation}. Disjunctions and existentials are nearly as straightforward, keeping in mind the monotonicity of $\ominus_\pi$ (line~(\ref{prop:diamonds1})). Thus we only need to verify the condition on $\Box$. Without loss of generality, suppose that $\varphi(x)$ has only one free variable. By the induction hypothesis, there is an element $e_{\varphi}$ of $\mathcal{A}$ such that for all $a,b$ in $N$ and $D$ in $P(\mathcal{A})$ one has
\begin{equation}
e_{\varphi}: \|a=b\|_{\mu}(D) \wedge \|\varphi(a)\|_{\mu}(D) \leadsto \|\varphi(b)\|_{\mu}(D)
\end{equation}
Since infimums are intersections, we can keep this uniformity in the following:
\begin{equation}
e_{\varphi}: \|a=b\|_{\mu}(D) \wedge \overline{\inf \|\varphi(a)\|_{\mu}}(D) \leadsto \overline{\inf \|\varphi(b)\|_{\mu}}(D)
\end{equation}
Further this uniformity persists while applying the $\ominus_D$-operator:
\begin{equation}
\ominus_D \|a=b\|_{\mu}(D) \wedge \ominus_D \overline{\inf \|\varphi(a)\|_{\mu}}(D) \leq \ominus_D \overline{\inf \|\varphi(b)\|_{\mu}}(D)
\end{equation}
But then this is also a uniform witness to 
\begin{equation}
\|a=b\|_{\mu}(D) \wedge \|\Box \varphi(a)\|_{\mu} (D)\leq \|\Box \varphi(b)\|_{\mu} (D)
\end{equation}
\end{proof}

The following proposition is an indicator of the compatibility of the semantics for the existential and universal quantifiers on the modal structures, given above in Definition~\ref{thetheorem}. For the non-modal structures, see Proposition~\ref{prop:soundness} below.
\begin{prop}
Let $\mathcal{N}$ be a uniform $P(\mathcal{A})$-valued $L$-structure with quantifier~$\mathbb{Q}$. Let $\varphi(\overline{x},y)$ be a modal $L$-formula. Then on the modal $\mathbb{B}$-valued $L$-structure $\mu[\mathcal{N}]$, both $(\forall \; y \; \varphi(\overline{x}, y)) \Leftrightarrow (\neg \exists \; y \; \neg \varphi(\overline{x}, y))$ and $(\exists \; y \; \varphi(\overline{x}, y)) \Leftrightarrow (\neg \forall \; y \; \neg \varphi(\overline{x}, y))$ are valid.
\end{prop}
\begin{proof}
It suffices to prove the first since the second follows by replacing $\varphi$ with its negation in the first and by negating both sides of the biconditional. Further, to ease readability, we consider the special case where the tuple $\overline{x}$ consists just of a single variable $x$. By definition, one has $\|\neg \exists \; y \; \neg \varphi(x, y))\|_\mu(D) = [\ominus_D \bigcup_{y\in N} (\mathbb{Q}(y) \wedge (\|\varphi(x,y)\|_\mu(D) \Rightarrow \bot^{\pi}(D))) ]\Rightarrow \bot^{\pi}(D)$. By $\bot^{\pi}(D)\equiv D$ and~(\ref{prop:diamonds10}), one then has the equivalence, uniform in $x$ and~$D$:
\begin{equation}\label{eqn:correctness1}
\|\neg \exists \; y \; \neg \varphi(x, y))\|_\mu(D) \equiv [\bigcup_{y\in N} (\mathbb{Q}(y) \wedge (\|\varphi(x,y)\|_\mu(D) \Rightarrow D))]\Rightarrow D
\end{equation}
Since $\|\forall \; y \; \varphi(x,y)\|_\mu(D)$ is an element of $\mathbb{B}(N)$, one has another equivalence uniform in $x$ and $D$:
\begin{equation}\label{eqn:correctness2}
\|\forall \; y \; \varphi(x,y)\|_\mu(D) \equiv [ (\bigcap_{y\in N} (\mathbb{Q}(y)\Rightarrow \|\varphi(x,y)\|_\mu(D))) \Rightarrow D] \Rightarrow D
\end{equation}
By~(\ref{help2}), to show $\|\forall \; y \; \varphi(x,y)\|_\mu(D)\leq \|\neg \exists \; y \; \neg \varphi(x, y))\|_\mu(D)$, it suffices to show
\begin{equation}\label{eqn:correctness4}
[\bigcup_{y\in N} (\mathbb{Q}(y) \wedge (\|\varphi(x,y)\|_\mu(D) \Rightarrow D))]\leq  [\bigcap_{y\in N} (\mathbb{Q}(y)\Rightarrow \|\varphi(x,y)\|_\mu(D))] \Rightarrow D
\end{equation}
By Proposition~\ref{prop:onpcas}, choose $f$ such that $fen =  (p_1 e)(n (p_0 e))$. Suppose that $e$ is in the antecedent of~(\ref{eqn:correctness4}); we must show that $fe$ is in the consequent. So suppose that $n$ is in $[\bigcap_{y\in N} (\mathbb{Q}(y)\Rightarrow \|\varphi(x,y)\|_\mu(D))]$; we must show that $fen$ is in $D$. By hypothesis, $p_0 e$ is in $\mathbb{Q}(y_0)$ and $p_1 e$ is in $(\|\varphi(x,y_0)\|_\mu(D) \Rightarrow D)$ for some $y_0$ from $\mathcal{N}$. Then $n(p_0e)$ is in $\|\varphi(x,y_0)\|_\mu(D)$, and so $(p_1 e)(n (p_0 e)) $ is in $D$, which is just to say that $fen$ is in $D$. 

For the converse, since $\|\varphi(x,y)\|_\mu$ is a member of $\mathbb{B}(N\times N)$, one also has the equivalence:
\begin{equation}\label{eqn:correctness3}
\|\forall \; y \; \varphi(x,y)\|_\mu(D) \equiv \bigcap_{y\in N} (\mathbb{Q}(y)\Rightarrow [(\|\varphi(x,y)\|_\mu(D)\Rightarrow D)\Rightarrow D ] )
\end{equation}
By Proposition~\ref{prop:onpcas}, choose $f$ such that $fenm = epnm$. Suppose that $e$ is in the right-hand side of~(\ref{eqn:correctness1}); we show that $fe$ is in the right-hand side of~(\ref{eqn:correctness3}). So suppose that $y\in N$ is fixed and $n$ is in $\mathbb{Q}(y)$ and $m$ is in $(\|\varphi(x,y)\|_\mu(D) \Rightarrow D)$; we must show that $fenm$ is in $D$. Then $pnm$ is in $ (\mathbb{Q}(y) \wedge (\|\varphi(x,y)\|_\mu(D) \Rightarrow D))$, and so by the hypothesis on $e$, we have $epnm$ is in $D$, which is just to say that $fenm$ is in $D$. So we have $\|\neg \exists \; y \; \neg \varphi(x, y))\|_\mu(D)\leq \|\forall \; y \; \varphi(x,y)\|_\mu(D)$.
\end{proof}

The following proposition tells us that validities are equivalent to their necessitations.
\begin{prop}
Let $\mathcal{N}$ be a uniform $P(\mathcal{A})$-valued $L$-structure with quantifier~$\mathbb{Q}$. Suppose that $\varphi(\overline{x})$ be a modal $L$-formula such that $\|\varphi(\overline{a})\|_\mu(D)\equiv \top$ uniformly in $D$ and $\overline{a}$. Then $\|\Box \varphi(\overline{a})\|_\mu(D)\equiv\|\varphi(\overline{a})\|_\mu(D)\equiv \top$ uniformly in $D$ and $\overline{a}$. In particular, suppose that $\varphi$ is a modal $L$-sentence such that $\|\varphi\|_{\mu}(D)\equiv \top$ uniformly in~$D$. Then $\|\Box\varphi\|_{\mu}(D)\equiv \|\varphi\|_{\mu}(D)\equiv \top$ uniformly in $D$.
\end{prop}
\begin{proof}
This is an application of Proposition~\ref{S4}.ii, by setting $g$ equal to the function $\|\varphi(\overline{x})\|_{\mu}: N^n \times P(\mathcal{A})\rightarrow P(\mathcal{A})$ and by setting $h$ equal to $\top$. 
\end{proof}

The next proposition records that the Converse Barcan Formula~$\mathsf{CBF}$~(\ref{eqn:CBF}) is valid and that the schema $(\exists \; x \; \Box \; \varphi(x))\Rightarrow (\Box \; \exists \; x \; \varphi(x))$ is valid. While we use the latter validity less frequently, we do employ it in the proof of Theorem~\ref{thm:zfr-valued} and Proposition~\ref{prop:failsurenegstable} below. For a counterexample to the Barcan formula~(\ref{eqn:BF}), see Proposition~\ref{prop:failureofbarcan} below. It is unknown to us whether the schema $(\forall \; x \; \Diamond \; \varphi(x))\Rightarrow (\Diamond \; \forall \; x \; \varphi(x))$ is valid. 
\begin{prop}\label{prop:CBF} Let $\mathcal{N}$ be a uniform $P(\mathcal{A})$-valued $L$-structure with quantifier~$\mathbb{Q}$. Then Converse Barcan Formula~$\mathsf{CBF}$~(\ref{eqn:CBF}) and the schema $(\exists \; x \; \Box \; \varphi(x))\Rightarrow (\Box \; \exists \; x \; \varphi(x))$ are valid on the modal $\mathbb{B}$-valued $L$-structure $\mu[\mathcal{N}]$.
\end{prop}
\begin{proof}
Let's first argue for the Converse Barcan Formula~$\mathsf{CBF}$~(\ref{eqn:CBF}). Let $\varphi(x)$ be a modal formula, perhaps with parameters, which we suppress for the sake of readability. So we must show $ \| \Box \; \forall \; x \; \varphi(x) \|_{\mu}(D)\leq \| \forall \; x \; \Box \varphi(x) \|_{\mu}(D)$, uniformly in $D$ from $P(\mathcal{A})$. For all $c^{\prime}$ in $N$, by taking compositions we have the following:
\begin{equation}
[\mathbb{Q}(c^{\prime}) \wedge \bigcap_{E\in P(\mathcal{A})} \bigcap_{c\in N} (\mathbb{Q}(c) \Rightarrow \|\varphi(c)\|_{\mu}(E))] \leq \bigcap_{E\in P(\mathcal{A})} \|\varphi(c^{\prime})\|_{\mu}(E)
\end{equation}
Then by~(\ref{prop:diamonds2}) this is $\leq [\ominus_D \bigcap_{E\in P(\mathcal{A})} \|\varphi(c^{\prime})\|_{\mu}(E)]$. Then by moving the $\mathbb{Q}(c^{\prime})$ to the antecedent, and then taking intersections over all $c^{\prime}$ from $N$ and then applying the~$\ominus_D$-operator again, we obtain:
\begin{equation}
\ominus_D  [\bigcap_{E\in P(\mathcal{A})} \bigcap_{c\in N} (\mathbb{Q}(c) \Rightarrow \|\varphi(c)\|_{\mu}(E))]  \leq \ominus_D [\bigcap_{c^{\prime}\in N} (\mathbb{Q}(c^{\prime}) \Rightarrow \ominus_D \bigcap_{E\in P(\mathcal{A})} \|\varphi(c^{\prime})\|_{\mu}(E))]
\end{equation}
The antecedent is $ \| \Box \; \forall \; x \; \varphi(x) \|_{\mu}(D)$ and consequent is $\ominus_D \| \forall \; x \; \Box \varphi(x) \|_{\mu}(D)$, and the latter is equivalent to  $\| \forall \; x \; \Box \varphi(x) \|_{\mu}(D)$ since this is an element of $\mathbb{B}$. (If $\varphi(x)$ had $n$ parameters, then it would be an element of $\mathbb{B}(N^n)$).

Now let's argue for the schema $(\exists \; x \; \Box \; \varphi(x))\Rightarrow (\Box \; \exists \; x \; \varphi(x))$. It suffices to find a witness to the following reduction, uniformly in $D$ and~$c$:
\begin{equation}\label{eqn:helloreduction4}
\mathbb{Q}(c) \wedge \ominus_D \bigcap_{E\in P(\mathcal{A})} \|\varphi(c)\|_{\mu}(E)\leq \ominus_D \bigcap_{E \in P(\mathcal{A})} \bigcup_{x\in N} (\mathbb{Q}(x) \wedge \|\varphi(x)\|_{\mu}(E))
\end{equation}
But by~(\ref{prop:diamonds2}) and~(\ref{prop:diamonds4}), we have $\mathbb{Q}(c) \wedge \ominus_D \bigcap_{E} \|\varphi(c)\|_{\mu}(E)\leq \ominus_D (\mathbb{Q}(c) \wedge \bigcap_{E} \|\varphi(c)\|_{\mu}(E))$, uniformly in $D$ and $c$. Further, uniformly in $c$, the identity function is a witness to the reductions $(\mathbb{Q}(c) \wedge \bigcap_{E} \|\varphi(c)\|_{\mu}(E))\leq \bigcap_{E} (\mathbb{Q}(c) \wedge \|\varphi(c)\|_{\mu}(E))\leq \bigcap_{E} \bigcup_{x\in N} (\mathbb{Q}(x) \wedge \|\varphi(x)\|_{\mu}(E))$, and hence these reductions persist when prefaced by the $\ominus_D$-operator by~(\ref{prop:diamonds1}).
\end{proof}

The following definition describes a series of constraints that one can put on the quantifiers~$\mathbb{Q}$, and the idea of the subsequent theorem and proposition is that these constraints have consequences for what schemas of modal predicate logic are valid on the structure. There is no analogue of this definition in the original Flagg paper. See immediately after the proof of Theorem~\ref{thm:ECTisvalid} for a discussion of how this relates to what is in the original paper.
\begin{definition}\label{defn:Qproperties}
Let $\mathcal{N}$ be a uniform $P(\mathcal{A})$-valued $L$-structure with quantifier~$\mathbb{Q}$. If one has $\bigcup_{c\in N} \mathbb{Q}(c)\neq \emptyset$, then $\mathbb{Q}$ is said to be \emph{non-degenerate}. If $\bigcap_{c\in N} \mathbb{Q}(c)\neq \emptyset$, then $\mathbb{Q}$ is said to be \emph{uniform}. If $\mathbb{Q}(c)\equiv \top$ uniformly in $c\in N$, then the quantifier $\mathbb{Q}$ is said to be \emph{classical}. If for all $L$-terms $t(x_1, \ldots, x_n)$ with all free variables displayed there is $e_t\in \mathcal{A}$ such that $e_t: \mathbb{Q}(a_1)\wedge \cdots \wedge \mathbb{Q}(a_n) \leadsto \mathbb{Q}(t(a_1, \ldots, a_n))$ for all $a_1, \ldots a_n$ from $N$, then the quantifier $\mathbb{Q}$ is said to be \emph{term-friendly}.
\end{definition}

Obviously, all classical quantifiers are term-friendly and uniform, and all uniform quantifiers are non-degenerate.

\begin{thm}\label{prop:S4soundness} (Soundness Theorem for $Q^{\circ}_{eq}.\mathsf{S4}$ and $Q_{eq}.\mathsf{S4}$). Let $\mathcal{N}$ be a uniform $P(\mathcal{A})$-valued $L$-structure with non-degenerate quantifier~$\mathbb{Q}$. Then all the theorems of $Q^{\circ}_{eq}.\mathsf{S4}+\mathsf{CBF}$ are valid on the modal $\mathbb{B}$-valued $L$-structure $\mu[\mathcal{N}]$. Further, if $\mathbb{Q}$ is uniform, then all the theorems of $Q_{eq}.\mathsf{S4}$ are valid on the modal structure.
\end{thm}
\begin{proof}
The propositional part follows from Proposition~\ref{S4}. For the predicate part of $Q^{\circ}_{eq}.\mathsf{S4}+\mathsf{CBF}$, first note that we have $\mathsf{CBF}$ by Proposition~\ref{prop:CBF}. The first axiom for $Q^{\circ}_{eq}.\mathsf{S4}$ in \citet[pp. 133-134]{Fitting1998aa}  is \textsc{Vacuous Quantification}, namely $\forall \; x \; \varphi \equiv \varphi$ when $\varphi$ does not contain $x$ free. But since $\mathbb{Q}$ is assumed to be non-degenerate, choose $n$ in $\mathbb{Q}(d)$ for some $d\in N$. To define a witness $\bigcap_{c\in N}(\mathbb{Q}(c)\Rightarrow \|\varphi\|_{\mu}(D))\leq \|\varphi\|_{\mu}(D)$, let $t(x,y)\equiv yx$ and by Proposition~\ref{prop:onpcas} choose $f$ such that $fne = en$. Then $fn$ is a witness. For the converse direction, $k$ suffices since if $a$ is in $\|\varphi\|_{\mu}(D)$ then $ka$ is in $(\mathbb{Q}(c)\Rightarrow \|\varphi\|_{\mu}(D))$ for all $c\in N$, since if $b\in \mathbb{Q}(c)$ then $kab = a$ is in $\|\varphi\|_{\mu}(D)$.

The second axiom in \citet[pp. 133-134]{Fitting1998aa}  is \textsc{Universal Distributivity}, namely $[\forall \; x \; (\varphi(x)\Rightarrow \psi(x))]\Rightarrow [(\forall \; x \; \varphi(x))\Rightarrow (\forall \; x \; \psi(x))]$. But the validity of this follows straightforwardly by taking compositions. The third axiom is \textsc{Permutation}, which says that $\forall \; x \; \forall \; y \; \varphi(x,y) \Leftrightarrow \forall \; y \; \forall \; x \; \varphi(x,y)$. Let $t(x,y,z)=yxz$, and choose $f$ such that $fenm = t(e,n,m) = emn$; then $f$ performs the desired reduction.

The fourth axiom is \textsc{Universal Instantiation Axiom}~(\ref{eqn:UniversalInstantiationAxiom}) from \S\ref{sec:intro}. So one must show that the following has top value:
\begin{equation}
\bigcap_{c\in N} [\mathbb{Q}(c) \Rightarrow [ (\bigcap_{d\in N} (\mathbb{Q}(d) \Rightarrow \|\varphi(d)\|_{\mu}(D) )) \Rightarrow (\|\varphi(c)\|_{\mu}(D))  ]]
\end{equation}
But given an element $n$ of $\mathbb{Q}(c)$ and an element $e$ of $(\bigcap_{d\in N} (\mathbb{Q}(d) \Rightarrow \|\varphi(d)\|_{\mu}(D) ))$, it of course follows that $en$ is an element of $\|\varphi(c)\|_{\mu}(D)$. Let $t(x,y)\equiv xy$, and choose $f$ from $\mathcal{A}$ such that $fen = en$ by Proposition~\ref{prop:onpcas}. Then $kf$ is an element of $\mathcal{A}$ which sends everything to $f$ in that $kfb=f$ for all $b$ from $\mathcal{A}$. Hence it is a witness to \textsc{Universal Instantiation Axiom} having top value.

The final components of the deductive system of \citet[pp. 133-134]{Fitting1998aa}  are the rules \textsc{Modus Ponens} and \textsc{Universal Generalization}. The former follows by the usual considerations related to composition. The latter is the rule to infer from $\varphi$ to $\forall \; x \; \varphi(x)$. Suppose then that $\|\varphi(c)\|_{\mu}(D)$ has top value in $P(\mathcal{A})$ uniformly in $c$ and $D$ via index $e$, so that $eb$ is in $\|\varphi(c)\|_{\mu}(D)$ for all $b$ from $\mathcal{A}$, uniformly in $c$ and $D$. Then $ke$ is constant function which sends everything to $e$. Then this is a witness to $\bigcap_{c\in N} (\mathbb{Q}(c)\Rightarrow \|\varphi(c)\|_{\mu}(D))$ having top value, in that $keb=e$ is in $(\mathbb{Q}(c)\Rightarrow \|\varphi(c)\|_{\mu}(D))$ for every $b$ from $\mathcal{A}$.

Now we verify~(\ref{eqn:necid})-(\ref{eqn:subs}). Obviously~(\ref{eqn:necid}) follows from $e_{\mathrm{ref2}}$ in Definition~\ref{defn:uniform} and from (\ref{prop:diamonds1}). As for~(\ref{eqn:subs}), this follows directly from Proposition~\ref{prop:subsnecwow}.

Finally, suppose that  $\mathbb{Q}$ is uniform. It suffices to show that the free-variable variant $(\forall \; x \; \varphi(x)) \Rightarrow \varphi(y)$ of the \textsc{Universal Instantiation Axiom} has top value. We must show that we have $\bigcap_{c\in N} (\mathbb{Q}(c)\Rightarrow \|\varphi(c)\|_{\mu}(D))\leq \|\varphi(d)\|_{\mu}(D)$ uniformly in $d,D$. So suppose that $e$ is in the antecedent. Since $\mathbb{Q}$ is uniform, choose an element of $n$ of $\bigcap_{c\in N} \mathbb{Q}(c)$. Then for each $d$ and $D$, one has that $en$ is an element $\|\varphi(d)\|_{\mu}(D)$. 
\end{proof}

Sometimes in what follows (cf. Proposition~\ref{prop:heytingaxioms}, Proposition~\ref{prop:newHA}, and the proof Theorem~\ref{thm:zfr-valued} in Appendix~\ref{sec:appendixmccarty}), we will need to apply a similar soundness theorem for the uniform $P(\mathcal{A})$-valued structures themselves. In the following proposition, the intuitionistic predicate calculus $\mathsf{IQC}$ with equality is given by the intuitionstic propositional calculus $\mathsf{IPC}$ formulated in a natural deduction system, together with the usual natural deduction rules for quantifiers, as well as the axioms~(\ref{eqn:necid})-(\ref{eqn:subs}) for identity (cf. \cite[volume 1 p. 48]{Troelstra1988aa}).
\begin{prop}\label{prop:soundness}
Suppose that $\mathcal{N}$ is a $P(\mathcal{A})$-valued $L$-structure with term-friendly non-degenerate quantifier~$\mathbb{Q}$. 

(I)~Suppose that $\varphi_1(\overline{x})$, \ldots, $\varphi_n(\overline{x}), \psi(\overline{x})$ are $L$-formulas, with all free variables amongst those displayed. Suppose that $\mathsf{IQC}, \varphi_1(\overline{x}), \ldots, \varphi_n(\overline{x})\vdash \psi(\overline{x})$. Then the $L$-sentence $\forall \; \overline{x} \; ((\bigwedge_{i=1}^n \varphi_i(\overline{x}))\Rightarrow \psi(\overline{x}))$ is valid on $\mathcal{N}$. 

(II)~Hence, if $\varphi_1, \ldots, \varphi_n$, $\psi$ are $L$-sentences and $\mathsf{IQC}, \varphi_1, \ldots, \varphi_n \vdash \psi$, and if $\varphi_1, \ldots, \varphi_n$ are valid on $\mathcal{N}$, then so too is $\psi$.
\end{prop}
\begin{proof}
The proof of~(I) is by induction on the length of the derivation. For the base case, the identity axiom~(\ref{eqn:necid}) follows from the clause pertaining to $e_{\mathrm{ref}2}$ in Definition~\ref{defn:uniform}, while the substitution axiom~(\ref{eqn:subs}) follows from Proposition~\ref{prop:uniform valuation}. The other base case is where $\psi(\overline{x})$ is one of the $\varphi_i(\overline{x})$, and in this case an appropriate projection function will witness the validity. The projection functions also allow one to expand the antecedent as needed. 

For the inductive steps, one considers first the propositional rules and then the quantifier rules. But the propositional rules follow from the observation made at the outset of \S\ref{sec:heyting} that Heyting prealgebras are sound for the intuitionstic propositional calculus $\mathsf{IPC}$. For the quantifier rules, it will be convenient to abbreviate the antecedent as $\Phi(\overline{x})\equiv \bigwedge_{i=1}^n \varphi_i(\overline{x})$ and to drop excess free variables to enhance readability.

For the ``for all'' elimination rule, we must show that if $\forall \; y \; (\Phi(y)\Rightarrow \forall \; x \; \psi(x))$ is valid on $\mathcal{N}$ then so is $\forall \; y \; \forall \; x \; (\Phi(y)\Rightarrow \psi(t(x)))$, where $t$ is an $L$-term. Since $\mathbb{Q}$ is term-friendly, choose $e_t$ such that $e_t:\mathbb{Q}(a)\leadsto \mathbb{Q}(t(a))$ for all $a$ from $\mathcal{N}$. By Proposition~\ref{prop:onpcas}, choose $f$ such that $fenmu = enu (e_tm)$. Supposing that $e$ is in $\mathbb{Q}(y)\Rightarrow (\|\Phi(y)\|\Rightarrow \bigcap_{x\in N} (\mathbb{Q}(x)\Rightarrow \|\psi(x)\|))$ for all $y$ from $\mathcal{N}$, we must show that $fe$ is in $\mathbb{Q}(y)\Rightarrow \bigcap_{x\in N} (\mathbb{Q}(x)\Rightarrow \|\Phi(y)\Rightarrow \psi(t(x))\|)$ for all $y$ from $\mathcal{N}$. So suppose that $y$ from $\mathcal{N}$ is fixed and $n$ is in $\mathbb{Q}(y)$. Then we must show that $fn$ is in $(\mathbb{Q}(x)\Rightarrow \|\Phi(y)\Rightarrow \psi(t(x))\|)$ for all $x$ in $\mathcal{N}$. So let $x$ in $\mathcal{N}$ and suppose that $m$ is in $\mathbb{Q}(x)$ and $u$ is in $\|\Phi(y)\|$. We must show that $fenmu$ is in $\|\psi(t(x))\|$. By choice of~$f$, this is the same as showing that $enu (e_tm)$ is in $\|\psi(t(x))\|$. But by hypothesis on $e,n,u$, we have that $enu$ is in $\mathbb{Q}(z)\Rightarrow \|\psi(z)\|$ for all $z$ from $\mathcal{N}$, and by hypothesis on $m$ and $e_t$, we have that $e_tm$ is in $\mathbb{Q}(t(x))$, so that we are done.

For the ``for all'' introduction rule, suppose our induction hypothesis gives us that $\forall \; y \; \forall \; x \; (\Phi(y)\Rightarrow \psi(x))$ is valid on $\mathcal{N}$; then we must show that $\forall \; y \; (\Phi(y)\Rightarrow \forall \; x \; \psi(x))$ is valid on $\mathcal{N}$. By Proposition~\ref{prop:onpcas}, choose $f$ such that $fenmu = enum$. Supposing that $e$ is in $\mathbb{Q}(y)\Rightarrow (\bigcap_{x\in N} (\mathbb{Q}(x) \Rightarrow (\|\Phi(y) \Rightarrow \psi(x)\|)))$ for all $y$ in $\mathcal{N}$, we show that $fe$ is in $\mathbb{Q}(y)\Rightarrow (\|\Phi(y)\| \Rightarrow \bigcap_{x\in N} (\mathbb{Q}(x) \Rightarrow \|\psi(x)\|))$ for all $y$ in $\mathcal{N}$. So fix $y$ in $\mathcal{N}$ and suppose that $n$ is in $\mathbb{Q}(y)$. We must show that $fen$ is in $(\|\Phi(y)\| \Rightarrow \bigcap_{x\in N} (\mathbb{Q}(x) \Rightarrow \|\psi(x)\|))$. So suppose that $m$ in $\|\Phi(y)\|$; we must show that $fenm$ is in $\mathbb{Q}(x) \Rightarrow \|\psi(x)\|$ for all $x$ in $\mathcal{N}$. So fix $x$ in $\mathcal{N}$, and suppose that $u$ is in $\mathbb{Q}(x)$; then we must show that $fenmu$ is in $\|\psi(x)\|$. But by choice of $f$, this is the same as showing that $enum$ is in $\|\psi(x)\|$. By hypothesis on $e$ and $n$ and $u$ we have $enu$ is in  $(\|\Phi(y) \Rightarrow \psi(x)\|)$, and by hypothesis on $m$, we have that $enum$ is in $\|\psi(x)\|$, which is what we wanted to show. 

For the ``there exists'' elimination rule, we must show that if both $\forall \; y \; (\Phi(y) \Rightarrow \exists \; x \; \psi(x))$ and $\forall \; y \; \forall \; z \; (\Phi(y) \wedge \psi(z) \Rightarrow \xi(y))$ are valid in $\mathcal{N}$, then $\forall \; y \; (\Phi(y) \Rightarrow \xi(y))$ is valid in $\mathcal{N}$. But supposing that $e_1$ and $e_2$ are witnesses to the former one may check that $fe_1e_2$ is a witness to the latter, where one chooses $f$ from Proposition~\ref{prop:onpcas} such that $fe_1e_2nu= e_2n (p_0e_1nu)(p u (p_1 e_1nu))$. One may do this by beginning with the antecedent of~(\ref{eqn:diagramchase3}), then moving to~(\ref{eqn:diagramchase1}) reading left-to-right, and then moving to to~(\ref{eqn:diagramchase2}) reading left-to-right, and then finishing at the consequent of~(\ref{eqn:diagramchase3}), where the idea is that the witnesses in the pca are written out below the parts of the formula which they are realizing. 
\begin{align}
& \forall \; \underset{n}{y} \; (\underset{u}{\Phi(y)} \Rightarrow \exists \; \underset{p_0e_1nu}{x} \; \underset{p_1 e_1nu}{\psi(x)}) \label{eqn:diagramchase1} \\
& \forall \; \underset{n}{y} \; \forall \; \underset{p_0e_1nu}{z} \; (\underset{u}{\Phi(y)} \wedge \underset{p_1 e_1nu}{\psi(z)} \Rightarrow \underset{fe_1e_2nu}{\xi(y)}) \label{eqn:diagramchase2} \\
& \forall \; \underset{n}{y} \; (\underset{u}{\Phi(y)} \Rightarrow \underset{fe_1e_2nu}{\xi(y)}) \label{eqn:diagramchase3}
\end{align}
This ``diagram chase'' method of verification is sometimes a helpful counterpoint to the types of verification exemplified in the previous two paragraphs.

For the ``there exists'' introduction rule, we must show that if $\forall \; y \; \forall \; x \;  (\Phi(y) \Rightarrow \psi(t(x)))$ is valid in $\mathcal{N}$ for some $L$-term $t(x)$, then so is $\forall \; y \; (\Phi(y) \Rightarrow \exists \; x \; \psi(x))$. Since $\mathbb{Q}$ is term-friendly, choose $e_t$ such that $e_t:\mathbb{Q}(a)\leadsto \mathbb{Q}(t(a))$ for all $a$ from $\mathcal{N}$. Since $\mathbb{Q}$ is non-degenerate, choose $m_0$ with $m_0$ in $\mathbb{Q}(x_0)$ for some $x_0$ from $\mathcal{N}$. By Proposition~\ref{prop:onpcas}, choose $f$ such that $fnu = p(e_t(m_0))(enm_0u)$. Supposing that $e$ is in $\mathbb{Q}(y)\Rightarrow (\bigcap_{x \in N} (\mathbb{Q}(x) \Rightarrow (\|\Phi(y)\Rightarrow \psi(t(x))\|)))$ for all $y$ from $\mathcal{N}$, we then show that $fe$ is in $\mathbb{Q}(y)\Rightarrow (\|\Phi(y)\| \Rightarrow (\bigcup_{x\in N} \mathbb{Q}(x) \wedge \|\psi(x)\|))$ for all $y\in N$. So suppose that $y$ from $\mathcal{N}$ is fixed and $n$ is in $\mathbb{Q}(y)$ and $u$ is in $\|\Phi(y)\|$. It suffices to show that $fnu$ is in $\mathbb{Q}(t(x_0)) \wedge \|\psi(t(x_0))\|$, which by definition of $f$ is to show that $e_t(m_0)$ is in $\mathbb{Q}(t(x_0))$ and $enm_0u$ is in $\|\psi(t(x_0))\|$. But both of these follow directly from our hypotheses on $e_t, m_0, x_0, n, u$.

This finishes the proof of part~(I). For part~(II), suppose $\varphi_1, \ldots, \varphi_n$, $\psi$ are $L$-sentences and $\mathsf{IQC}, \varphi_1, \ldots, \varphi_n \vdash \psi$, and that $\varphi_1, \ldots, \varphi_n$ are valid on $\mathcal{N}$. Let $x$ be a variable, which we may assume does not appear in any of these sentences. Then the $L$-sentence $\forall \; x \; \Phi$ is also valid on $\mathcal{N}$, which by part~(I) implies that $\forall \; x \; \psi$ is valid on $\mathcal{N}$. But then one may argue just as in the \textsc{Vacuous Quantification} part of the previous proposition that $\psi$ is also valid on $\mathcal{N}$, since we're assuming that the quantifier $\mathbb{Q}$ is non-degenerate.
\end{proof}



Let's now explicitly record the simplifying effect of the classical quantifiers on the semantics:

\begin{prop}\label{prop:BF}  Let $\mathcal{N}$ be a uniform $P(\mathcal{A})$-valued $L$-structure with classical quantifier~$\mathbb{Q}$. Then the quantifier clauses in the semantics have the equivalents $\|\exists \; z \; \varphi(\overline{a},z)\|\equiv \bigcup_{c\in N} \|\varphi(\overline{a}, c)\|$ and $\|\forall \; z \; \varphi(\overline{a},z)\|\equiv \bigcap_{c\in N} \|\varphi(\overline{a}, c)\|$. Similarly, the quantifier clauses in the semantics for the modal $\mathbb{B}$-valued $L$-structure $\mu[\mathcal{N}]$ have equivalents $\| \exists z \; \varphi(\overline{a}, z)\|_{\mu}(D)\equiv  \ominus_D \bigcup_{c\in N}\|\varphi(\overline{a},c)\|_{\mu}(D)$ and $\| \forall z \; \varphi(\overline{a}, z)\|_{\mu}(D)\equiv  \bigcap_{c\in N} \|\varphi(\overline{a},c)\|_{\mu}(D)$.
\end{prop}
\begin{proof}
A classical quantifier gives $\mathbb{Q}(c)$ top value uniformly, so these simplifications follow from the behavior of top in Heyting and Boolean prealgebras.
\end{proof}

We close this section by noting the stability of atomic formulas under the semantics, together with some helpful characterizations pertaining to the case of negated atomics, which we will use in Proposition~\ref{prop:counterstabilityatomics} to produce some counterexamples to the stability of negated atomics.

\begin{prop} \label{prop:atomics}
Let $\mathcal{N}$ be a uniform $P(\mathcal{A})$-valued $L$-structure with quantifier~$\mathbb{Q}$. Let $R(\overline x)$ be an atomic formula. Then 

\vspace{2mm}

\noindent (i) \emph{Atomic Formulas are Stable}: $\| R (\overline x)\|_{\mu}  \equiv \|\Box R(\overline x)\|_{\mu}$.

\vspace{2mm}

\noindent (ii) \emph{Formula for Negated Atomics}: One has $\|\neg R(\overline{a})\|_{\mu} (D) \equiv \|R(\overline{a})\|\Rightarrow D$ uniformly for all $\overline{a}$ in $N$ and $D$ in $P(\mathcal{A})$. 

\vspace{2mm}

\noindent (iii) \emph{Formula for Necessitations of Negated Atomics}: Suppose that $N_0$ is a subclass of $N$, and for all $\overline{a}$ in $N_0$ one has that $\|R(\overline{a})\|\neq \bot$. Then uniformly for all $\overline{a}$ in $N_0$ and $D$ in $P(\mathcal{A})$ one has $\|\Box \neg R(\overline{a})\|_{\mu} (D) \equiv D$.
\end{prop}
\begin{proof}
For (i) this follows from Proposition~\ref{S4}.ii and the fact that for atomic $\varphi(\overline{x})$, we have that $\|\varphi(\overline{a})\|_{\mu}(D) = \mu(\|\varphi(\overline{a})\|)$ (cf. Definition~\ref{thetheorem}). For (ii), we simply chase out definitions and appeal to~(\ref{prop:diamonds10}) to obtain
$\|\neg R(\overline{a})\|_{\mu} (D) = [\ominus_D \|R(\overline{a})\| \Rightarrow \ominus_D \bot] \equiv [\ominus_D \|R(\overline{a})\| \Rightarrow D] \equiv [\|R(\overline{a})\| \Rightarrow D]$. For the (iii), suppose that $\|R(\overline{a})\|\neq \bot$ for all $\overline{a}$ from $N_0$. For $E=\bot$, we then have that $\ominus_E \|R(\overline{a})\|=\top$ and so 
$\|\neg R(\overline{a})\|_{\mu} (E) \equiv [\ominus_E \|R(\overline{a})\| \Rightarrow \ominus_E \bot] \equiv \top \Rightarrow \bot \equiv \bot$. Hence one then has that $\|\Box \neg R(\overline{a})\|_{\mu} (D) \equiv \ominus_D \bigcap_E \|\neg R(\overline{a})\|_{\mu} (E) =\ominus_D \bot \equiv D$.
\end{proof}


\section{The G\"odel Translation and Flagg's Change of Basis Theorem}\label{sec:GTplusCB}

The aim of this section is to show that an important theorem from Flagg's original paper, namely \cite[Theorem 5.4 p. 168]{Flagg1985aa}, generalizes to semantics from the previous section. In our view, it's expedient to separate this theorem into two parts, the first of which pertains to the G\"odel translation (cf. \citet[p. 288]{Troelstra2000aa}, \citet[p. 147]{Flagg1985aa}):

\begin{definition}\label{GodelTrans} (G\"odel Translation) 
For every non-modal $L$-formula $\varphi$, we define its G\"odel translation $\varphi^{\Box}$ to be the following modal $L$-formula in the same free variables:
\vspace{-8mm}
\begin{multicols}{2}
\begin{align}
& \varphi^{\Box} =  \varphi \mbox{ if $\varphi$ atomic}\notag \\
& (\varphi \wedge \psi)^{\Box} = \varphi^{\Box} \wedge \psi^{\Box} \notag\\
& (\varphi \vee \psi)^{\Box} = \varphi^{\Box} \vee \psi^{\Box} \notag
\end{align}
\columnbreak
\begin{align}
& \mbox{\;} \notag \\
& (\varphi \Rightarrow \psi)^{\Box} = \Box (\varphi^{\Box} \Rightarrow \psi^{\Box})\notag \\
& (\exists x \; \varphi)^{\Box} = \exists x \; \varphi^{\Box} \notag\\
& (\forall  x \; \varphi)^{\Box} = \Box (\forall x \; \varphi^{\Box})\notag
\end{align}
\end{multicols}
\end{definition}
\vspace{-8mm}

Note that if $R(x,y)$ is a binary atomic, then by definition we will have
$(\exists \; x \; (R(x,y) \wedge \varphi(x)))^{\Box}  = \exists \; x \; (R(x,y) \wedge \varphi^{\Box}(x))$, and by appealing to the equivalence $\Box \; \forall \; x \; \Box \; \varphi \equiv \; \Box \; \forall x \; \varphi$ in $Q^{\circ}_{eq}.\mathsf{S4}+\mathsf{CBF}$ or $Q_{eq}.\mathsf{S4}$, we may obtain 
\begin{equation}\label{eqn:remarkgodelatomics}
(\forall \; x \; (R(x,y) \Rightarrow \varphi(x)))^{\Box} = \Box (\forall \; x \; \Box (R(x,y) \Rightarrow \varphi^{\Box}(x))) \equiv  \; \Box \; \forall \; x \; (R(x,y) \Rightarrow \varphi^{\Box}(x))
\end{equation}
In the setting of set theory in \S\ref{sec:setheory}, the traditional application will be to the case in which the binary relation $R$ is just membership $\in$. Using the standard shorthand $\exists \; x \in y \;\varphi(x)$  for $\exists \; x \; (x \in y \wedge \varphi (x))$ and $\forall \; x\in y\; \varphi (x)$  for $\forall \; x\; (x \in y \Rightarrow \varphi(x))$, we see that  existential $\Delta_0$-formulas are treated compositionally by the G\"odel translation and universal $\Delta_0$-formulas are treated likewise but with an initial box-operator placed in front:
\begin{align}
(\exists \; x\in y \; \varphi(x))^{\Box} & \equiv \exists \; x \in y \;  \varphi^{\Box}(x)\label{eqn:remarkgodelatomics1} \\
(\forall \; x\in y \; \varphi(x))^{\Box} & \equiv \; \Box \; \forall \; x\in y \; \varphi^{\Box}(x)\label{eqn:remarkgodelatomics2}
\end{align}
 Similar remarks apply to $\Delta_0$-formulas in the setting of arithmetic of \S\ref{sec:arithmetic}, wherein the binary relation is just the less-than relation $<$. A related observation that we shall apply often in what follows is that blocks of quantifiers are treated compositionally by the G\"odel translation, modulo one box operator being placed in front of a block of universal quantifiers (again appealing to $\Box \; \forall \; x \; \Box \; \varphi \equiv \; \Box \; \forall x \; \varphi$):
\begin{align}
(\exists \; \overline{x} \; \varphi(\overline{x}))^{\Box} & \equiv \exists \; \overline{x} \; \varphi^{\Box}(\overline{x}) \\
(\forall \; \overline{x} \; \varphi(\overline{x}))^{\Box} & \equiv \Box \; \forall \; \overline{x} \; \varphi^{\Box}(\overline{x})
\end{align}

Flagg's result on the G\"odel translation was that, in the setting of arithmetic, a validity on the non-modal structure had a valid G\"odel translation on the modal structure. The below theorem indicates that the same relationship obtains generally in the semantics described in the previous section.

\begin{thm}\label{translationthingie}
Let $\mathcal{N}$ be a uniform $P(\mathcal{A})$-valued $L$-structure with quantifier~$\mathbb{Q}$. Then for every non-modal $L$-formula $\varphi(\overline{x})$, one has $\mu(\| \varphi(\overline{x})\|) = \|\varphi^{\Box}(\overline{x})\|_{\mu}$. Hence, for every non-modal $L$-sentence $\varphi$, one has that $\varphi$ is valid in the uniform $P(\mathcal{A})$-valued structure $\mathcal N$ iff $\varphi^{\Box}$ is valid on the modal $\mathbb{B}$-valued $L$-structure $\mu[\mathcal N]$.
\end{thm}
\begin{proof}
The proof is by induction on the complexity of $\varphi(\overline{x})$. For base cases  $t(\overline{x})=s(\overline{x})$, $R(\overline x)$, and $\bot$, the result follows immediately from the fact that the G\"odel translation is the identity in these cases. For $\varphi(\overline{x})$ a conjunction, disjunction, or conditional, the result follows from the definitions in the semantics as well as the results in lines~(\ref{j2}), (\ref{j2.5}), and~(\ref{corthebobm}) about how $\mu$ acts on conjunctions, disjunctions, and conditionals.  

For the universal quantifier, for the sake of readability consider the case of $\varphi(x)\equiv \forall \; y \; \psi(x,y)$. By the induction hypothesis $\mu(\|\psi(x,y)\|) \equiv \|\psi^{\Box}(x,y)\|_{\mu}$. Then we must show that  
\begin{equation}
\ominus_D (\bigcap_{b\in N} (\mathbb{Q}(b) \Rightarrow \|\psi(a,b)\|) \equiv \ominus_D \bigcap_{E\in P(\mathcal{A})} \bigcap_{b\in N} (\mathbb{Q}(b)\Rightarrow \ominus_E \|\psi(a,b)\| )
\end{equation}
For the left-to-right reduction, note that we have $(\mathbb{Q}(b) \Rightarrow \|\psi(a,b)\|)\leq (\mathbb{Q}(b)\Rightarrow \ominus_E \|\psi(a,b)\|)$ by line~(\ref{prop:diamonds2}) uniformly in $E$ and $a,b$. Then this persists when taking intersections over $b\in N$ and $E$ from $P(\mathcal{A})$ and by adding the $\ominus_D$ operator to both sides by line~(\ref{prop:diamonds1}). For the converse, first note that by line~(\ref{prop:diamonds12}) we have that $\bigcap_{E\in P(\mathcal{A})} (\bigcap_{b\in N} (\mathbb{Q}(b) \Rightarrow \ominus_E \|\psi(a,b)\|)) \leq \bigcap_{b\in N} (\mathbb{Q}(b) \Rightarrow \|\psi(a,b)\|)$. Hence by applying $\ominus_D$ to both sides, this becomes the desired converse.

For the existential quantifier, again consider the case of $\varphi(x)\equiv \exists \; y \; \psi(x,y)$. Then we evaluate $\mu(\|\exists \; y \; \psi(x,y)\|)(a,D)$ as follows, applying Proposition~\ref{mylittlehelper} in conjunction with the induction hypothesis to obtain that $\ominus_D \bigcup_{b\in N} (\mathbb{Q}(b) \wedge \|\psi(a,b)\|) \equiv \ominus_D \bigcup_{b\in N} (\mathbb{Q}(b) \wedge \ominus_D \|\psi(a,b)\|) \equiv  \ominus_D \bigcup_{b\in N} (\mathbb{Q}(b) \wedge \|\psi^{\Box}(a,b)\|_{\mu}(D))$, which is just $\|\exists \; y \; \psi^{\Box}(a,y)\|_{\mu}(D)$.
\end{proof}

The following theorem provides a way of expanding a structure $\mathcal N$ by adding a new predicate symbol to represent the necessitation of any given formula. This will be useful for many of the proofs in the subsequent sections, since it  implies that if any expansion of $\mathcal{N}$ validates a schema $J(\varphi)$, then the modal structure $\mu[\mathcal{N}]$ validates the schema $J^{\Box}(\Box \varphi)$. 

\begin{thm}\label{prop:changeofbasis}
(Change of Basis Theorem). Let  $\mathcal{N}$ be a uniform $P(\mathcal{A})$-valued $L$-structure with quantifier~$\mathbb{Q}$. Let $\varphi(\overline{x})$ be an $n$-ary modal $L$-formula with all free variables displayed and let $G(\overline{x})$ be a new $n$-ary predicate. Further, define $\mathbf{G}:N^n\rightarrow P(\mathcal{A})$ by $\mathbf{G}(\overline{a}) = \bigcap_{E\in P(\mathcal{A})} \|\varphi(\overline{a})\|_{\mu}(E)$ and let $\mathcal{N}(\mathbf{G})$ be the expansion of $\mathcal{N}$ to an $L\cup \{G\}$-structure by interpreting $G$ by $\mathbf{G}$. Then (i)~$\mathcal{N}(\mathbf{G})$ is a uniform $P(\mathcal{A})$-valued $L\cup \{G\}$-structure with quantifier $\mathbb{Q}$, and (ii)~the valuation of the atomic formula $G(\overline{x})$ on the modal $\mathbb{B}$-valued $L\cup \{G\}$-structure $\mu[N(\mathbf{G})]$ is the same as the valuation of the modal formula $\Box \varphi(\overline{x})$ on the modal $\mathbb{B}$-valued $L$-structure $\mu[\mathcal{N}]$.
\end{thm}
\begin{proof}
For (i), we need to ensure that there is an element $e$  of $\mathcal{A}$ such that for all $a_1, \ldots, a_n, b_1, \ldots, b_n$ from $N$, the element $e$ is a witness to the following reduction: $\|a_1=b_1\| \wedge \cdots \wedge \|a_n=b_n\| \wedge \| G(a_1, \ldots, a_n)\| \leq \|G(b_1, \ldots, b_n)\|$. Let $e$ be an element of $\mathcal{A}$ which is a witness to the following, uniformly in $D$ from Proposition~\ref{prop:subsnecwow}: $\|a_1=b_1\|_{\mu}(D) \wedge \cdots \wedge \|a_n=b_n\|_{\mu}(D) \wedge \|\varphi(a_1, \ldots, a_n)\|_{\mu}(D)\leq \|\varphi(b_1, \ldots, b_n)\|_{\mu}(D)$.
Then we have the following:
\begin{align}
& \|a_1=b_1\| \wedge \cdots \wedge \|a_n=b_n\| \wedge \|G(a_1, \ldots, a_n)\| \label{eqn:bigone1}\\ 
\leq & \|a_1=b_1\| \wedge \cdots \wedge \|a_n=b_n\| \wedge \bigcap_{E\in P(\mathcal{A})} \|\varphi(a_1, \ldots, a_n)\|_{\mu}(E) \label{eqn:bigone2}\\ 
 \leq & \ominus_D \|a_1=b_1\| \wedge \cdots \wedge \ominus_D \|a_n=b_n\| \wedge  \|\varphi(a_1, \ldots, a_n)\|_{\mu}(D) \label{eqn:bigone3}\\
 \leq & \|a_1=b_1\|_{\mu}(D) \wedge\cdots \wedge \|a_n=b_n\|_{\mu}(D) \wedge  \|\varphi(a_1, \ldots, a_n)\|_{\mu}(D) \label{eqn:bigone4}\\ 
 \leq & \|\varphi(b_1, \ldots, b_n)\|_{\mu}(D)\label{eqn:bigone5}
\end{align}
In this, line~(\ref{eqn:bigone2}) follows from interpreting $G$ by $\mathbf{G}$.  Further, line~(\ref{eqn:bigone3}) follows on its first $n$-components from line~(\ref{prop:diamonds2}) and on the last component by the identity function (since we're dealing with an intersection). Finally, line~(\ref{eqn:bigone4}) follows from the semantics for identity in the modal structure, while line~(\ref{eqn:bigone5}) follows from the property of element~$e$. The reduction from~(\ref{eqn:bigone1}) to (\ref{eqn:bigone5}) then suffices by taking intersections over all $D$ from $P(\mathcal{A})$ (since $\mathbf{G}$ is defined as an intersection). So this completes the verification that $\mathcal{N}(\mathbf{G})$ is a uniform $P(\mathcal{A})$-valued $L\cup \{G\}$-structure. 

For part~(ii) of the proposition, simply note the following, where for the sake of disambiguation we superscript all the valuations with names for their structures:
\begin{equation}
\|G(\overline{a})\|_{\mu}^{\mu[\mathcal{N}(\mathbf{G})]} (D) = \ominus_D \|G(\overline{a})\|^{\mathcal{N}(\mathbf{G})} = \ominus_D \bigcap_{E\in P(\mathcal{A})} \|\varphi(\overline{a})\|_{\mu}^{\mu[\mathcal{N}]}(E) = \|\Box \varphi(\overline{a})\|_{\mu}^{\mu[\mathcal{N}]} 
\end{equation}
In this equation, the first equality follows from the interpretation of atomics in the modal structures, the second follows from the definition of $\mathbf{G}$ which serves as the interpretation of $G$, and the last comes from the definition of the box (cf. Definition~\ref{inf(f)}).
\end{proof}



 \section{Epistemic Arithmetic and Epistemic Church's Thesis}\label{sec:arithmetic}

Let $f_1, f_2, \ldots$ be a standard enumeration of the primitive recursive functions, and let $L_{0}$ be the signature $\{0,S, f_1, f_2,\ldots\}$. This is the signature of Heyting arithmetic $\mathsf{HA}$ (cf. \citet[volume 1 p. 126]{Troelstra1988aa}). In this section we'll work exclusively with Kleene's first model $\mathcal{K}_1$ (cf. \S\ref{sec:heytingfrompca}), and the following structure shall be the focus of our study:
\begin{definition}
Let $\mathcal{N}_{0}$ be the uniform $P(\mathcal{K}_1)$-valued $L_{0}$-structure with domain $N=\omega$ and quantifier $\mathbb{Q}(n)=\{n\}$, wherein $0$ and the primitive recursive functions are interpreted as themselves, $S$ is interpreted as successor, and  equality is interpreted disjunctively: 
\begin{equation}\label{eqn:myhowidentity}
\| n=m \| =
\begin{cases}
\top      & \text{if $n=m$}, \\
\bot      & \text{if $n\neq m$}.
\end{cases}
\end{equation}
\end{definition}
\noindent To see that this is indeed a uniform $P(\mathcal{K}_1)$-valued structure, note that $e_{\mathrm{ref1}}, e_{\mathrm{sym}}, e_{\mathrm{tran}}$ can simply be taken to be indexes for the identity function, while $e_{\mathrm{ref2}}$ can be taken to be an index for the function which sends everything to zero. Note that in the terminology of Definition~\ref{defn:Qproperties}, the quantifier $\mathbb{Q}(n)=\{n\}$ is non-degenerate, non-uniform, non-classical, and term-friendly. It's non-degenerate because the union of the $\mathbb{Q}(n)$ is non-empty, and by the same token it's non-uniform because the intersection of the $\mathbb{Q}(n)$ is empty. To see that it is non-classical, suppose there were an index $e$ witnessing $e:\omega\leadsto \{n\}$ for each $n$. Then, e.g., $e0$ would be an element of both $\{0\}$ and $\{1\}$. Finally, this quantifier is term-friendly because all the terms $n\mapsto t(n)$ in the language determine a recursive function with index $e_t$, which witnesses $e_t: \{n\}\leadsto \{t(n)\}$. 

Let $L$ be an expansion of $L_{0}$ by any number of new relation or function symbols. Then an expansion $\mathcal{N}$ of $\mathcal{N}_0$ is given by specifying maps $\mathbf{G}: \omega^n\rightarrow P(\mathcal{K}_1)$ to provide the interpretation of $\mathbf{G}=\|G(\overline{x})\|$ for each new relation symbol, along with interpretations of the new function symbols. It's easy to see that \emph{any} such choice will produce a uniform $P(\mathcal{K}_1)$-valued $L$-structure. For ease of readability, suppose that a new predicate $G(x)$ is unary. We must show that there is a uniform witness to the reduction $\|m = n \| \wedge \|G(m)\| \leq \|G(n)\|$ for all $n,m\geq 0$. If $\|m = n \|$ is empty, then any index will be a witness to the reduction. However, if it's not empty, then $m=n$ and the sets $\|G(m)\|$ and $\|G(n)\|$ are equal, so the second projection function is a witness to the reduction. A similar argument works for the new atomics produced by new function symbols. However, note that if one expands the structure by symbols for non-recursive functions, then one will no longer have witnesses for the quantifiers being term-friendly. Hence, in this section, we work with expansions of the signature $L_0$ to signatures $L$ by new constant, relation, and function symbols, and we restrict attention to $L$-structures  $\mathcal{N}$ which are expansions of the $L_0$-structure $\mathcal{N}_0$ where the new functions are interpreted by \emph{recursive} functions.

The first result is that theorems of Heyting arithmetic have top value on these structures. If $L$ is an expansion of $L_0$, then of course $\mathsf{HA}$ in that signature simply contains, in addition, the instances of the induction schema in that signature.
\begin{prop}\label{prop:heytingaxioms}
 If $\varphi$ is an $L$-sentence such that $\mathsf{HA}\vdash \varphi$, then $\varphi$ is valid on $\mathcal{N}$.
\end{prop}
\begin{proof}
By Proposition~\ref{prop:soundness}, it suffices to ensure that the axioms of $\mathsf{HA}$ in the expanded signature are valid on the structure. For the axiom $S0\neq 0$, note that $\|S0=0\|$ is empty and so any index is a witness to $\|S0{=}0  \Rightarrow  \bot\|$. The identity function can again be used to verify the universal closures of any of the defining equations for the primitive recursive functions. Further, for induction, it suffices to find a witness to the following:
\begin{equation}\label{eqn:inductioninthenats}
\| \varphi(0) \| \wedge \bigcap_{n\geq 0} (\mathbb{Q}(n)\Rightarrow (\|\varphi(n)\|\Rightarrow \|\varphi(Sn)\|))\leq \bigcap_{n\geq 0} (\mathbb{Q}(n) \Rightarrow \|\varphi(n)\|)
\end{equation}
So let $p e_0 e_1$ be in the antecedent. Simply choose a recursive function $j$ such that $j(e_0, e_1) 0=e_0$ and $j(e_0, e_1)(n+1)= (e_1 n)( j(e_0, e_1) (n))$. Then one can verify by induction on $n\geq 0$ that $j(e_0, e_1)$ is in $\mathbb{Q}(n) \Rightarrow \|\varphi(n)\|$. Since $\mathbb{Q}(n)=\{n\}$, this is just to verify that for all $n\geq 0$, one has $j(e_0,e_1)(n) \in \|\varphi(n)\|$. For $n=0$, one has that $j(e_0, e_1)0 = e_0$ which is an element of $\|\varphi(0)\|$ by hypothesis. Suppose the result holds for $n$, so that $j(e_0, e_1)n\in \|\varphi(n)\|$. Let $\ell = j(e_0, e_1)n$. Since $n\in \mathbb{Q}(n)$ and $\ell\in \|\varphi(n)\|$, it follows by the hypothesis on $e_1$ that $e_1 n \ell\in  \|\varphi(n+1)\|$. But by definition of $j$ and $\ell$, this is just to say that $j(e_0, e_1) (n+1)\in \|\varphi(n+1)\|$.  
\end{proof}

In these next results, we appeal often to the G\"odel translation (cf. Definition~\ref{GodelTrans}) and to its simple  consequences that we noted circa~(\ref{eqn:remarkgodelatomics}).
\begin{thm}\label{thm:EA}
All the theorems of $\mathsf{EA}^{\circ}$ are valid on the modal structure $\mu[\mathcal{N}]$. Further, each instance of the following modal analogue of the induction axiom is valid on the modal structure $\mu[\mathcal{N}]$, where $\theta(x)$ can be any modal formula:
\begin{equation}\label{eqn:inductionmodal}
[\Box \; \theta(0) \wedge \Box \; \forall \; x \; (\Box \; \theta(x) \Rightarrow \Box \; \theta(Sx))]\Rightarrow [\Box \; \forall \; x \; \theta(x)]
\end{equation}
Finally, on the modal structure $\mu[\mathcal{N}]$, the following is not valid (cf. (\ref{eqn:UniversalInstantiationAxiom2})), in that it does not have top value uniformly in $y$:
\begin{equation}\label{eqn:Epredcounter}
(\forall \; x \; \mathsf{E}(x)) \Rightarrow \mathsf{E}(y)
\end{equation}
where recall $\mathsf{E}(x)$ denotes the existence predicate, which was defined immediately following~(\ref{eqn:subs}).
\end{thm}
\begin{proof}
As consequences of Heyting arithmetic, the axioms of Robinson's~Q are valid in the $P(\mathcal{K}_1)$-valued structure $\mathcal{N}$. Then they are valid on the modal structure $\mu[\mathcal{N}]$ since they are obviously implied by their G\"odel translation. For instance, the G\"odel translation of Q2 is equivalent to $\Box \; \forall \; x, y \; \Box (Sx=Sy \Rightarrow x=y)$, which implies Q2 by an application of the $\mathsf{T}$-axiom. For the induction axiom, we must show that for all modal formulas $\varphi(x)$ we have a uniform witness to the following reduction, uniform in $D$:
\begin{equation}\label{eqn:inductioninthenats2}
\| \varphi(0) \|_{\mu}(D) \wedge \bigcap_{n\geq 0} (\mathbb{Q}(n)\Rightarrow (\|\varphi(n)\|_{\mu}(D)\Rightarrow \|\varphi(Sn)\|_{\mu}(D)))\leq \bigcap_{n\geq 0} (\mathbb{Q}(n) \Rightarrow \|\varphi(n)\|_{\mu}(D))
\end{equation}
But the same index used to verify~(\ref{eqn:inductioninthenats}) also works in this case. As for~(\ref{eqn:inductionmodal}), this follows from the G\"odel translation of the induction axiom in conjunction with the Change of Basis Theorem~\ref{prop:changeofbasis}. 

As for~(\ref{eqn:Epredcounter}), suppose not. Then there would be index which witnesses the following reduction, uniformly in $n\geq 0$ and $D$ from $P(\mathcal{K}_1)$:
\begin{equation}
\bigcap_{m\geq 0} (\mathbb{Q}(m) \Rightarrow \ominus_D \bigcup_{\ell\geq 0} (\mathbb{Q}(\ell) \wedge \ominus_D \|m=\ell\|)) \leq \ominus_D \bigcup_{\ell\geq 0} (\mathbb{Q}(n) \wedge \ominus_D \|n=\ell\|)
\end{equation}
By appealing to Proposition~\ref{mylittlehelper}, this is equivalent to:
\begin{equation}
\bigcap_{m\geq 0} (\mathbb{Q}(m) \Rightarrow \ominus_D \bigcup_{\ell\geq 0} (\mathbb{Q}(\ell) \wedge \|m=\ell\|)) \leq \ominus_D \bigcup_{\ell\geq 0} (\mathbb{Q}(n) \wedge \|n=\ell\|)
\end{equation}
Then by the semantics for identity, there would then be an $e$ witnessing the reduction $\bigcap_{m\geq 0} (\mathbb{Q}(m) \Rightarrow \ominus_D \mathbb{Q}(m)) \leq \ominus_D \mathbb{Q}(n)$, uniformly in $n\geq 0$ and $D$ from $P(\mathcal{K}_1)$. But by~(\ref{prop:diamonds2}), choose $e^{\prime}$ such that $e^{\prime}$ is in $(\mathbb{Q}(m) \Rightarrow \ominus_D \mathbb{Q}(m))$ for all $m\geq 0$ and all $D$ in $P(\mathcal{K}_1)$. But then $ee^{\prime}$ is an element of $\ominus_D \mathbb{Q}(n)$ for all $n\geq 0$ and $D$ in $P(\mathcal{K}_1)$. Letting $i$ be an index for the identity map, by choosing $D=\mathbb{Q}(n)$, one has $ee^{\prime}i$ is an element of $\mathbb{Q}(n)=\{n\}$ for all $n\geq 0$, a contradiction.
\end{proof}

Let us now record some further notation. In the arithmetical setting, the $\Delta_0$-formulas are the smallest class containing the atomics, closed under the propositional connectives and closed under bounded quantifiers---i.e., if $\varphi(x)$ is $\Delta_0$, then so are $\exists \; x<y \; \varphi(x)$ and $\forall \; x<y \; \varphi(x)$, which are respectively abbreviations for $\exists \; x \; (x<y \; \wedge \; \varphi(x))$ and $\forall \; x \; (x<y\Rightarrow \varphi(x))$. Of course, in the setting of Heyting arithmetic, we take $<$ to be defined as a certain atomic, but it provably has all of the usual features, e.g., is a linear ordering etc. (cf. \citet[volume 1 pp. 124 ff]{Troelstra1988aa}). Just as one can show, in Heyting arithmetic, that the law of the excluded middle holds for quantifier-free formulas (cf. \citet[volume 1 pp. 128]{Troelstra1988aa}), so can one show the same of the $\Delta_0$-formulas. 

Similarly, in the arithmetical setting, the $\Sigma_1$-formulas are formulas of the form $\exists \; \overline{x} \; \varphi(\overline{x})$ where $\varphi(\overline{x})$ is $\Delta_0$, and the $\Pi_1$-formulas are formulas of the form $\forall \; \overline{x} \; \varphi(\overline{x})$, where $\varphi(\overline{x})$ is $\Delta_0$. If $\mathbb{N}$ is the standard model of arithmetic, then the usual argument from G\"odel's incompleteness theorems implies that if $R\subseteq \mathbb{N}^n$ is definable by both a $\Sigma_1$-formula $\varphi(\overline{x})$ and a $\Pi_1$-formula $\psi(\overline{x})$, then for all $\overline{n}$ from $\mathbb{N}^n$ one has that $R(\overline{n})$ implies $\mathsf{HA}\vdash \varphi(\overline{n})$, while $\neg R(\overline{n})$ implies $\mathsf{HA}\vdash \neg \psi(\overline{n})$. This circumstance is sometimes expressed by saying that the number-theoretic relation $R$ is \emph{strongly representable} in the theory $\mathsf{HA}$. By Proposition~\ref{prop:heytingaxioms}, it then follows further that $R(\overline{n})$ implies $\varphi(\overline{n})$ is valid in $\mathcal{N}$, while $\neg R(\overline{n})$ implies $\neg \psi(\overline{n})$ is valid $\mathcal{N}$. Hence, we may take any $\Delta_1$-definable predicate over the natural numbers to be definable in the structure $\mathcal{N}$, and we will express this as the \emph{strong representability} of $\Delta_1$-definable predicates in the structure $\mathcal{N}$.

The usual argument (cf. \citet[volume 1 p. 199, volume 2 pp. 725-726]{Troelstra1988aa}) then shows: \begin{prop}
Church's thesis is valid on the structure $\mathcal{N}$:
\begin{equation}\label{eqn:CTinFlaggarithmeticsection}
[\forall \; n \; \exists \; m \; \varphi(n,m)] \; \; \Rightarrow [\exists \; e \;  \forall \; n \; \exists \; m \; \exists \; p \; (T(e,n,p) \wedge U(p,m) \wedge \varphi(n,m))]
\end{equation}
\end{prop}
\begin{proof}
Suppose that $e$ is a member of the antecedent. This is an abbreviation for the set $\bigcap_{n\geq 0} [\mathbb{Q}(n) \Rightarrow \bigcup_{m\geq 0} (\mathbb{Q}(m)\wedge \|\varphi(n,m)\|)]$. Then index $e$ on input $n$ returns an element of $ \mathbb{Q}(m)\wedge \|\varphi(n,m)\|$ for some $m\in \omega$. Hence, the computable function $n\mapsto p_1(e n)$ returns an element of $\|\varphi(n,m)\|$ for the value $m=p_0(e n)$. In conjunction with strong representability, we can use this to uniformly obtain a member of the consequent.
\end{proof}

\begin{thm}\label{thm:ECTisvalid}
$\mathsf{ECT}$~(\ref{eqn:whatwegotnow}) is valid on the modal structure $\mu[\mathcal{N}]$.
\end{thm}
\begin{proof}
One proceeds by computing the G\"odel translation of~(\ref{eqn:CTinFlaggarithmeticsection}) in the particular case where $\varphi(n,m)$ is an atomic $G(n,m)$, which perhaps contains parameters which are suppressed for the sake of readability. This is a conditional of the form $\Phi\Rightarrow \Psi$, and so its G\"odel translation will be of the form $\Box (\Phi^{\Box}\Rightarrow \Psi^{\Box})$. So let's proceed by computing $\Phi^{\Box}$ and $\Psi^{\Box}$ separately. Since atomics are not changed by the G\"odel translation, we have that $\Phi^{\Box}$ is $\Box [\forall \; n \; \exists \; m \;  G(n,m))]$. Now let's turn to  $\Psi^{\Box}$, which has the form $\exists \; e \; \Box \; \forall \; n \; \exists \; m \; \exists \; p \; (T(e,n,p) \wedge U(p,m) \wedge G(n,m))^{\Box})$. Since $T(e,n,p)$ and $U(p,m)$ are primitive recursive, we can drop the G\"odel translation on these, and similarly for the atomic $G$. Hence, in sum, on the modal structure $\mu[\mathcal{N}]$, we have
\begin{equation}
[\Box (\forall \; n \; \exists \; m \;  G(n,m))]\Rightarrow  [\exists \; e \; \Box \; \forall \; n \; \exists \; m \; \exists \; p \; (T(e,n,p) \wedge U(p,m) \wedge G(n,m))]
\end{equation}
Then we obtain $\mathsf{ECT}$~(\ref{eqn:whatwegotnow}) from this by the Change of Basis Theorem~\ref{prop:changeofbasis}.
\end{proof}

As mentioned in the introductory section, \cite{Flagg1985aa} rather suggested that his result established the consistency of $\mathsf{EA}+\mathsf{ECT}$. See, in particular, the first rule governing ``for all'' on \cite[p. 146]{Flagg1985aa}. Given that Theorem~\ref{thm:EA} displays an explicit counterexample to the free-variable variant~(\ref{eqn:UniversalInstantiationAxiom2}) of the \textsc{Universal Instantiation Axiom}~(\ref{eqn:UniversalInstantiationAxiom}), there must obviously be some place where our treatment differs. In short, Flagg identifies two concepts which we have distinguished, namely the semantics for identity and the quantifier, and in particular, Flagg assumes that the semantics for identity are rather defined by $\|n =n\|=\{n\}$ (cf. the semantics for identity on \citet[p. 150]{Flagg1985aa}). 

Now, as mentioned previously, the Heyting prealgebra $P(\mathcal{K}_1)$ is only two valued up to equivalence, and so for each particular $n\geq 0$, we have that $\{n\}\equiv \top$. But this equivalence is not uniform, as one can readily see by the discussion of the first paragraph of this section showing that the quantifier $\mathbb{Q}(n)=\{n\}$ is non-classical. But in the uniform $P(\mathcal{K}_1)$-semantics, one needs to have that $\|n=n\|\equiv \top$ uniformly, since this is just the condition pertaining to $e_{\mathrm{ref1}}, e_{\mathrm{ref2}}$ in Definition~\ref{defn:uniform}. The reason for our insistence on this condition is that we want $x=x$ to be valid on our modal structures, since this is an instance of~(\ref{eqn:necid}), an axiom of both $Q^{\circ}_{eq}.\mathsf{S4}$ and $Q_{eq}.\mathsf{S4}$.

On the modal structures, the semantics for $x=x$ is given by $\|x=x\|_\mu(D)$ which by definition is $\ominus_D \|x=x\|$. This is valid because $e_{\mathrm{ref1}}$ from Definition~\ref{defn:uniform} is a witness to the first reduction in $\top \leq \|x=x\|\leq \ominus_D \|x=x\|$, where the last reduction follows uniformly in $D$ by (\ref{prop:diamonds2}). By contrast, $\ominus_D \{x\}$ is not uniformly equivalent to $\top$, since if it were then we would have a uniform witness~$e$ to the reduction $(\{n\}\Rightarrow \{n\})\leq \{n\}$, and by taking $j$ to be an index for the identity function we would have $ej=n$ for any $n$, a contradiction. Thus if one follows Flagg in defining $\|x =x\|=\{x\}$ and one further defines $\|x=x\|_\mu(D)\equiv \ominus_D(\|x=x\|)$, then $x=x$ would not be valid, and thus the semantics would not be sound for $Q^{\circ}_{eq}.\mathsf{S4}$. However, from Flagg's definition of his deductive system, it seems that $x=x$ is a consequence of this deductive system (cf. first rule for identity on \citet[p. 145]{Flagg1985aa}). 

Another subtlety is that Flagg seems to suggest that one should define identity in the modal structure by  $\|x=x\|_\mu(D) =\{x\}$ (cf. clause~(I) of identity in \citet[Definition 4.2 p. 162]{Flagg1985aa}). However, this is not an element of the Boolean algebra $\mathbb{B}(N)$. For, suppose that one has uniformly in $x$ from $N$ and $D$ from $P(\mathcal{A})$ that $(\{x\}\Rightarrow D)\Rightarrow D \leq \{x\}$. But by choosing $D=\bot$, one has that this reduction can be rewritten as $(\bot \Rightarrow \bot) \leq \{x\}$ and since $(\bot \Rightarrow \bot)= \top$, this implies that the sets $\{x\}$ are not disjoint as $x$ varies, which is patently false. 

For these reasons, we have deviated from Flagg's own treatment by distinguishing between the role of the quantifier and the role of the semantics for the identity relation, and we insist that there always be uniform witnesses for the identity relation.

To close off this section, let's now note a result about the stability of $\Sigma_1$-formulas, and its consequences for the Barcan Formula. It's also helpful to note this because it shows us that in the proof of Proposition~\ref{thm:ECTisvalid}, it wasn't important that we took the Kleene T-predicate to be represented as a primitive recursive term, but rather we could have used any $\Sigma_1$-formula witnessing the strong representability of the Kleene $T$-predicate. This proposition is stated in \citet[p. 149]{Flagg1985aa}:

\begin{prop}\label{thm:sigma1stable}
The $\Sigma_1$-formulas are stable in the modal structure $\mu[\mathcal{N}]$ and are moreover equivalent to their G\"odel translations. That is, for every $\Sigma_1$-formula $\varphi(\overline{x})$, we have that $\forall \; \overline{x} \; [\varphi(\overline{x}) \Leftrightarrow (\Box \varphi(\overline{x})) \Leftrightarrow \varphi^{\Box}(\overline{x})]$ is valid in the modal structure $\mu[\mathcal{N}]$.
\end{prop}
\begin{proof}
For this proof, let's work axiomatically in the expansion  $\mathsf{EA}^{+}$ of $\mathsf{EA}^{\circ}$ by the G\"odel translations of all theorems of $\mathsf{HA}$. By the earlier results in this section, all the theorems of  $\mathsf{EA}^{+}$ are valid on the modal structure.

First we show that the results hold for all $\Delta_0$-formulas by induction on complexity of formulas. For atomics this follows from the stability of atomics and fact that the G\"odel translation doesn't change atomics. The inductive steps for conjunction and disjunction follow trivially from the inductive hypotheses and $Q^{\circ}_{eq}.\mathsf{S4}$. 

For conditionals, suppose that the result holds for $\varphi(\overline{x})$ and $\psi(\overline{x})$. It then suffices to show  $(\varphi(\overline{x})\Rightarrow \psi(\overline{x}))\Rightarrow \Box (\varphi(\overline{x})\Rightarrow \psi(\overline{x}))$. Since, as mentioned above, $\mathsf{HA}$ proves $\forall \; \overline{x} \; (\varphi(\overline{x}) \vee \neg \varphi(\overline{x}))$, we have that the G\"odel translation of this sentence holds in $\mathsf{EA}^{+}$, and this is equivalent to $\Box \; \forall \; \overline{x} \; (\varphi(\overline{x}) \vee \Box \neg \varphi(\overline{x}))$ since $\varphi(\overline{x})$ is by induction hypothesis equivalent to its own G\"odel translation and stable. In the case where $\varphi(\overline{x})$ holds we can then infer from $\varphi(\overline{x})\Rightarrow \psi(\overline{x})$ to $\psi(\overline{x})$ and hence to $\Box \psi(\overline{x})$ by induction hypothesis, then to $\Box (\varphi(\overline{x})\Rightarrow \psi(\overline{x}))$. In the case where $ \Box \neg \varphi(\overline{x})$ holds, we may directly infer that $\Box (\varphi(\overline{x})\Rightarrow \psi(\overline{x}))$. 

For the case of the bounded quantifiers, suppose that the result holds for $\varphi(y)$ and consider the case of the universal quantifier. Define the formula $\theta(x) \equiv [\forall \; y <x \; \varphi(y) ]\Rightarrow \Box  (\forall \; y<x \; \varphi(y)) $. It suffices to show that $\theta(0)$ and $\forall \; x \; (\theta(x)\Rightarrow \theta(Sx))$. Since $\mathsf{EA}^+$ proves $\forall \; y \; (\neg y<0)$, it also proves $\Box  (\forall \; y<0 \; \varphi(y))$. Now suppose that $\theta(x)$; we must show that $\theta(Sx)$. But since $\mathsf{EA}^{+}$ proves $y<Sx \Leftrightarrow (y<x \vee y=x)$, it also proves the necessitation of this. Hence in $\mathsf{EA}^{+}$, the formula $\theta(Sx)$ is equivalent to $[(\forall \; y <x \; \varphi(y)) \wedge \varphi(x)] \Rightarrow \Box (\forall \; y<x \; \varphi(y)) \wedge \Box \varphi(x)$, which follows from induction hypothesis. The case of the bounded existential quantifier, and the case for the $\Sigma_1$-formulas is exactly similar.
\end{proof}

The following proposition complements Proposition~\ref{prop:CBF}. One of Kleene's original examples of a sentence of arithmetic which was true but not realizable pertained to the halting set (cf. \citet[\S{9} pp. 115-116]{Kleene1945aa}, \citet[p. 71]{Kleene1943aa}). The idea of the below proof is to use Kleene's halting set example to show the invalidity of the Barcan Formula.

\begin{prop}\label{prop:failureofbarcan}
It is not the case that the Barcan Formula~$\mathsf{BF}$~(\ref{eqn:BF}) is valid on the modal $\mathbb{B}$-valued $L_0$-structure $\mu[\mathcal{N}_0]$.
\end{prop}
\begin{proof}
So suppose that this was valid. This reductio hypothesis in conjunction with Proposition~\ref{prop:CBF} implies that $[\forall \; x \; \Box \; \varphi(x)] \Leftrightarrow [\Box \; \forall \; x \;\varphi(x)]$ is also valid. 

Consider the halting set $\emptyset^{\prime}$, which for the sake of definiteness we take to have membership conditions $y\in \emptyset^{\prime}$ iff $\exists \; s \; T(y,y,s)$, where $T$ is again Kleene's $T$-predicate. But then we can argue that ``every number is either in or not in the halting set'' is equivalent to its own G\"odel translation:
\begin{align}
 \|(\forall \; y \; y\in \emptyset^{\prime} \vee y\notin \emptyset^{\prime})^{\Box}\|_{\mu}
\equiv & \|\Box \; \forall \; y \; (\exists \; s \;   T(y,y,s) \vee \Box ( (\exists \; t \; T(y,y,t))\Rightarrow \bot)\|_{\mu} \label{eqn:BCD1}\\
\leq & \| \forall \; y \; (\exists \; s \; T(y,y,s))\vee ((\exists \; t \; T(y,y,t))\Rightarrow \bot)\|_{\mu} \label{eqn:BCD25}\\
\leq & \| \forall \; y \; \forall \; t \; \exists \; s \; (T(y,y,s)\vee \neg T(y,y,t))\|_{\mu} \label{eqn:BCD2}\\
\leq & \| \forall \; y \; \forall \; t \; \Box \; \exists \; s \; (T(y,y,s)\vee \neg T(y,y,t))\|_{\mu} \label{eqn:BCD3}\\
\leq & \| \Box \; \forall \; y \; \forall \; t \; \exists \; s \; (T(y,y,s)\vee \neg T(y,y,t))\|_{\mu} \label{eqn:BCD4}\\
\leq & \| \Box \; \forall \; y \; (\exists \; s \; (T(y,y,s))\vee (\forall \; t \; \Box \; \neg T(y,y,t))\|_{\mu} \label{eqn:BCD5}\\
\leq & \| \Box \; \forall \; y \; (\exists \; s \; (T(y,y,s))\vee (\Box \; \forall \; t \; \neg T(y,y,t))\|_{\mu} \label{eqn:BCD6}\\
\leq & \| \Box \; \forall \; y \; (\exists \; s \; (T(y,y,s))\vee \Box \; ((\exists \; t \; T(y,y,t))\Rightarrow \bot) \|_{\mu}
\end{align}
In this, the inference from~(\ref{eqn:BCD1}) to~(\ref{eqn:BCD25}) follows from dropping the boxes via the $\mathsf{T}$-axiom, and the next step is just by manipulating the quantifiers in the usual way that is permitted in classical logic. The inference from~(\ref{eqn:BCD2}) to~(\ref{eqn:BCD3}) follows from Proposition~\ref{thm:sigma1stable}, and the inference from~(\ref{eqn:BCD3}) to~(\ref{eqn:BCD4}) follows from the reductio hypothesis. The inference from~(\ref{eqn:BCD4}) to~(\ref{eqn:BCD5}) follows from Proposition~\ref{thm:sigma1stable} again, and the inference from (\ref{eqn:BCD5}) to~(\ref{eqn:BCD6}) follows again from the reductio hypothesis, and the final step is again by usual equivalences in classical logic.

Since $\mathsf{PA}$ proves that ``every number is either in or not in the halting set,'' it is valid on the modal $\mathbb{B}$-valued $L_0$-structure $\mu[\mathcal{N}_0]$. Since it is equivalent on this structure to its G\"odel translation, by Theorem~\ref{translationthingie} this sentence is valid on the $P(\mathcal{A})$-valued $L_0$-structure $\mathcal{N}_0$. 

Let $e$ be a witness to this validity, and note that $\|y\in \emptyset^{\prime}\| \equiv \top$ iff $y\in \emptyset^{\prime}$, and $\|y\notin \emptyset^{\prime}\| \equiv \top$ iff $y\notin \emptyset^{\prime}$. By Proposition~\ref{prop:onpcas}, choose $e^{\prime}$ such that $e^{\prime}y = p_0 (ey)$. Then for all $y\geq 0$, one has that $e^{\prime}y = k$ iff $y\in \emptyset^{\prime}$, while $e^{\prime}y = \breve{k}$ iff $y\notin \emptyset^{\prime}$, contradicting the non-computability of the halting set. 
\end{proof}

As mentioned in the introductory section, Theorem~\ref{thm:Flaggthm} is then a direct consequence of this proposition in conjunction with Theorem~\ref{thm:ECTisvalid} and Theorem~\ref{thm:EA}.



\section{Epistemic Set Theory and Epistemic Church's Thesis}\label{sec:setheory}

Letting $\mathcal{A}$ be an arbitrary pca, in this section we look at the modalization of a uniform $P(\mathcal{A})$-valued set-theoretic structure, with the goal being to establish Theorem~\ref{thm:main}. So let $\mathcal{A}$ be a pca and let $\kappa>\left|\mathcal{A}\right|$ be strongly inaccessible. Then we define the following sequence of sets:
\begin{equation}\label{eqn:booleanvaluedmccarty}
V^{\mathcal{A}}_0 =\emptyset, \hspace{5mm} V^{\mathcal{A}}_{\alpha+1} =  P(\mathcal{A}\times V^{\mathcal{A}}_{\alpha}), \hspace{5mm} V^{\mathcal{A}}_{\lambda} = \bigcup_{\beta<\lambda} V^{\mathcal{A}}_{\beta} \mbox{ if $\lambda$ limit}
\end{equation}
This definition is due to \citet[p. 87]{McCarty1984aa}, \citet[p. 157]{McCarty1986aa} in the case of $\mathcal{A}=\mathcal{K}_1$. The observation that McCarty's construction essentially works without modification for an arbitrary pca $\mathcal{A}$ is indicated in \citet[\S{5} p. 293]{Rathjen2006aa}. Obviously the definition in~(\ref{eqn:booleanvaluedmccarty}) mirrors that of the Boolean-valued models of set theory, e.g., \citet[p. 21]{Bell1985aa}.

As usual, we define the rank of an element of $\bigcup_{\beta<\kappa} V^{\mathcal{A}}_{\beta}$ to be the least $\beta$ such that $V^{\mathcal{A}}_{\beta}$ contains it. By an easy induction on rank, one sees that all elements of $\bigcup_{\beta<\kappa} V^{\mathcal{A}}_{\beta}$ are subsets of $\mathcal{A}\times V^{\mathcal{A}}_{\beta}$ for some $\beta<\kappa$. We use the notation $\langle\cdot,\cdot\rangle$ for the set-theoretic ordered pair, so that $x\in \mathcal{A} \times V^{\mathcal{A}}_{\beta}$ iff $x=\langle e,y\rangle$ for some $e\in \mathcal{A}$ and $y\in V^{\mathcal{A}}_{\beta}$. This shouldn't be conflated with the pairing element $p$ in the pca $\mathcal{A}$ itself.

In this section, we always work with the classical quantifier~$\mathbb{Q}(x) =\top$. Hence as mentioned in Proposition~\ref{prop:BF}, this has the effect of removing the expression $\mathbb{Q}(x)$ from the clauses for $\exists$ and $\forall$. Further, the natural signature for the structure $V^{\mathcal{A}}_{\kappa}$ is the signature~$L_0$ consisting just of the binary membership relation. Hence, to put a uniform $P(\mathcal{A})$-valued structure on $V^{\mathcal{A}}_{\kappa}$, it suffices to specify the interpretation of the identity function and the binary  membership relation. One does this by defining the maps 
 $\| \cdot=\cdot \|: V^{\mathcal{A}}_{\kappa}\times V^{\mathcal{A}}_{\kappa} \rightarrow P(\mathcal{A})$ and $\|\cdot \in\cdot \|:V^{\mathcal{A}}_{\kappa}\times V^{\mathcal{A}}_{\kappa}\rightarrow P(\mathcal{A})$
 recursively by the following, which again may with profit be compared to the case of the Boolean valued models of set theory in \citet[p. 23]{Bell1985aa}:
\begin{align}
\|a\in b\| & =  \{p{e_0}{e_1} : \exists  c \; \langle e_0, c\rangle \in b \; \wedge \; e_1\in \| a=c\|)\} \label{def:elementhood}\\
\|a= b\| & = \{e_0 : \forall \; \langle n,c\rangle \in a \; e_0{n}\in \|c\in b\|\} \wedge \{e_1 : \forall \; \langle n, c\rangle \in b \; e_1{n}\in \|c\in a\|\} \notag
\end{align}
Further, for a set $Z\subseteq \mathcal{A}$, we let $p_0Z$ be the set of $x\in \mathcal{A}$ such that $pxy\in Z$ for some $y\in \mathcal{A}$, and similarly, we let $p_1Z$ be the set of $y\in \mathcal{A}$ such that $pxy\in Z$ for some $x\in \mathcal{A}$. Hence, one has that $p_0 \|a=b\| = \{e_0 : \forall \; \langle n,c\rangle \in a \; e_0{n}\in \|c\in b\|\}$ and $p_1\|a=b\| = \{e_1 : \forall \; \langle n, c\rangle \in b \; e_1{n}\in \|c\in a\|\}$.

Now we can verify that this generates a uniform $P(\mathcal{A})$-valued structure. While the statement of this proposition uses notions specific to this paper (such as uniform $P(\mathcal{A})$-valued structures), in the case of $\mathcal{A}=\mathcal{K}_1$, its proof is that of \citet[92-95]{McCarty1984aa}, and resembles the usual proof of substitution in Boolean-valued models, as in \citet[Theorem 1.17 p. 24]{Bell1985aa}. We include the proof in Appendix~\ref{sec:appendixmccarty}, since much of the modal semantics rides on uniformity considerations and since these uniformity considerations aren't explicit in the statement of the results in the unpublished \citet[92-95]{McCarty1984aa}, and while they are explicit in the statement of the results in \citet[Lemma 4.2 p. 291]{Rathjen2006aa} a complete proof is not given there.
\begin{prop}\label{prop:itsuniform} $V^{\mathcal{A}}_{\kappa}$ is a uniform $P(\mathcal{A})$-valued $L_0$-structure. 
\end{prop}
\begin{proof}
See Appendix~\ref{sec:appendixmccarty}.
\end{proof}

The following theorem is due to McCarty in the case of $\mathcal{A}=\mathcal{K}_1$. In this case, the result is stated without proof in \citet[Proposition 3.1 p. 158]{McCarty1986aa}, while the proof is given in \citet[pp. 96-97]{McCarty1984aa}, which remains unpublished. Hence, for the sake of completeness, we give the proof in Appendix~\ref{sec:appendixmccarty}. In the statement of McCarty's theorem, the theory $\mathsf{IZF}$ has as axioms the usual $\mathsf{ZF}$ axioms, but with collection in lieu of replacement, and with the induction scheme instead of the foundation axiom (cf. \citet[p. 58]{McCarty1984aa}, \citet[p. 158]{McCarty1986aa}). The reader should be warned that in other sources, such as \citet[Chapter VIII]{Beeson1985aa}, the name ``$\mathsf{IZF}$'' is rather used for a two-sorted theory, with one sort for numbers and another for sets. It's also  worth mentioning that the only place where we employ the hypothesis that $\kappa>\left|\mathcal{A}\right|$ is in the verification of collection.

\begin{thm}\label{thm:zfr-valued} All the theorems of $\mathsf{IZF}$ are valid on the uniform $P(\mathcal{A})$-valued $L_0$-structure $V_{\kappa}^{\mathcal{A}}$, and indeed on any expansion of this structure.
\end{thm}
\begin{proof}
Again, see Appendix~\ref{sec:appendixmccarty}.
\end{proof}

Since we're always working with classical quantifiers in this section, by Theorem~\ref{prop:S4soundness} the modal structures will be sound for $Q_{eq}.\mathsf{S4}$. Thus in the modal theories considered in this section, the background modal predicate logic will be $Q_{eq}.\mathsf{S4}$ (as opposed to $Q^{\circ}_{eq}.\mathsf{S4}+\mathsf{CBF}$). Now we show that this structure models the epistemic set theory $\mathsf{eZF}$ (cf. Definition~\ref{defn:eZF}), which forms part of Theorem~\ref{thm:main}:
\begin{prop}\label{prop:eZF}
All the axioms of $\mathsf{eZF}$ are valid on the modal $\mathbb{B}$-valued $L_0$-structure $\mu[V_{\kappa}^{\mathcal{A}}]$. \end{prop}
\begin{proof}
For Axiom~I (Modal Extensionality), this is implied by the G\"odel translation of the usual axiom of extensionality, since we can write its antecedent in a $\Delta_0$-fashion, namely as $(\forall \; z\in x \; z\in y)\wedge (\forall \; z\in y \; z\in x)$, keeping in mind the observations on G\"odel translations of $\Delta_0$-formulas from~(\ref{eqn:remarkgodelatomics}). 

As for Axiom~IV (Pairing) and Axiom~V (Union) of $\mathsf{eZF}$, note first that the usual non-modal versions of pairing and union assert the that for any set $\overline{x}$ there is a set $y=f(\overline{x})$, wherein the graph of the function~$f$ may be written in a $\Delta_0$-fashion. Then Axioms~IV and~V will follow directly from G\"odel translations of the usual non-modal versions of pairing and union. The arguments for Axiom VII (The Modal Power Set Axiom) and Axiom VIII (Infinity) are exactly parallel, where for the latter we make use of the second part of Proposition~\ref{prop:CBF} to get the second existential quantifier behind the initial box operator.

The final axioms of $\mathsf{eZF}$ are Comprehension$^{\Box}$, Collection$^{\Box}$, and Scedrov's Modal Foundation (Axiom III). We give the argument for Comprehension$^{\Box}$, since the arguments for the other two are similar. For Comprehension$^{\Box}$, first consider the following instance of ordinary non-modal comprehension, wherein we assume that the formula $G$ is atomic, perhaps in a signature extending that of set theory, and where for the sake of simplicity we assume that there is only one parameter variable:
\begin{equation}
\forall \; p \; \forall \; x \; \exists \; y \; \forall \; z \; (z\in y \Leftrightarrow (z\in x \; \wedge \; G(z, p)))
\end{equation}
The G\"odel translation of this implies the following, which is thus valid on the modal structure augmented with an interpretation for the atomic:
\begin{equation}
\forall \; p \; \forall \; x \; \exists \; y \; \Box \; \forall \; z \; (z\in y \Leftrightarrow (z\in x \; \wedge \; G(z, p)))
\end{equation}
Now suppose that $\varphi(z,p)$ is an arbitrary modal formula. By Flagg's Change of Basis Theorem~\ref{prop:changeofbasis}, we have the following is valid on the structure:
\begin{equation}
\forall \; p \; \forall \; x \; \exists \; y \; \Box \; (\forall \; z \; (z\in y \Leftrightarrow (z\in x \; \wedge \; \Box \varphi (z, p)))
\end{equation}
\end{proof}

In the remainder of the section, we build towards the proof of the other parts of Theorem~\ref{thm:main}. Given the previous proposition, it suffices then to establish a failure of the stability of negated atomics and the validity of $\mathsf{ECT}$~(\ref{eqn:whatwegotnow}) in the case $\mathcal{A}=\mathcal{K}_1$, and we do this in Proposition~\ref{prop:counterstabilityatomics} and Proposition~\ref{eqn:finallyECTsettheory}. In doing this, we work with the ersatzes of the individual natural numbers and the set of natural numbers within $V_{\kappa}^{\mathcal{A}}$. We first recall the definition of the so-called \emph{Curry numerals} in a pca $\mathcal{A}$ \cite[Definition 1.3.2 p. 12]{Oosten2008aa}, which can easily be shown to be distinct in any pca:
\begin{equation}\label{eqn:currynumerals}
\widetilde{0} = skk, \hspace{10mm} \widetilde{n+1} = p\breve{k}\widetilde{n}
\end{equation}
Like in \citet[Definition 3.6 pp. 105-106]{McCarty1984aa}, one then defines corresponding elements of $V_{\kappa}^{\mathcal{A}}$ as follows: 
\begin{equation}\label{eqn:mccartynumbers}
\overline{n} = \{\langle \widetilde{m}, \overline{m}\rangle: m<n\}, \hspace{5mm} \overline{\omega} =  \{\langle \widetilde{m}, \overline{m}\rangle: m<\omega\}
\end{equation}
As one can see by inspection of the proof of Theorem~\ref{thm:zfr-valued} presented in Appendix~\ref{sec:appendixmccarty}, these are the witnesses to the axiom of infinity in $V_{\kappa}^{\mathcal{A}}$.

First let's note the elementary proposition: 
\begin{prop}\label{prop:mnequality}
For all $m,n\geq 0$, one has (i)~$n<m$ iff $\|\overline{n}\in \overline{m}\| \neq \emptyset$, and (ii)~$n=m$ iff $\|\overline{n}=\overline{m}\|\neq \emptyset$. \end{prop}
\begin{proof}
For the left-to-right direction of (i), suppose that $n<m$. Then $\langle \widetilde{n}, \overline{n}\rangle\in \overline{m}$ and so $p \widetilde{n} i_0 \in \|\overline{n}\in \overline{m}\|$, where $i_0$ is an element of $\mathcal{A}$ such that $i_0\in \|a=a\|$ for any $a\in V_{\kappa}^{\mathcal{A}}$. For the left-to-right direction of (ii), we  have that $i_0\in \|a=a\|$ for any element $a$ of the structure. 

For the right-to-left direction, first consider a third condition: (iii) $m<n$ iff $\|\overline{m}\in \overline{n}\| \neq \emptyset$. Then we argue by simultaneous induction on $m$, that the right-to-left directions of (i)-(iii) hold. For the base case, consider $m=0$. For (i), since $\overline{0}=\emptyset$, we have that $\|\overline{n}\in \overline{m}\|=\emptyset$. For (ii), if $n\neq 0$ then $\overline{n}\neq \emptyset$ which implies that $\|\overline{n}=\overline{m}\|=\emptyset$. For (iii), suppose that $\|\overline{m}\in \overline{n}\|\neq \emptyset$ but not $0<n$. Then $n=0$ and so $\overline{n}=\emptyset$ and then $\|\overline{m}\in \overline{n}\| =\emptyset$.

Now suppose that the result holds for $m$; we show it holds for $m+1$. For (i), suppose that $\|\overline{n}\in \overline{m+1}\|\neq \emptyset$. Choose $p e_0 e_1$ in $\|\overline{n}\in \overline{m+1}\|$. Then there is $c$ with $\langle e_0, c\rangle \in \overline{m+1}$ with $e_1\in \|c=\overline{n}\|$. Then there is $\ell<m+1$ with $e_0=\widetilde{\ell}$ and $c=\overline{\ell}$. Then $\|\overline{\ell}=\overline{n}\|\neq \emptyset$ implies $\ell=n$ by the induction hypothesis for part~(ii). Then $n<m+1$, which is what we wanted to show. 

For (ii), suppose that $\|\overline{n}= \overline{m+1}\|\neq \emptyset$. Choose $e\in p_1 (\|\overline{n}= \overline{m+1}\|)$. Since $\langle \widetilde{m}, \overline{m}\rangle\in \overline{m+1}$, we have that $e \widetilde{m} \in \|\overline{m}\in \overline{n}\|$, so that $m<n$ by the induction hypothesis for~(iii). Choose $i\in p_0 (\|\overline{n}= \overline{m+1}\|)$. Then for all $\ell<n$, one has that $i \widetilde{\ell}\in \|\overline{\ell}\in \overline{(m+1)}\|$ and so by the previous paragraph (i.e. the $m+1$ case for (i)), we have that $\ell<m+1$. Thus in particular $n-1<m+1$, so that $n<m+2$. Hence collecting all this together, we have that $m<n<m+2$, so that $n=m+1$.

For (iii), suppose that $\|\overline{m+1}\in \overline{n}\|\neq \emptyset$. Choose $p e_0 e_1$ in $\|\overline{m+1}\in \overline{n}\|$. Then there is $c$ with $\langle e_0, c\rangle \in \overline{n}$ with $e_1\in \|c=\overline{m+1}\|$. Then there is $\ell<n$ with $e_0=\widetilde{\ell}$ and $c=\overline{\ell}$. Then $\|\overline{\ell} = \overline{m+1}\|\neq \emptyset$ and so by the previous paragraph (i.e. the $m+1$ case for (ii)), we have that $\ell=m+1$. Then $m+1<n$, which is what we wanted to show.
\end{proof}

To develop our counterexample to the stability of negated atomics (cf. Proposition~\ref{prop:counterstabilityatomics}), we need to work not only with the ersatzes of the natural numbers, but also with ersatzes of subsets of natural numbers. Hence, for any $X\subseteq \omega$, let $\chi_X:\omega\rightarrow \{k, \breve{k}\}$ be the characteristic function of $X$, defined by  $\chi_X(n)=k$ if $n\in X$ and $\chi_X(n)=\breve{k}$ if $n\notin X$. 
\begin{prop}
For any $X\subseteq \omega$, define $\widehat{X}=\{\langle p \widetilde{n} \chi_X(n), \overline{n}\rangle: n\geq 0\}$, which is an element of $V_{\omega+1}^{\mathcal{A}}$. Then $\|\overline{n}\in \widehat{X}\| = \{ p (p \widetilde{n} \chi_X(n)) e : e\in \|\overline{n}=\overline{n}\|\}$ for all $n\geq 0$.
\end{prop}
\begin{proof}
First suppose that $p e_0 e_1 \in \|\overline{n}\in \widehat{X}\|$. Then there is $d$ with $\langle e_0,d\rangle \in \widehat{X}$ and $e_1\in \|\overline{n}=d\|$. Then $\langle e_0,d\rangle = \langle p \widetilde{m} \chi_X(m), \overline{m}\rangle$ for some $m\geq 0$. But since $e_1\in \|\overline{n}=\overline{m}\|$ we have $n=m$ by the previous proposition, so that $\langle e_0,d\rangle = \langle p \widetilde{n} \chi_X(n), \overline{n}\rangle$, so that $p e_0 e_1 = p( p \widetilde{n} \chi_X(n)) e_1 $ with $e_1\in \|\overline{n}=\overline{n}\|$, which is what we wanted to show. For the converse containment, simply note that if $e\in \|\overline{n}=\overline{n}\|$ then $d=\overline{n}$ is a witness to $\langle p \widetilde{n} \chi_X(n), d\rangle \in \widehat{X} \; \wedge \; e\in \|d=\overline{n}\|$, so that $p (p \widetilde{n} \chi_X(n))e \in \|\overline{n}\in \widehat{X}\|$.
\end{proof}

Finally, we have the counterexample to the stability of negated atomics, which is part of Theorem~\ref{thm:main}:
\begin{prop}\label{prop:counterstabilityatomics} In the modal $\mathbb{B}$-valued $L_0$-structure $\mu[V_\kappa^{\mathcal{A}}]$, it is not the case that for every negated atomic formula $\neg R\overline{x}$, we have that $\|\neg  R (\overline x) \|_{\mu}  \equiv \|\Box(\neg  R(\overline x))\|_{\mu}$. In particular, a counterexample is the negated atomic formula ``$y\notin x$.''
\end{prop}
\begin{proof}
Suppose not. Then there would be an index $e\geq 0$ such that for all $a,b$ in $V_{\kappa}^{\mathcal{A}}$ and all $D$ from $P(\mathcal{A})$, one has that there is a uniform witness to $\|b\notin a \|_{\mu}(D) \leq \|\Box (b\notin a) \|_{\mu}(D)$. For any $n\geq 0$ and any $X\subseteq \omega$ one has that $\|\overline{n}\in \widehat{X}\|$ is non-empty by the previous proposition. Then by Proposition~\ref{prop:atomics}~(ii)-(iii), for all $n\geq 0$, $X\subseteq \omega$ and $D$ from $P(\mathcal{A})$, one has that there is a uniform witness to $(\|\overline{n}\in \widehat{X}\|\Rightarrow D) \leq D$. But taking $D=\|\overline{n}\in \widehat{X}\|$, one has that for all all $n\geq 0$, $X\subseteq\omega$, there is a uniform witness~$e$ to $(\|\overline{n}\in \widehat{X}\|\Rightarrow \|\overline{n}\in \widehat{X}\|) \leq \|\overline{n}\in \widehat{X}\|$. Let $e^{\prime}$ be an index for the identity function, so that $e^{\prime}$ is in $(\|\overline{n}\in \widehat{X}\|\Rightarrow \|\overline{n}\in \widehat{X}\|)$ for all $n\geq 0$ and $X\subseteq \omega$. Then by choosing a pair of distinct numbers $n, m$, the previous proposition implies that $p_0 p_0 (ee^{\prime})$ is equal to both $\widetilde{n}$ and $\widetilde{m}$, a contradiction.
\end{proof}

In these last two propositions of this section, we work over Kleene's first model $\mathcal{K}_1$. We begin with the following proposition from \citet[p. 158]{McCarty1984aa}, whose proof we include for the sake of completeness:
\begin{prop}\label{prop:mccartychurch}
On any expansion of the uniform $P(\mathcal{K}_1)$-valued structure $V_{\kappa}^{\mathcal{K}_1}$, the following is valid:
\begin{multline}
[\forall \; n\in \overline{\omega} \; \exists \; m\in \overline{\omega} \; \varphi(n,m)]\Rightarrow \\ [\exists \; e\in \overline{\omega} \; \forall \; n\in \overline{\omega} \; \exists \; m\in \overline{\omega} \; \exists \; p\in \overline{\omega} \; (T(e,n,p) \wedge U(p,m) \wedge \varphi(n,m))]
\end{multline}
\end{prop}
\begin{proof}
Suppose that $i\in \|\forall \; n\in \overline{\omega} \; \exists \; m\in \overline{\omega} \; \varphi(n,m)\|$. Then for all $b\in V_{\kappa}^{\mathcal{K}_1}$, one has that $i$ is in $\| b\in \overline{\omega} \Rightarrow \exists \; m\in \overline{\omega} \; \varphi(b,m)\|$. Then by definition of $\overline{\omega}$ one has that for all $n\geq 0$ that 
\begin{equation}
i(p \widetilde{n} i_0)\in \|\exists \; m\in \overline{\omega} \; \varphi(\overline{n},m)\|=\bigcup_{m\in V_{\kappa}^{\mathcal{K}_1}} \|m\in \overline{\omega}\|\wedge \|\varphi(\overline{n},m)\|
\end{equation}
where $i_0$ is the program such that $i_0\in \|a=a\|$ for all $a\in V_{\kappa}^{\mathcal{K}_1}$. Then let $e_0, e_1$ be the indexes for the program such that for all $n\geq 0$ one has $e_0 n = p_0 i(p \widetilde{n} i_0)$ and $e_1 n = p_1 i(p \widetilde{n} i_0)$. Then $e_0, e_1$ are total. Hence, for each $n\geq 0$ one may compute $m,p\geq 0$ such that $T(e_0,n,p) \wedge U(p,m)$. Since these are recursive, one may effectively find from $e_0,n,p,m$ a proof of $T(e_0,\overline{n},\overline{p}) \wedge U(\overline{p},\overline{m})$ from $\mathsf{IZF}$, and then ${e_1}n$ returns an element of $\|\varphi(\overline{n},\overline{m})\|$.
\end{proof}

\begin{prop}\label{eqn:finallyECTsettheory}
On the modal structure $\mu[V_{\kappa}^{\mathcal{K}_1}]$, the following is valid:
\begin{multline}
[\Box \; \forall \; n\in \overline{\omega} \; \exists \; m\in \overline{\omega} \; \Box \varphi(n,m)] \Rightarrow\\
 [\exists \; e\in \overline{\omega} \; \Box \; \forall \; n\in \overline{\omega} \; \exists \; m\in \overline{\omega} \; \exists \; p\in \overline{\omega} \; (T(e,n,p) \wedge U(p,m) \wedge \Box \varphi(n,m))]
\end{multline}
\end{prop}
\begin{proof}
By invoking the previous proposition, the proof is exactly the same as in the arithmetic case in \S\ref{sec:arithmetic}, using the Change of Basis Theorem~\ref{prop:changeofbasis} and the stability of formulas which are $\Sigma_1$-definable in the signature of arithmetic in exactly the same way.
\end{proof}

As mentioned in the introduction, Theorem~\ref{thm:main} follows from this proposition and the earlier Proposition~\ref{prop:eZF} and Proposition~\ref{prop:counterstabilityatomics}.


\section{Troelstra's Elementary Analysis and Kleene's Second Model}\label{sec:Troelstrakleene}

Now we focus on building models of fragments of second-order arithmetic with very limited amounts of comprehension. This will be relative to the pca $\mathcal{K}_2$, namely Kleene's second model (cf. \S\ref{sec:heytingfrompca}), and we shall concentrate our efforts on the theory of \emph{elementary analysis} $\mathsf{EL}$ studied by Troelstra (cf. \citet[\S\S{3.3, 3.4} pp. 982-983]{Troelstra1977aa}, \citet[volume 1 \S{3.6} pp. 144 ff]{Troelstra1988aa}, \citet[\S{2.5} pp. 425-426]{Troelstra1998aa}).

Let us extend the single-sorted signature of Heyting arithmetic $\mathsf{HA}$ by second sort reserved for functions from natural numbers to natural numbers. The modal semantics developed in \S\ref{sec:intuitiontomodal} carries over straightforwardly to the many-sorted setting. Since there are only two sorts, we'll simply reserve the lower-case Roman letters $\ell,m,n,x,y,z$ for numbers and we'll reserve the lower-case case Greek letters $\alpha, \beta, \gamma, \delta$ for functions from natural numbers to natural numbers. The only primitive that we add to the signature is the application function $(\alpha,n)\mapsto \alpha(n)$, which takes a number-theoretic function~$\alpha$ and a number~$n$ and evaluates~$\alpha$ at $n$.

We work with the standard model of second-order arithmetic, which we call $\mathcal{N}_0^2$. In this, the number sort is interpreted as $\omega = \{0,1,2,3,\ldots\}$ and the function sort is interpreted as Baire space $\omega^{\omega}$. The elements of the signature of Heyting arithmetic are interpreted exactly the same as in \S\ref{sec:arithmetic}, except that they are taken to be $P(\mathcal{K}_2)$-valued instead of $P(\mathcal{K}_1)$-valued. In many treatments of second-order arithmetic, equality for second-order objects would be defined in terms of coextensionality. However, since the modal semantics of \S\ref{sec:intuitiontomodal} always has identity built-in, we go ahead and assume that identity terms between functions are well-formed and interpreted disjunctively just as in~(\ref{eqn:myhowidentity}).

The application function $(\alpha,n)\mapsto \alpha(n)$ is interpreted by the usual application function given by the metatheory. The same argument as deployed \emph{vis-\`a-vis} first-order arithmetic in \S\ref{sec:arithmetic} shows that $\mathcal{N}_0^2$ is a uniform $P(\mathcal{K}_2)$-valued structure, and that any expansion $\mathcal{N}^2$ of $\mathcal{N}_0^2$ by new function or relation symbols is similarly a uniform $P(\mathcal{K}_2)$-valued structure. As for the quantifiers, here we must proceed a little differently from \S\ref{sec:arithmetic}, since the quantifiers must map each element of the domain to a subset of $\mathcal{K}_2$ instead of $\mathcal{K}_1$. Recall from \S\ref{sec:heytingfrompca} that $[\sigma]$ is the clopen through the finite string $\sigma$ in Baire space, so that $[(n)]=\{\alpha \in \omega^{\omega}: \alpha(0)=n\}$. 

We then define the quantifiers in $\mathcal{N}^2$  as $\mathbb{Q}(n) = [(n)]$ and $\mathbb{Q}(\alpha) =\{\alpha\}$, so that $\omega^{\omega}=\bigsqcup_{n\geq 0} \mathbb{Q}(n)$ and $\omega^{\omega}=\bigsqcup_{\alpha \in \omega^{\omega}} \mathbb{Q}(\alpha)$ provide us with two distinct partitions of Baire space. Using the terminology from Definition~\ref{defn:Qproperties}, it's easy to see that these quantifiers are non-degenerate, non-uniform, and non-classical (where these notions are relativized to sorts in the obvious way). They are also term-friendly, providing that we only introduce new recursive number-theoretic functions, like in \S\ref{sec:arithmetic}, and providing that we don't introduce any new functions defined on second-order objects outside of the application function. For, assuming this, it then suffices to show that there is $\gamma$ in $\mathcal{K}_2$ such that $\gamma: \mathbb{Q}(n)\wedge \mathbb{Q}(\alpha)\leadsto \mathbb{Q}(\alpha(n))$ for all $n\geq 0$ and $\alpha$. Since the function $F:\omega^{\omega}\times \omega^{\omega}\rightarrow \omega^{\omega}$ given by $F(\beta,\alpha)=\alpha(\beta(0))^{\frown} \overline{0}$ is a continuous function, by Proposition~\ref{prop:oostenfacts}.I choose $\delta$ such that $\delta \alpha \beta = F(\beta, \alpha)$. Then by Proposition~\ref{prop:onpcas} choose $\gamma$ such that $\gamma \gamma^{\prime} = \delta (p_0 \gamma^{\prime}) (p_1 \gamma^{\prime})$. This is why the quantifiers on this structure are term-friendly.

\begin{prop}\label{prop:newHA}
All of the axioms of $\mathsf{HA}$ are valid on the structure $\mathcal{N}^2$, as well as the following recursion axiom and choice schema and law of the excluded middle for the new atomics:
\begin{align}
& \forall \; n_0 \; \forall \; \alpha \; \exists \; \gamma \; \gamma(0) =n_0 \; \wedge \; \forall \; n \; \gamma(S(n))=\alpha(\gamma(n))\label{eqn:recursion} \\ 
& [\forall \; n \; \exists \; m \; \varphi(n,m)]\Rightarrow \exists \; \gamma \; [\forall \; n \; \varphi(n, \gamma(n))]\label{eqn:choicetroes1} \\
& \forall \; \alpha \; \forall \; n \; \forall \; m \; (\alpha(n)=m \vee \alpha(n)\neq m)\label{eqn:LEMnewatomics}
\end{align}
\end{prop}
\begin{proof}
Again, by Proposition~\ref{prop:soundness}, it suffices to verify the validity of the axioms. For $\mathsf{HA}$, the proof proceeds much as in the proof of Proposition~\ref{prop:heytingaxioms}, modulo needing to work with $\mathcal{K}_2$ instead of $\mathcal{K}_1$. In parallel to equation~(\ref{eqn:inductioninthenats}), for induction it suffices to find a witness to the following:
\begin{equation}\label{eqn:inductioninthenats4}
\| \varphi(0) \| \wedge \bigcap_{n\geq 0} ([(n)]\Rightarrow (\|\varphi(n)\|\Rightarrow \|\varphi(Sn)\|)\leq \bigcap_{n\geq 0} ([(n)] \Rightarrow \|\varphi(n)\|)
\end{equation}
For each $n\geq 0$, by Proposition~\ref{prop:onpcas} choose $\gamma_n$ such that $\gamma_n \alpha \beta =  (\alpha((n)^{\frown} \overline{0})) (\beta )$. Then by Proposition~\ref{prop:oostenfacts}.II, the map $G_n(\alpha, \beta)= \gamma_n\alpha\beta$ has $G_{\delta}$ domain $E_n$ and is continuous on this domain. Define a partial function of three variables $F:\omega^{\omega}\times \omega^{\omega}\times \omega^{\omega} \dashrightarrow \omega^{\omega}$ by $F(\alpha_0, \alpha_1, (0)^{\frown} \beta) = \alpha_0$ and $F(\alpha_0, \alpha_1, (n+1)^{\frown} \beta) = \gamma_n \alpha_1 F(\alpha_0, \alpha_1, (n)^{\frown} \beta)$. We claim that $F$ has $G_{\delta}$ domain and is continuous on this domain. Since the domains $D_n$ of $F\upharpoonright (\omega^{\omega}\times \omega^{\omega} \times [(n)])$ are disjoint and separated by opens, it suffices to show by induction on $n\geq 0$ that $F\upharpoonright (\omega^{\omega}\times \omega^{\omega} \times [(n)])$ is partial continuous and that its domain $D_n$ is $G_{\delta}$. For $n=0$, this is trivially the case. Suppose it holds for $n$. To show it holds for $n+1$, define $\widehat{F}(\alpha_0, \alpha_1, (n)^{\frown} \beta) = (\alpha_1, F(\alpha_0, \alpha_1, (n)^{\frown} \beta))$ which is continuous on $D_n$ by induction hypothesis. Then $(\alpha_0, \alpha_1, (n+1)^{\frown}\beta)\in D_{n+1}$ iff both $(\alpha_0, \alpha_1, (n)^{\frown} \beta)\in D_n$ and $(\alpha_0, \alpha_1, (n)^{\frown} \beta) \in \widehat{F}^{-1}(E_n)$, which is $G_{\delta}$ by induction hypothesis. Further $F$ is continuous on $D_{n+1}$ since it is the composition of two continuous functions. So indeed $F$ has $G_{\delta}$ domain and is continuous on this domain. Then by Proposition~\ref{prop:oostenfacts}.I there is a $\gamma$ such that $\gamma\alpha_0 \alpha_1 \beta= F(\alpha_0,\alpha_1,\beta)$. By Proposition~\ref{prop:onpcas}, choose $\gamma^{\prime}$ such that $\gamma^{\prime}\alpha = \gamma (p_0\alpha) (p_1 \alpha)$. Now, to verify~(\ref{eqn:inductioninthenats4}), suppose that $\alpha=p\alpha_0\alpha_1$ is in the antecedent of this reduction, so that $p_0\alpha = \alpha_0$ and $p_1\alpha =\alpha_1$. Then an easy induction on $n\geq 0$ shows that $\gamma^{\prime}\alpha = \gamma\alpha_0 \alpha_1$ is in the consequent of this reduction.

For the recursion axiom~(\ref{eqn:recursion}), for ease of readability, consider the specific case where $n_0$ has been fixed ahead of time, and let's find a witness to the following reduction, uniformly in $\alpha$:
\begin{equation}\label{eqn:recver}
\{\alpha\} \leq \bigcup_{\gamma} (\{\gamma\} \wedge [\bigcap_{n\geq 0} ([(n)]\Rightarrow (\|\gamma(0)=n_0\| \wedge \|\gamma(S(n))=\alpha(\gamma(n))\|))])
\end{equation}
Define a function $F:\omega^{\omega}\rightarrow \omega^{\omega}$ by $F(\alpha)(0)=n_0$ and $F(\alpha)(n+1)=\alpha(F(\alpha)(n))$. Then  $F:\omega^{\omega}\rightarrow \omega$ is continuous iff $\pi_n \circ F: \omega^{\omega}\rightarrow \omega$ is continuous, wherein $\pi_n$ denotes the projection onto the $n$-th component, so that $\pi_n(\beta) = \beta(n)$, and wherein $\omega$ is given the discrete topology. Clearly $\pi_0 \circ F$ is continuous since it is a constant function. Suppose that $\pi_n \circ F$ is continuous. Then $(\pi_{n+1}\circ F)^{-1}(\{k\}) = \bigcup_{\ell\geq 0} \{\alpha: \alpha(\ell)=k \; \wedge \; \alpha\in (\pi_{n}\circ F)^{-1}(\{\ell\})\}$ is open since it's a union of sets which are an intersection of a clopen and an open. Hence $F:\omega^{\omega}\rightarrow \omega^{\omega}$ is indeed continuous. By Proposition~\ref{prop:oostenfacts}.I there is a $\gamma$ such that $F(\alpha)=\gamma\alpha$. Choose $\beta_{\wedge}$ such that $\beta_{\wedge}$ is a uniform witness to $\top \leq \top \wedge \top$. By Proposition~\ref{prop:onpcas} choose $\delta$ such that $\delta\alpha = p ((\gamma \alpha)(\beta_{\wedge}))$. Let's verify that $\delta$ is a witness to~(\ref{eqn:recver}). So suppose that $\alpha$ is given. Then by construction $(\gamma\alpha)(0)=F(\alpha)(0)=n_0$ and so $\|(\gamma\alpha)(0)=n_0\|=\top$ and $(\gamma\alpha)(Sn) = F(\alpha)(n+1) = \alpha(F(\alpha)(n)) = \alpha( (\gamma\alpha)(n))$, so that $\|(\gamma\alpha)(Sn)=\alpha( (\gamma\alpha)(n))\|=\top$. Then by construction, $\beta_{\wedge}$ is an element of $ [(n)]\Rightarrow (\|\gamma(0)=n_0\| \wedge \|\gamma(S(n))=\alpha(\gamma(n))\|)]$ for all $n\geq 0$, which finishes the argument for the recursion axiom~(\ref{eqn:recursion}).

For the choice schema~(\ref{eqn:choicetroes1}), we must find a witness to the following reduction:
\begin{equation}\label{eqn:choicetroes1verif}
 \bigcap_{n\geq 0} [(n)]\Rightarrow \bigcup_{m\geq 0} ([(m)] \wedge \|\varphi(n,m)\|) \leq \bigcup_{\gamma} (\{\gamma\} \wedge (\bigcap_{n\geq 0} ([(n)] \Rightarrow \|\varphi(n,\gamma(n))\|)))
\end{equation}
Define a partial map $F:\omega^{\omega}\dashrightarrow \omega^{\omega}$ by $(F(\alpha))(n) = (p_0(\alpha ((n)^{\frown} \overline{0})))(0)$. For each $n\geq 0$, by Proposition~\ref{prop:onpcas}, there is $\delta_n$ such that $\delta_n\alpha = p_0(\alpha ((n)^{\frown} \overline{0}))$. By Proposition~\ref{prop:oostenfacts}.II, this map is continuous on its $G_{\delta}$ domain $D_n$, and hence its projection $\alpha\mapsto (p_0(\alpha ((n)^{\frown} \overline{0})))(0)$  onto its zero-th component  is also continuous with $G_{\delta}$ domain $D_n$, where we view $\omega$ as having the discrete topology. Then $F$ has $G_{\delta}$ domain $D=\bigcap_n D_n$ and $F:D\rightarrow \omega^{\omega}$ is continuous. By Proposition~\ref{prop:oostenfacts}.II, there is $\gamma$ such that $F(\alpha)=\gamma\alpha$. Similarly, there is $\gamma^{\prime}$ such that for all $\alpha$ and $n\geq 0$ and $\beta$ one has $(\gamma^{\prime}\alpha)((n)^{\frown} \beta) = p_1(\alpha((n)^{\frown}\overline{0}))$. By Proposition~\ref{prop:onpcas} choose $\delta$ such that $\delta \alpha = p((\gamma\alpha)(\gamma^{\prime}\alpha))$. Suppose that $\alpha$ is in the antecedent of~(\ref{eqn:choicetroes1verif}). Then we show that $\delta\alpha$ is in the consequent of~(\ref{eqn:choicetroes1verif}). For each $n\geq 0$, there is $m\geq 0$ such that $\alpha ((n)^{\frown} \overline{0})$ is in $[(m)] \wedge \|\varphi(n,m)\|$, so that $p_0(\alpha ((n)^{\frown} \overline{0}))$ is in $[(m)]$ and $p_1(\alpha ((n)^{\frown} \overline{0}))$ is in $\|\varphi(n,m)\|$. Then by the definition of $F$ and $\gamma, \gamma^{\prime}$, we have $(\gamma\alpha)(n)=(F(\alpha))(n)= (p_0(\alpha ((n)^{\frown} \overline{0})))(0)=m$, and for all $\beta$ we have $(\gamma^{\prime}\alpha)((n)^{\frown}\beta)=p_1(\alpha(n^{\frown}\overline{0}))\in \|\varphi(n,m)\|=\|\varphi(n, (\gamma\alpha)(n))\|$. Hence indeed $\delta\alpha$ is in the consequent of~(\ref{eqn:choicetroes1verif}), which finishes the verification of the choice schema~(\ref{eqn:choicetroes1}).

For the law of the excluded middle for the new atomics~(\ref{eqn:LEMnewatomics}), first define $F:\omega^{\omega}\times \omega^{\omega} \times \omega^{\omega}\rightarrow \omega^{\omega}$ by $F(\alpha, \beta, \gamma)= pkk$ if $\alpha(\beta(0)) = \gamma(0)$, while $F(\alpha, \beta, \gamma)= p\breve{k}\breve{k}$ otherwise. Then clearly $F$ is continuous and so by Proposition~\ref{prop:oostenfacts}.I, choose $\delta$ such that $\delta \alpha \beta \gamma = F(\alpha, \beta, \gamma)$. It suffices to show that $\delta\alpha$ is in $[(n)]\Rightarrow ([(m)]\Rightarrow (\|\alpha(n)=m\| \vee \|\alpha(n)\neq m)\|)$ for all $\alpha$ and $n,m\geq 0$. Fix $\alpha$ and $n,m\geq 0$ and $\beta\in [(n)]$; we must show that $\delta \alpha \beta$ is in $[(m)]\Rightarrow (\|\alpha(n)=m\| \vee \|\alpha(n)\neq m)\|)$. So suppose that $\gamma\in [(m)]$; we must show that $\delta\alpha \beta \gamma=F(\alpha, \beta, \gamma)$ is in $\|\alpha(n)=m\| \vee \|\alpha(n)\neq m)\|$. If $\alpha(\beta(0))=\gamma(0)$ then $\alpha(n)=m$ and so $\|\alpha(n)=m\|=\top$ and hence $F(\alpha, \beta ,\gamma)=pkk$ is in $\|\alpha(n)=m\| \vee \|\alpha(n)\neq m\|$. If $\alpha(\beta(0))\neq \gamma(0)$ then $\alpha(n)\neq m$ and so  $\|\alpha(n)=m\|=\bot$ and $\|\alpha(n)\neq m\|=\top$ and hence $F(\alpha, \beta ,\gamma)=p\breve{k}\breve{k}$ is in $\|\alpha(n)=m\| \vee \|\alpha(n)\neq m)\|$.
\end{proof}

Now, if $t(x)$ is term, then $\forall \; x \; \exists \; y \; t(x)=y$ is valid in the structure $\mathcal{N}^2$, since the quantifiers in this structure are term-friendly. Hence by the axiom of choice~(\ref{eqn:choicetroes1}), one has that there is $\gamma$ such that $\gamma(x)=t(x)$. Hence, without loss of generality, we may assume that we have $\lambda$-terms in the language, and write  $\lambda{x}.t$ in lieu of $\gamma$. These satisfy $\lambda$-conversion $(\lambda{x}.t)(n) = t(x/n)$ since $\gamma(n)=t(n)$. Similarly, we may assume that we have a term~${\bf r}_{n_0, \alpha}$ which provides a witness to the recursion axiom~(\ref{eqn:recursion}). Further, by a conceptually-slight but notationally-heavy modification of the proof of the recursion axiom~(\ref{eqn:recursion}), one may replace the base case~$n_0$ by a term $t(x_1, \ldots, x_{\ell})$ with $\ell$-parameter places and obtain a modified version of the recursion axiom which reads as follows:
\begin{equation}\label{eqn:recursion24}
\forall \; \alpha \; \exists \; \gamma \; \forall \; \overline{x} \; [\gamma(0) =t(\overline{x}) \; \wedge \; \forall \; n \; \gamma(S(n), \overline{x})=\alpha(\gamma(n), \overline{x})]
\end{equation}
In this, one thinks of the elements $\alpha, \gamma$ of Baire space as taking $\ell+1$-many inputs by using a primitive recursive pairing function on natural numbers, which is available in the structure due to it validating Heyting arithmetic $\mathsf{HA}$. Proceeding in this fashion, one then introduces the term ${\bf r}_{t,\alpha}$ for the $\gamma$ from~(\ref{eqn:recursion24}).

Troelstra's development of the theory $\mathsf{EL}$ of elementary analysis proceeded deductively rather than semantically, and so it was natural for him to work in a logic enriched by $\lambda$-terms and the recursion terms ${\bf r}_{t,\alpha}$. Since we can introduce these terms as abbreviations, it's then a consequence of the above proposition that all the theorems of $\mathsf{EL}$ are valid on the definitional expansion of $\mathcal{N}^2$ induced by these abbreviations. See the citations in the beginning paragraph of this section for references to Troelstra's explicit description of the syntax and axioms of $\mathsf{EL}$.

Troelstra further developed the basics of oracle computability in $\mathsf{EL}$, and showed for instance that if we define $(\alpha\beta)(n)=m$ in the natural $\Sigma^0_1$-way in arithmetic as $\exists \; \ell \; (\alpha(((n)^{\frown} \beta)\upharpoonright \ell)=m+1 \wedge \forall \; \ell_0<\ell \; \alpha(((n)^{\frown} \beta)\upharpoonright \ell_0)=0)$, then one can show in $\mathsf{EL}$ that $\{e_0\}^{\alpha\oplus \beta} = \alpha(\beta)$, where the left-hand side is written out in terms of oracle computations and where $e_0$ can be taken to be the same index as in~(\ref{eqn:appinkleenesecond}) \cite[volume 1 \S\S 3.7.9-3.7.10 p. 157]{Troelstra1988aa}. Another consequence which it is useful to take note of is that the usual Normal Form Theorem  for oracle computation is provable in $\mathsf{EL}$ \cite[volume 1 \S\S 3.7.6 p. 155]{Troelstra1988aa}.

These observations allow us to establish the following helpful proposition, where in part~(I), the expression ``$\{e\}^{\gamma}=\beta$'' in the context $\|\{e\}^{\gamma}=\beta\|$ is an abbreviation for ``$\forall \; n \; \exists \; q \; (T(e,n,\gamma, q) \wedge U(q,\beta(n)))$,'' where $T$ is the oracle-computability version of Kleene's $T$-predicate given by the Normal Form Theorem, and where both $T$ and $U$ may be represented by a term in the signature of $\mathcal{N}_0^2$ (using the $\lambda$-terms and ${\bf r}_{t,\alpha}$-terms introduced above). Further, in part~(II) of this proposition, the expression ``$\gamma\alpha =\beta$'' in the context $\|\gamma\alpha=\beta\|$ is an abbreviation for $\|\{e_0\}^{\gamma\oplus \alpha} = \beta\|$.
\begin{prop}
(I) For every index $e\geq 0$ there is $\delta_e$ such that for all $\gamma,\beta$ one has $\{e\}^{\gamma}=\beta$ implies $\delta_e\gamma \in \|\{e\}^{\gamma}=\beta\|$. (II) Hence, there is $\delta$ such that for all $\alpha, \beta, \gamma$ one has that $\gamma \alpha = \beta$ implies $\delta \gamma \alpha \in \|\gamma \alpha =\beta\|$. 
\end{prop}
\begin{proof}
First let's note why (II) follows from (I). By (I), one has that $\{e_0\}^{\gamma\oplus \alpha}=\beta$ implies $\delta_{e_0} (\gamma\oplus \alpha) \in \| \{e_0\}^{\gamma\oplus \alpha} = \beta\|$. Since $(\gamma,\alpha)\mapsto \gamma\oplus \alpha$ is continuous, by Proposition~\ref{prop:oostenfacts}.I, choose $p^{\prime}$ such that $p^{\prime} \gamma \alpha = \gamma \oplus\alpha$. Then by Proposition~\ref{prop:onpcas}, let $\delta$ be such that $\delta \gamma \alpha = \delta_{e_0} (p^{\prime} \gamma \alpha)$. Then $\gamma\alpha =\beta$ implies $\delta \gamma \alpha = \delta_{e_0} (p^{\prime} \gamma \alpha)=\delta_{e_0}(\gamma\oplus \alpha)\in \| \{e_0\}^{\gamma\oplus \alpha} = \beta\|$. By the remark in the previous paragraph, we have that $\|\{e_0\}^{\gamma\oplus\alpha} =\beta\|$ is identical to $\| \gamma \alpha = \beta\|$.

Now we prove (I). Choose index $e^{\prime}$ so that $e^{\prime}$ on input $n$ with oracle $\gamma$ searches for a halting computation of $e$ on input $n$ with oracle $\gamma$. By Proposition~\ref{prop:oostenfacts}.I, choose $\delta$ such that $\delta \gamma = \{e^{\prime}\}^{\gamma}$. The function $H$ which on input~$\alpha$ returns $H(\alpha) = ((\delta (p_0 \alpha))(n)) ^{\frown} \overline{0}$ for $n=(p_1\alpha)(0)$ is a partial continuous function with $G_{\delta}$ domain, and so by  Proposition~\ref{prop:oostenfacts}.I, choose $\eta$ with $\eta \alpha = H(\alpha)$. Then $\eta (p \gamma ((n)^{\frown} \beta)) = ((\delta \gamma)(n))^{\frown} \overline{0}$, so that
\begin{equation}\label{eqn:temp141234123}
 \beta\in [(n)] \mbox{ implies } \eta (p \gamma \beta) \in [(q)] \mbox{ for } q=(\delta \gamma)(n)
\end{equation}
Let $\delta_{\wedge}$ be any witness for $\top \leq \top \wedge \top$. By Proposition~\ref{prop:onpcas}, choose $\delta_e^{\prime}$ such that $\delta_e^{\prime} \beta = p (\eta (p \gamma \beta)) (\delta_{\wedge} \beta)$. Suppose that $\{e\}^{\gamma}=\beta$. It suffices to show that $\delta_e^{\prime}$ is a witness to $[(n)] \leq \bigcup_{q\geq 0} [(q)] \wedge (\|T(e,n,\gamma,q)\| \wedge \|U(q, \beta(n)\|)$ for all $n\geq 0$. Suppose that $\beta \in [(n)]$. Then by~(\ref{eqn:temp141234123}) we have $\eta (p \gamma \beta) \in [(q)]$ for $q=(\delta \gamma)(n)$. By definition of $\delta$, it follows that $q$ is a halting computation of $e$ on input $n$ with oracle $\gamma$, so that $T(e,n,\gamma,q)$ and since $\{e\}^{\gamma}=\beta$ by hypothesis, also $U(q,\beta(n))$. Then by the semantics for the atomics, it follows that $(\|T(e,n,\gamma,q)\| \wedge \|U(q, \beta(n)\|) = \top \wedge \top$, so that $\delta_{\wedge}\beta$ is in this set. 
\end{proof}

In \citet[p. 427]{Troelstra1998aa}, it is noted that in discussions of Kleene's second model $\mathcal{K}_2$, the role of Church's Thesis is taken over by the following schema, which is called \emph{generalized continuity}:
\begin{prop}\label{prop:genconprop} In $\mathcal{N}^2$, the following Generalized Continuity schema is valid:
\begin{equation}\label{eqn:plainGC}
[\forall \; \alpha \; \exists \; \beta \; \varphi(\alpha, \beta)] \Rightarrow [\exists \; \gamma \; \forall \; \alpha \; \exists \; \beta \; (\gamma\alpha =\beta \; \wedge \; \varphi(\alpha, \beta))]
\end{equation}
\end{prop}
\begin{proof}
It suffices to show that there is witness such that when given an element $\gamma$ of 
\begin{equation}\label{eqn:pre1341234}
\bigcap_{\alpha} (\{\alpha\}\Rightarrow \bigcup_{\beta} (\{\beta\} \wedge \|\varphi(\alpha, \beta)\|))
\end{equation}
returns an element of the set
\begin{equation}\label{eqn:post1341234}
\bigcup_{\gamma^{\prime}} (\{\gamma^{\prime}\} \wedge [\bigcap_{\alpha} (\{\alpha\} \Rightarrow \bigcup_{\beta} (\{\beta\} \wedge (\| \gamma^{\prime} \alpha = \beta \| \wedge \| \varphi(\alpha, \beta)\|)))])
\end{equation}
By Proposition~\ref{prop:onpcas} choose $\delta_i$ so that $\delta_i \gamma \alpha = p_i (\gamma \alpha)$, so that if $\gamma$ is an element of~(\ref{eqn:pre1341234}) then $\delta_1\gamma \alpha \in \|\varphi(\alpha, \delta_0 \gamma \alpha)\|$ and both $\delta_i \gamma\hspace{-1mm}\downarrow$ and $\delta_i \gamma\alpha\hspace{-1mm}\downarrow$ for all $\alpha$. By Proposition~\ref{prop:onpcas}, choose $\delta^{\prime}$ such that $\delta^{\prime} \gamma \alpha = p (\delta_0 \gamma \alpha)(p (\delta (\delta_0\gamma) \alpha)(\delta_1\gamma\alpha))$, where $\delta$ is from part~II of the previous proposition. Then similarly, let $\delta^{\prime\prime}$ such that $\delta^{\prime\prime}\gamma=p(\delta_0 \gamma)(\delta^{\prime}\gamma)$. Then we claim that $\delta^{\prime\prime}$ is the desired witness. So let $\gamma$ be from~(\ref{eqn:pre1341234}); we must show that $\delta^{\prime\prime}\gamma$ is in~(\ref{eqn:post1341234}). For this it suffices to show that $\delta^{\prime}\gamma\alpha$ is an element of:
\begin{equation}\label{eqn:post1341234v2}
\bigcup_{\beta} (\{\beta\} \wedge (\| \delta_0\gamma \alpha = \beta \| \wedge \| \varphi(\alpha, \beta)\|))
\end{equation}
Let $\beta = \delta_0\gamma \alpha$. Then since $(\delta_0\gamma) \alpha = \beta$, by the previous proposition, we have that $\delta ((\delta_0\gamma) \alpha)$ is in $\|\delta_0 \gamma \alpha = \beta\|$, where we keep in mind from the discussion prior to the previous proposition that the syntax of ``$\delta_0 \gamma \alpha = \beta$'' in the context $\|\delta_0 \gamma \alpha = \beta\|$ is \emph{not} an atomic. Finally, as remarked above, we also have that  $\delta_1\gamma \alpha \in \|\varphi(\alpha, \beta)\|$, and so we are done.
\end{proof}

Using the apparatus constructed thus far, we can then deduce the consistency of $\mathsf{HA}$ and an epistemic version of generalized continuity:
\begin{thm}\label{thm:genalizedcont}
In $\mu[\mathcal{N}^2]$, all the theorems of $\mathsf{EA}^{\circ}$ are valid, as is 
\begin{equation}\label{eqn:EGC}
[\Box \; \forall \; \alpha \; \exists \; \beta \; \Box \varphi(\alpha, \beta)] \Rightarrow [\exists \; \gamma \; \Box \; \forall \; \alpha \; \exists \; \beta \; (\gamma\alpha =\beta \; \wedge \; \Box \varphi(\alpha, \beta))]
\end{equation}
\end{thm}
\begin{proof}
As for the theorems of $\mathsf{EA}^{\circ}$, the argument proceeds just as in the proof of Theorem~\ref{thm:EA}, using Proposition~\ref{prop:newHA} in lieu of Proposition~\ref{prop:heytingaxioms}. As for~(\ref{eqn:EGC}) this follows directly from taking the G\"odel translation of~(\ref{eqn:plainGC}) in the case of an atomic and applying the Change of Basis Theorem~\ref{prop:changeofbasis}. In particular, Proposition~\ref{thm:sigma1stable} holds on this structure as well, so that since $\gamma\alpha =\beta$ in~(\ref{eqn:EGC}) is a $\Pi^0_2$-formula, it's implied by its own G\"odel translation. The reason that Proposition~\ref{thm:sigma1stable} holds on this structure is that the law the excluded middle holds for the new atomics by~(\ref{eqn:LEMnewatomics}).
\end{proof}


\section{Scott's Graph Model of the Untyped Lambda Calculus}\label{sec:untypedlambda}

To illustrate the generality of our construction, let's consider now Scott's graph model $\mathcal{S}$ described in \S\ref{sec:heytingfrompca}. Since in this pca the application operation~(\ref{eqn:appREpca}) is total, there is a natural way of viewing $\mathcal{S}$ as a uniform $P(\mathcal{S})$-valued structure in the sparse functional pca signature~$L$ (cf. immediately after Proposition~\ref{prop:onpcas} for a definition of this signature). On the uniform $P(\mathcal{S})$-valued $L$-structure, we interpret identity disjunctively as in~(\ref{eqn:myhowidentity}), and we use the quantifier $\mathbb{Q}(x) =\{x\}$ which is non-degenerate, non-uniform, and non-classical (cf. Definition~\ref{defn:Qproperties}). Further, it is term-friendly by Proposition~\ref{prop:onpcas}.

Then in parallel to Proposition~\ref{prop:scottmeyer}, we can show that 
\begin{prop}
On the uniform $P(\mathcal{S})$-valued $L$-structure $\mathcal{S}$, all of the following axioms are valid: $\forall \; x,y \; kxy=x$ and $\forall \; x,y,z \; sxyz = (xz)(yz)$ and $1_2 k = k$ and $1_3 s = s$ and the Meyer-Scott axiom $\forall \; a, b \; ((\forall \; x \; ax=bx)\Rightarrow 1a=1b)$. 
\end{prop}
\begin{proof}
For $\forall \; x \; \forall \; y \; kxy=x$, it suffices to find an element of $\bigcap_{x} (\{x\} \Rightarrow \bigcap_{y} (\{y\} \Rightarrow \|kxy=x\|))$. But in fact any element of $c$ is an element of this set. For, suppose that $x$ is given; we must show that $cx\in \bigcap_{y} (\{y\} \Rightarrow \|kxy=x\|)$. So suppose that $y$ is given; we must show that $cxy\in \|kxy=x\|$. But since $kxy=x$ is true, we have $\|kxy=x\|=\top$, and so trivially $cxy\in \|kxy=x\|$. A similar argument works in the case of $s$; and the validity of the identities $1_2 k = k$ and $1_3 s = s$ follows trivially since they are true in the classical model. For the Meyer-Scott Axiom, it suffices to find a witness to the reduction $\bigcap_{x} (\{x\}\Rightarrow \|ax=bx\|) \leq \|1a=1b\|$, uniformly in $a,b$. But again, any element $c$ will do. For, suppose that $y$ is an element of $\bigcap_{x} (\{x\}\Rightarrow \|ax=bx\|)$; we must show that $cy$ is in $\|1a=1b\|$. But for all $x$, we have that $yx$ is in $\|ax=bx\|$, so that it is true, and hence since the Meyer-Scott axiom holds classically, we have that $1a=1b$ is true, and so $\|1a=1b\|=\top$, and so $cy\in \|1a=1b\|$.
\end{proof}

Similar to the previous sections, it's convenient to consider expansions $\mathcal{S}^{\ast}$ of $\mathcal{S}$ by new relation or function symbols, which are for the same reasons still uniform $P(\mathcal{S})$-valued structures; and for the same reasons as in the previous sections, here too we refrain from introducing function symbols which violate term-friendliness. Then we have that the following choice principle holds on these expansions. However, note that in this choice principle, ``$cx=y$'' is literally the atomic formula, and so the proof of this choice principle is less involved than the proof of Proposition~\ref{prop:genconprop}.
\begin{prop}
On the uniform $P(\mathcal{S})$-valued $L$-structure $\mathcal{S}^{\ast}$, the following choice principle holds for any formula in the signature:
\begin{equation}\label{eqn:choicescott}
(\forall \; x \; \exists \; y \; \varphi(x,y))\Rightarrow (\exists \; c \; \forall \; x \; \exists \; y \; (cx=y \; \wedge \; \varphi(x,y)))
\end{equation}
\end{prop}
\begin{proof}
By Proposition~\ref{prop:onpcas}, choose $b_0, b_1, b_2$ from  $\mathcal{S}$ such that $b_0ax = p_0(ax)$ and $b_1ax = p((p_0(ax))(p (k) (p_1(ax))))$ and $b_2a = p(b_0 a)(b_1 a)$. Then $b_2$ is a witness to the reduction. For, suppose that $a$ is a member of the antecedent. Then for all $x$ one has that $p_1 (ax) \in \|\varphi(x, p_0(ax))\|$. Define $c=b_0a$, and let $x$ be arbitrary. Then by construction $cx=b_0ax = p_0(ax)$. Further, if we define $y=p_0(ax)$, we have that $p (k) (p_1(ax)) \in \|cx=y\| \wedge \|\varphi(x,y)\|$ and hence $b_1 ax \in \{y\} \wedge (\|cx=y\| \wedge \|\varphi(x,y)\|)$. Hence, since $x$ was arbitrary we have that $b_1a\in \| \forall \; x \; \exists \; y \; (cx=y \; \wedge \; \varphi(x,y)) \|$. Since $c=b_0a$, we then have that $b_2a \in \|\exists \; c \; \forall \; x \; \exists \; y \; (cx=y \; \wedge \; \varphi(x,y))\|$.
\end{proof}

\begin{prop}
On the modal $\mathbb{B}$-valued $L$-structure $\mu[\mathcal{S}]$, all of the following axioms are valid: $\forall \; x,y \; kxy=x$ and $\forall \; x,y,z \; sxyz = (xz)(yz)$ and $1_2 k = k$ and $1_3 s = s$ and the modal version of the Meyer-Scott axiom $\forall \; a, b \; ((\Box (\forall \; x \; ax=bx))\Rightarrow 1a=1b)$. In addition, one has the following choice principle, for any modal formula in the signature: 
\begin{equation}\label{eqn:boxedchoice}
\Box \; (\forall \; x \; \exists \; y \; \Box \; \varphi(x,y))\Rightarrow (\exists \; c \; \Box \; \forall \; x \; \exists \; y \; (cx=y \; \wedge \; \Box \; \varphi(x,y)))
\end{equation}
However, Meyer-Scott axiom itself $\forall \; a, b \; ((\forall \; x \; ax=bx)\Rightarrow 1a=1b)$ is not valid on $\mu[\mathcal{S}]$.
\end{prop}
\begin{proof}
The positive parts of the proposition follow from the G\"odel translation and Flagg's Change of Basis Theorem, by considerations which are by now routine. For the negative part, suppose for the sake of contradiction that the Meyer-Scott axiom itself $\forall \; a, b \; ((\forall \; x \; ax=bx)\Rightarrow 1a=1b)$ was valid on $\mu[\mathcal{S}]$. By Proposition~\ref{prop:counterextension}, fix $a,b\in \mathcal{S}$ such that $1a\neq 1b$ and for all $x\in \mathcal{S}$ one has $ax\neq bx$. Since $\ominus_D \bot \equiv D$ uniformly, for each $x\in \mathcal{S}$ we would have $\|ax= bx\|_{\mu}(D)\equiv D$ uniformly, and similarly $\|1a=1b\|_{\mu}(D)\equiv D$ uniformly. Then our reductio hypothesis yields a witness~$c$ to the reduction $\{x\}\Rightarrow D\leq D$ uniformly in $D$ and $x$. But since $skk$ is the identity function in pcas, one has that $skk$ is a member of $\{x\}\Rightarrow \{x\}$ for all $x$. But then taking $D=\{x\}$, we have that $c(skk)=x$ for all $x\in \mathcal{S}$, a contradiction.
\end{proof}

\begin{prop}\label{prop:failsurenegstable}
On the modal $\mathbb{B}$-valued $L$-structure $\mu[\mathcal{S}]$, the negated atomics are not stable, and in particular the negated atomic $x\neq y$ is not stable.
\end{prop}
\begin{proof}
\cite{Scott1975ab} p. 174 showed that the ``paradoxical combinator'' ${\bf y}$ used by Curry in his eponymous paradox (cf. \citet[Appendix B pp. 573 ff]{Barendregt1981aa}) is such that for all continuous $f:\mathcal{S}\rightarrow \mathcal{S}$ one has $f({\bf y} g) = {\bf y} g$, where $g=\mathrm{graph}(f)$ (cf. Proposition~\ref{prop:veryhlpeful112}.I). By Proposition~\ref{prop:veryhlpeful112}.I \& IV, choose ${\bf g}$ in $\mathcal{S}$ such that ${\bf g} c = \mathrm{graph}(\mathrm{fun}(c))$ for all $c$ from $\mathcal{S}$. Putting these two things together, one has that $c ({\bf y} ({\bf g}c)) = {\bf y} ({\bf g} c)$ for all $c$ from $\mathcal{S}$. By Proposition~\ref{prop:onpcas}, choose $\delta$ in $\mathcal{S}$ such that $\delta c = p({\bf y} ({\bf g}c))k$ for all $c$ in $\mathcal{S}$. Then $\delta$ is a witness to the validity of $\forall \; c \; \exists \; x \; cx = x$ on the uniform $P(\mathcal{S})$-valued $L$-structure $\mathcal{S}$.

The G\"odel translation of this is $\Box \; (\forall \; c \; \exists \; x \; cx=x)$ which is thus valid on the modal $\mathbb{B}$-valued $L$-structure $\mu[\mathcal{S}]$. In conjunction with the stability of identity and~(\ref{eqn:boxedchoice}), this implies that the following is valid on  the modal $\mathbb{B}$-valued $L$-structure $\mu[\mathcal{S}]$:
\begin{equation}
\Box \; (\forall \; x \; \exists \; y \; \Box \; \varphi(x,y))\Rightarrow \exists \; x \; \Box \; \varphi(x,x)
\end{equation}
If one applies this to $\varphi(x,y)\equiv x\neq y$, then because $\forall \; x \; \Diamond \; \neg \varphi(x,x)$ is valid on the modal $\mathbb{B}$-valued $L$-structure $\mu[\mathcal{S}]$, so too is 
\begin{equation}
\Diamond \; \exists \; x \; \forall \; y \; \Diamond \; x=y
\end{equation}

Suppose now for the sake of contradiction that the negated atomic $x\neq y$ was stable. This implies that $\Box \; \forall \; x,y \; (\Diamond (x=y) \Rightarrow x=y)$ is valid on the modal $\mathbb{B}$-valued $L$-structure $\mu[\mathcal{S}]$. In conjunction with the previous equation, we then obtain the validity of $\Diamond \; \exists \; x \; \forall \; y \; x=y$, which implies
\begin{equation}\label{eqn:juststop}
\Diamond \; \forall \; x, y \; x=y
\end{equation}
But by choosing distinct $x\neq y$ in $\mathcal{S}$, our reductio hypothesis gives that $\Box \; x\neq y$, which implies $\exists \; x, y \; \Box \; x\neq y$, which by Proposition~\ref{prop:CBF} implies $\Box \; \exists \; x, y \;  x\neq y$, which contradicts the previous equation.
\end{proof}


\section{Conclusions and Further Questions}\label{sec:conclusions}

Drawing the comparison to modal logic prior to Kripke semantics, Horsten once reported the concern about $\mathsf{ECT}$~(\ref{eqn:whatwegotnow}) and related systems that ``in the absence of a clear and unifying semantic framework there is the suspicion [\ldots] that we don't know what we are talking about'' \cite[p. 24]{Horsten1998aa}. While one can always adopt an instrumentalist attitude towards semantics, it's hard not to have some sympathy for Horsten's remark, and it's our hope that the realizability semantics developed here, which we take to generalize Flagg's 1985 construction, can add to our understanding of $\mathsf{ECT}$~(\ref{eqn:whatwegotnow}). On the one hand, we've tried to do so by making transparent how, e.g., these semantics don't validate the Barcan Formula~$\mathsf{BF}$~(\ref{eqn:BF}) or the stability of non-identity (cf. Proposition~\ref{prop:failureofbarcan} and Proposition~\ref{prop:failsurenegstable}), but they do validate the Converse Barcan Formula~$\mathsf{CBF}$~(\ref{eqn:CBF}) and the stability of identity and other atomics (cf. Proposition~\ref{prop:CBF} and Proposition~\ref{prop:atomics}). On the other hand, we've sought to specify the senses in which the Flagg construction does not depend crucially on the ordinary model of computation over the natural numbers. One degree of freedom is that we can vary the partial combinatory algebra over which we are working, and we can rather work over Kleene's second model or Scott's graph model of the untyped lambda-calculus (cf. \S\S\ref{sec:Troelstrakleene}-\ref{sec:untypedlambda}). But we can also vary the base theory from Flagg's original setting of arithmetic to the setting of set theory using McCarty's construction (cf. \S\ref{sec:setheory}).

But there is obviously much that is left unresolved by this paper. Four issues in particular seem especially noteworthy. First, one would like some understanding of how complex the consequence relation is for the modal semantics developed here. Given that the semantics for the box operator is defined in terms of intersections over arbitrary subsets of the underlying pca (cf. Definitions~\ref{inf(f)} and \ref{defn:boxremark}), one suspects that the induced consequence relation might be rather complex, and showing that it was not computably enumerable would be good evidence that there is no natural completeness theorem for these semantics. 

Second, given the centrality of the stability of atomics to the Flagg construction, the considerations of this paper tell us nothing about the consistency of $\mathsf{ECT}$~(\ref{eqn:whatwegotnow}) plus the failure of the stability of atomics against the background of an epistemic set theory. (See \citet[p. 149]{Flagg1985aa} for a discussion of the stability of atomics against the background of an epistemic arithmetic.)

Third, the main open technical question about $\mathsf{ECT}$~(\ref{eqn:whatwegotnow}) is whether it plus epistemic arithmetic is a conservative extension of Peano arithmetic $\mathsf{PA}$ (cf. \citet[p. 649]{Horsten1997aa}, \citet[p. 20]{Horsten1998aa}, \citet[p. 260]{Horsten2006aa}, \citet[Question 1 p. 462]{Halbach2000ab}). But the conservation of a given system over a subsystem can often be established by showing that any model of the subsystem can be expanded into a model of the given system. Because of this, as well as due to its intrinsic interest, one would want to know whether the model of $\mathsf{EA}^{\circ}+\mathsf{ECT}$ constructed in~\S\ref{sec:arithmetic} validates all the same arithmetical truths as the standard model of arithmetic.

Fourth and finally, we indicated in \S\ref{sec:intro} how the stability of atomics motivated the modification $\mathsf{eZF}$ of Goodman and Scedrov's theory $\mathsf{EZF}$, and in \S\ref{sec:setheory} we showed that modal versions of McCarty's model, which is closely related to the Boolean-valued models of set theory, validated this theory $\mathsf{eZF}$. However, Martin once suggested that ``[\ldots] a Boolean-valued  model would hardly seem to be a contender for the universe of sets'' \cite[p. 15]{Martin2001aa}. While the philosophical interpretation of Boolean-valued models is currently being debated, one might reasonably want some further assurance that $\mathsf{eZF}$ had the right to the name of a set~theory at all, and for this it would be natural to show that it actually interpreted ordinary set theory $\mathsf{ZF}$, or at least some fragment of this.


\appendix


\section{Appendix: Proposition on Kleene's Second Model}\label{appendix:specificpcas}

Here is the proof of Proposition~\ref{prop:oostenfacts}, which we stated in \S\ref{sec:heytingfrompca} in the context of describing Kleene's second model $\mathcal{K}_2$ but only used in \S\ref{sec:Troelstrakleene}:

\begin{proof}
For (I) in the case of $n=1$, let $D=\bigcap_n U_n$, where $U_n$ is open. Let $\pi_n:\omega^{\omega}\rightarrow \omega$ be the $n$-th projection $\pi_n(\alpha)=\alpha(n)$, so that $(\pi_n \circ G):D\rightarrow \omega$ is continuous where $\omega$ is given the discrete topology. For each $n,m\geq 0$, choose open $V_{n,m}\subseteq U_n$ such that $(\pi_n \circ G)^{-1}(\{m\}) =V_{n,m}\cap D$. Define $\gamma$ by $\gamma((n)^{\frown} \sigma)=m+1$ if $[\sigma]\subseteq V_{n,m}$ and $[\tau]\nsubseteq V_{n,m}$ for all $\tau\prec \sigma$ and $m$ is the least with this property, and $\gamma((n)^{\frown} \sigma)=0$ if there is no such $m$. Suppose that $\alpha \in D$ and $n\geq 0$ and $m=G(\alpha)(n)$, so that $\alpha\in (\pi_n \circ G)^{-1}(\{m\})=V_{n,m}\cap D$. Choose least $\ell$ such that $[\sigma]\subseteq V_{n,m}$ where $\sigma=\alpha\upharpoonright \ell$, and note that there is no $m^{\prime}<m$ with this property since it would then imply that $G(\alpha)(n)$ was equal to both $m^{\prime}, m$. Then $\gamma((n)^{\frown} \sigma)=m+1$ and $\gamma((n)^{\frown} \tau)=0$ for all $\tau \prec \sigma$. Hence $(\gamma\alpha)(n) = m$. Finally, note that if $\alpha\notin D$, then $\gamma \alpha$ is undefined. For, if $\alpha\notin D$, then choose $U_n$ with $\alpha\notin U_n$. Then $\gamma((n)^{\frown}\sigma)=0$ for all initial segments $\sigma$ of $\alpha$, and so $(\gamma\alpha)(n)$ is undefined.

For (I) in the case of $n=2$ and $D=\omega^{\omega}\times \omega^{\omega}$, see \citet[Lemma 1.4.1 p. 16]{Oosten2008aa}. Then by applying this to the homeomorphism $\Gamma_2:\omega^{\omega}\times \omega^{\omega}\rightarrow \omega^{\omega}$ given by $\Gamma_2(\alpha_1, \alpha_2)=\alpha_1\oplus \alpha_2$, one obtains $\gamma_2$ such that $\gamma_2\alpha_1\alpha_2 = \alpha_1\oplus\alpha_2$. 

Now let's show (I) in the case of $n=2$ and $D\subseteq \omega^{\omega}\times \omega^{\omega}$ being $G_{\delta}$. Its image $D^{\prime}$ under $\Gamma_2$ is also $G_{\delta}$, and $F^{\prime}: D^{\prime}\rightarrow \omega^{\omega}$ defined by $F^{\prime}(\alpha_1\oplus \alpha_2) = F(\alpha_1, \alpha_2)$ is also continuous. Then by (I) in the case of $n=1$, one has that there is $\gamma$ such that $\gamma \beta = F^{\prime}(\beta)$ for all $\beta$ in $D^{\prime}$. Then $\gamma (\gamma_2 \alpha_1 \alpha_2) = \gamma (\alpha_1\oplus \alpha_2) = F^{\prime}(\alpha_1\oplus \alpha_2) = F(\alpha_1, \alpha_2)$ for all $(\alpha_1, \alpha_2)$ in $D$. By Proposition~\ref{prop:onpcas}, choose $\gamma^{\prime}$ such that $\gamma^{\prime}\alpha_1 \alpha_2 = \gamma (\gamma_2 \alpha_1 \alpha_2)$. Then $\gamma^{\prime}$ is the desired witness, and so we are done with the case $n=2$. 

To finish the proof of (I), suppose that the result holds for $n\geq 2$. To show it holds for $n+1$, suppose that $D\subseteq (\omega^{\omega})^{n+1}$ is $G_{\delta}$. Then its image $D^{\prime}$ under the homeomorphism $\Gamma_{n+1}: (\omega^{\omega})^{n+1}\rightarrow (\omega^{\omega})^{n}$ given by $\Gamma_{n+1}(\alpha_1, \ldots, \alpha_{n-1}, \alpha_n, \alpha_{n+1})=(\alpha_1, \ldots,\alpha_{n-1}, \alpha_n\oplus \alpha_{n+1})$ is also $G_{\delta}$. Further the function $F^{\prime}: D^{\prime}\rightarrow \omega^{\omega}$ defined by $F^{\prime}(\alpha_1, \ldots, \alpha_{n-1}, \alpha_n\oplus \alpha_{n+1}) = F(\alpha_1, \ldots, \alpha_{n-1}, \alpha_n, \alpha_{n+1})$ is continuous. Then by induction hypothesis, choose $\gamma$ such that $\gamma \beta_1 \cdots \beta_n = F^{\prime}(\beta_1, \ldots, \beta_n)$ for all $(\beta_1, \ldots, \beta_n)$ in $D^{\prime}$. Then one has $\gamma \alpha_1 \cdots \alpha_{n-1} (\gamma_2 \alpha_n \alpha_{n+1}) = F^{\prime}(\alpha_1, \ldots, \alpha_{n}\oplus \alpha_{n+1}) = F(\alpha_1, \ldots, \alpha_n, \alpha_{n+1})$. Then by Proposition~\ref{prop:onpcas}, choose an element $\gamma^{\prime}$ such that $\gamma^{\prime}\alpha_1 \cdots \alpha_{n+1} = \gamma \alpha_1 \cdots \alpha_{n-1} (\gamma_2 \alpha_n \alpha_{n+1})$. Then $\gamma^{\prime}$ is the desired witness.

For (II), for the case $n=1$, it follows from~(\ref{eqn:appinkleenesecond}). For the case of $n=2$, $(\alpha_1, \alpha_2)$ is in the domain the map iff (i)~$\alpha_1$ is in the domain of the map $\alpha\mapsto \gamma\alpha$ and (ii)~for all $n\geq 0$ there is $\sigma_1\oplus \sigma_2$ such that the following three $\Sigma^0_1$-conditions in oracle $\gamma\oplus \alpha_1\oplus \alpha_2$ occur: (ii.1) for all $i<\left|\sigma_1\right|$ there is $\tau_i\preceq \gamma\oplus \alpha_1$ such that $\sigma_1(i)=\{e_0\}^{\tau_i}(i)$, and (ii.2) $\sigma_2\preceq \alpha_2$ and (ii.3) $\{e_0\}^{\sigma_1\oplus \sigma_2}(n)\hspace{-1mm}\downarrow$. Note that condition (ii.1) is $\Sigma^0_1$ in the oracle since the $\Sigma^0_1$-formulas are closed under bounded quantifiers. Since conditions~(i)-(ii) are $\Pi^0_2$ in the oracle, it follows that the domain in the case $n=2$ is $\Pi^0_2$ and thus $G_{\delta}$. For continuity, suppose that $\alpha_{i,k}\rightarrow \alpha_i$ where $(\alpha_{1,k}, \alpha_{2,k})$ is in the domain of the function. Then by case $n=1$, we have $\gamma\alpha_{1,k}\rightarrow \gamma \alpha_1$. Suppose that we are trying to secure agreement with $\gamma\alpha_1\alpha_2$ up to length $\ell$. Choose $\sigma\preceq \gamma\alpha_1 \oplus \alpha_2$ such that $\{e_0\}^{\sigma}(i)\hspace{-1mm}\downarrow$ for all $i<\ell$. Then choose $K\geq 0$ large enough so that if $k\geq K$, then $\gamma\alpha_{1,k}$ agrees with $\gamma \alpha_1$ up to length $\left|\sigma\right|$ and $\alpha_{i,k}$ agrees with $\alpha_2$ up to length $\left|\sigma\right|$.
\end{proof}


\section{Appendix: Heyting Prealgebras and Uniformity}\label{sec:appendixheyting}

In this section, we briefly set up a logic in which to state and prove quantifier-free facts about Heyting prealgebras, and in this logic we prove (\ref{prop:diamonds1})-(\ref{prop:diamonds10}), and we connect this to the discussion of uniform witnesses in Heyting prealgebras from \S\ref{sec:heytingfrompca}. The well-formed formulas of this logic are simply the quantifier-free formulas in the language of Heyting prealgebras \emph{without equality}, and the deduction rules are the simply the usual natural deduction rules for the intuitionistic propositional calculus $\mathsf{IPC}$ together with the substitution rule ``from $\varphi(x)$ infer $\varphi(t)$.'' The axioms of Heyting prealgebras are naturally quantifier-free, and so for instance the axiom governing the conditional may be written
\begin{equation}
x\wedge y\leq z \mbox{ iff } y\leq x\Rightarrow z \label{defnheyting}
\end{equation}
Let us call the system formed by this logic and the quantifier-free versions of the axioms of Heyting prealgebras the system \emph{elementary Heyting prealgebras} or $\mathsf{EHP}$. It's obviously just a Heyting-prealgebra analogue of primitive recursive arithmetic $\mathsf{PRA}$ (cf. \citet[volume 1 \S{3.2} 120 ff]{Troelstra1988aa}). 

First let's begin by noting some more elementary consequences of $\mathsf{EHP}$:
\begin{prop}\label{prop:moreelementary}
$\mathsf{EHP}$ proves~(\ref{help1})-(\ref{help3}) and hence each are true on any Heyting prealgebra.
\end{prop}
\begin{proof}
For, equation~(\ref{help1}) follows from~(\ref{defnheyting}) by setting $y$ equal to $x\Rightarrow z$. For equation~(\ref{help2}), suppose $x\leq y$; we want to show $y\Rightarrow z \leq x\Rightarrow z$. By equation~(\ref{defnheyting}) this is equivalent to $x\wedge (y\Rightarrow z)\leq z$. But since $x\leq y$, we have $x\wedge (y\Rightarrow z) \leq y \wedge  (y\Rightarrow z) \leq z$, where the second reduction comes from equation~(\ref{help1}). Finally, equation~(\ref{help3}) follows from two applications of equation~(\ref{help1}).
\end{proof}

Using this more elementary proposition, we may then establish:
\begin{prop}\label{prop:thediamondprop22}
$\mathsf{EHP}$ proves (\ref{prop:diamonds1})-(\ref{prop:diamonds10}), so that these are true on all Heyting prealgebras.
\end{prop}
\begin{proof}
For equation~(\ref{prop:diamonds1}), suppose that $x\leq y$. Then by equation~(\ref{help2}), we have $y\Rightarrow d \leq x\Rightarrow d$. Then by equation~(\ref{help2}) again, we have $(x\Rightarrow d)\Rightarrow d \leq (y\Rightarrow d)\Rightarrow d$. Then $\ominus_{d} x \leq \ominus_{d} y$.  

For equation~(\ref{prop:diamonds2}), note that by equation~(\ref{defnheyting}), it is equivalent to $(x\Rightarrow d) \wedge x \leq d$, which follows immediately from equation~(\ref{help1}).

For equation~(\ref{prop:diamonds3}), let $z$ be $(x\Rightarrow d)$. Then by equation~(\ref{help1}), we have $(z\Rightarrow d) \wedge z \leq d$. Since $z$ is $(x\Rightarrow d)$ we have $(z\Rightarrow d)\equiv \ominus_d x$. We thus have that $\ominus_d x \wedge z \leq d$. By equation~(\ref{defnheyting}), it follows that $z\leq \ominus_d x \Rightarrow d$ or $(x\Rightarrow d)\leq \ominus_d x \Rightarrow d$. From this we argue that $(x\Rightarrow d) \wedge \ominus_d \ominus_d x \leq (\ominus_d x \Rightarrow d) \wedge ((\ominus_d x\Rightarrow d)\Rightarrow d)\leq d$, where the last reduction comes from equation~(\ref{help1}). By equation~(\ref{defnheyting}), this thus yields that $\ominus_d \ominus_d x \leq (x\Rightarrow d)\Rightarrow d = \ominus_d x$.

For equation~(\ref{prop:diamonds5}), we want to show that  $\ominus_{d}\ominus_{d} x \equiv \ominus_{d} x$. But we already have $\ominus_{d}\ominus_{d} x \leq \ominus_{d} x$ from equation~(\ref{prop:diamonds3}). To see $\ominus_{d} x\leq \ominus_{d}\ominus_{d} x$, we note that we have $x\leq \ominus_{d}(x)$ from equation~(\ref{prop:diamonds2}), and applying equation~(\ref{prop:diamonds1}) with $y$ equal to $\ominus_d(x)$, we obtain $\ominus_d x\leq \ominus_d\ominus_{d}(x)$.

For equation~(\ref{prop:diamonds4}), let us first show that $\ominus_d (x\wedge y)\leq \ominus_d x \wedge \ominus_d y$. Since $\wedge$ is defined in Heyting prealgebras as the infimum, it suffices to show that $\ominus_d (x\wedge y)\leq \ominus_d(x)$ and $\ominus_d(x\wedge y)\leq \ominus_d (y)$. But these both follow directly from equation~(\ref{prop:diamonds1}), since $x\wedge y \leq x$ and $x\wedge y \leq y$. Now we show the converse, namely  $\ominus_d x \wedge \ominus_d y\leq \ominus_d (x\wedge y)$. By equation~(\ref{defnheyting}), this is equivalent to
\begin{equation}\label{eqn:123421342134dsafsdflkkj}
((x \wedge y)\Rightarrow d) \wedge \ominus_d x \wedge \ominus_d y \leq d
\end{equation}
which by expanding is equivalent to
\begin{equation}\label{whatwewant}
[((x \wedge y)\Rightarrow d) \wedge ((x\Rightarrow d)\Rightarrow d)]  \wedge\ominus_d(y) \leq d
\end{equation}
By two applications of equation~(\ref{defnheyting}) we can get $((x\wedge y) \Rightarrow  d)  \leq (y\Rightarrow (x\Rightarrow d))$. Hence we may argue from this and equation~(\ref{help3}) to equation~(\ref{whatwewant}) by noting that $[((x \wedge y)\Rightarrow d) \wedge ((x\Rightarrow d)\Rightarrow d)]  \wedge\ominus_d(y) \leq (y\Rightarrow d)\wedge \ominus_d y \leq d$.

For equation~(\ref{prop:diamonds6}), note that we have by equation~(\ref{help1}) that $x\wedge (x\Rightarrow y)\leq y$. By equation~(\ref{prop:diamonds1}), we have $\ominus_d (x \wedge (x\Rightarrow y))\leq \ominus_d (y)$. By equation~(\ref{prop:diamonds4}) we obtain $\ominus_d x \wedge \ominus_d (x\Rightarrow y) \leq \ominus_d (y)$. Then by equation~(\ref{defnheyting}), we have $\ominus_d (x\Rightarrow y)\leq \ominus_d x \Rightarrow \ominus_d y$.

For equation~(\ref{prop:diamonds9}), since $\vee$ is defined in Heyting prealgebras as the supremum, it suffices to show that $\ominus_d (x\vee y)\geq \ominus_d(x)$ and $\ominus_d(x\vee y)\geq \ominus_d (y)$. But these both follow from equation~(\ref{prop:diamonds1}), since $x\vee y \geq x$ and $x\vee y \geq y$. 

For~(\ref{prop:diamonds12}), note that $d\leq d$ implies that $\top \leq (d\Rightarrow d)$. Hence one has that $\ominus_d d \leq (\ominus_d d)\wedge (d\Rightarrow d)\leq d$.

For equation~(\ref{prop:diamonds8}), note that it is equivalent to $d\leq (x\Rightarrow d)\Rightarrow d$. But by equation~(\ref{defnheyting}) this is equivalent to  $d\wedge (x\Rightarrow d)\leq d$, which follows directly from~$\wedge$ being a lower bound. Note that equation~(\ref{prop:diamonds7}) follows directly from equation~(\ref{prop:diamonds2}) since this implies $\top \leq \ominus_d (\top)$ and we always have the converse.

For equation (\ref{prop:diamonds10}), first note that $(\ominus_d (x) \Rightarrow d) \leq (x\Rightarrow d)$ is equivalent to $[x\wedge (\ominus_d (x) \Rightarrow d)] \leq d$ by (\ref{defnheyting}). But this follows from $x\wedge (\ominus_d (x) \Rightarrow d) \leq \ominus_d(x) \wedge (\ominus_d (x) \Rightarrow d)\leq d$ by~(\ref{prop:diamonds2}) and~(\ref{help1}). Second note that $(x\Rightarrow d) \leq (\ominus_d (x) \Rightarrow d)$ is equivalent to $[\ominus_d (x) \wedge (x\Rightarrow d)] \leq d$ by (\ref{defnheyting}). But since $\ominus_d x = (x\Rightarrow d)\Rightarrow d$, this follows from ~(\ref{help1}).

\end{proof}

Now we formally justify the claim, made initially in \S\ref{sec:heytingfrompca} but used throughout the paper, that there are uniform witnesses to the reductions in (\ref{prop:diamonds1})-(\ref{prop:diamonds10}). In the statement of this proposition, recall that the ample relational signature for pcas was introduced immediately after Proposition~\ref{prop:onpcas}.
\begin{prop}\label{prop:formaljustificationunfirm}
Suppose that $t(\overline{x}), s(\overline{x})$ are two terms in the signature of Heyting prealgebras. Suppose that $\mathsf{EHP}$ proves $t(\overline{x})\leq s(\overline{x})$. Then there is a closed term $\tau$ in the ample relational signature of pcas such that for all pcas $\mathcal{A}$ and all non-empty sets $\mathcal{X}$ and all $\overline{f}$ in $\mathbb{F}(\mathcal{X})$, one has $\tau:t(\overline{f})\leadsto s(\overline{f})$.
\end{prop}
\begin{proof}
In the context of this proof, for $a\in \mathcal{A}$ let's write $a:t(\overline{f})\leadsto s(\overline{f})$ as $a:t(\overline{f})\leq s(\overline{f})$. Let's further inductively define $a:\varphi(\overline{f})$ for all $a$ from $\mathcal{A}$ and all formulas $\varphi(\overline{x})$ of $\mathsf{EHP}$ as follows:
\begin{align}
a: \varphi(\overline{f}) \wedge \psi(\overline{f}) & \mbox{ iff } a=p_0a_0a_1 \; \wedge \ a_0: \varphi(\overline{f}) \mbox{ and } a_1:\psi(\overline{f}) \\
a: \varphi(\overline{f}) \vee \psi(\overline{f}) & \mbox{ iff } (a=pkb \; \wedge \;  bÂÂÂ:\varphi(\overline{f})) \mbox{ or } (a=p{\breve{k}}b \; \wedge \; b:\psi(\overline{f})) \\
a: \varphi(\overline{f}) \Rightarrow \psi(\overline{f}) & \mbox{ iff } \forall \; b\in \mathcal{A}, \mbox{ if } b:\varphi(\overline{f}) \mbox{ then } ab: \psi(\overline{f}).\label{eqn:condi}
\end{align}
and where we say that that no $b$ from $\mathcal{A}$ is such that $b:\bot$. It then suffices to show by induction on length of proofs that if $\mathsf{EHP},\varphi_0(\overline{x}), \ldots, \varphi_n(\overline{x})\vdash \psi(\overline{x})$, then there is closed term $\tau$ such that for all pcas $\mathcal{A}$ and all non-empty sets $\mathcal{X}$ and all $\overline{f}$ in $\mathbb{F}(\mathcal{X})$, one has $\tau:(\varphi_0(\overline{f}) \wedge \cdots \wedge \varphi_n(\overline{f}))\Rightarrow \psi(\overline{f})$. To aid in readability, we'll write $\Phi(\overline{x})$ for  $\varphi_0(\overline{x}) \wedge \cdots \wedge \varphi_n(\overline{x})$, and we'll drop the free variables from $\Phi, \varphi, \psi$ when not needed.

One base case is where $\psi$ is one of the axioms of Heyting prealgebras, which are taken care of by the closed terms $e_1, \ldots, e_{12}$ mentioned immediately after Definition~\ref{recHeytingpre}. Another base case is where $\psi$ is one of the $\varphi_k$, in which case we may take $\tau$ to be an appropriate projection function. 

The induction steps correspond to the introduction and elimination rules of the intuitionistic propositional calculus $\mathsf{IPC}$ and the substitution rule. For the ``and'' introduction rule, by Proposition~\ref{prop:onpcas}, choose $\tau$ such that $\tau \tau_0 \tau_1 x = p(\tau_0x)(\tau_1x)$. Then $\tau_i: \Phi \Rightarrow \psi_i$ implies $\tau \tau_0 \tau_1: \Phi \Rightarrow (\psi_0 \wedge \psi_1)$. Similarly, for the ``and'' elimination rule, by Proposition~\ref{prop:onpcas} choose $\tau_i$ such that $\tau_i \tau x = p_i (\tau x)$. Then $\tau: \Phi\Rightarrow (\psi_0\wedge \psi_1)$ implies $\tau_i \tau: \Phi \Rightarrow \psi_i$. 

For the ``or'' introduction rule, by Proposition~\ref{prop:onpcas} choose $\tau, \tau^{\prime}$ such that $\tau \tau_0 x = pk (\tau_0 x)$ and $\tau^{\prime} \tau_1 x = p\breve{k} (\tau_1 x)$. Then $\tau_0: \Phi \Rightarrow \psi_0$ implies $\tau \tau_0: \Phi\Rightarrow (\psi_0 \vee \psi_1)$, while $\tau_1: \Phi \Rightarrow \psi_1$ implies $\tau^{\prime} \tau_1: \Phi\Rightarrow (\psi_0 \vee \psi_1)$.  For the ``or'' elimination rule, by Proposition~\ref{prop:onpcas} choose $\iota$ such that $\iota{xyz} =xyz$, so that $\iota{kab}=kab=a$ and $\iota\breve{k}ab=\breve{k}ab=b$, so that $\iota$ serves to distinguish cases. By Proposition~\ref{prop:onpcas}, choose $\tau^{\prime}$ such that $\tau^{\prime} \sigma \tau_0 \tau_1 x = (\iota (p_0 (\sigma{x}) )\tau_0 \tau_1) (px (p_1(\sigma{x})))$. Suppose that $\sigma: \Phi\Rightarrow (\psi_0 \vee \psi_1)$ and $\tau_i: (\Phi \wedge \psi_i) \Rightarrow \xi$. Suppose that $x: \Phi$. Then $\sigma{x}: (\psi_0\vee \psi_1)$. Then $\sigma{x} = pky$ for $y: \psi_0$ or $\sigma{x} = p{\breve{k}}y$ for $y: \psi_1$. In the former case $\iota (p_0(\sigma{x}))\tau_0 \tau_1= \iota k \tau_0 \tau_1=\tau_0$ and $px(p_1(\sigma{x}))=px(p_1(pky))=pxy$, so that in sum $\tau^{\prime} \sigma \tau_0 \tau_1 x =\tau_0 (pxy)$ and clearly $\tau_0 (pxy):\xi$. The argument in the latter case is analogous.

For the ``arrow'' introduction rule, by Proposition~\ref{prop:onpcas} choose $\tau^{\prime}$ such that $\tau^{\prime}\tau xy = \tau (pxy)$. Then $\tau: (\Phi \wedge \varphi) \Rightarrow \psi$ implies $\tau^{\prime}\tau: \Phi \Rightarrow (\varphi \Rightarrow \psi)$. For the elimination rule, $\tau: \Phi\Rightarrow \varphi$ and $\sigma: \Phi\Rightarrow (\varphi \Rightarrow \psi)$ together imply $s \sigma \tau : \Phi \Rightarrow \psi$.

For the ex falso rule, suppose that $\sigma: \Phi\Rightarrow \bot$. Since no $b$ from $\mathcal{A}$ is such that $b:\bot$, it follows from~(\ref{eqn:condi}) that there is no $a$ such that $a:\Phi$, since otherwise $b =\sigma a$ would be such that $b:\bot$. Hence, it follows that $\sigma: \Phi\Rightarrow \psi$ for any $\psi$, again by~(\ref{eqn:condi}). 

For the substitution rule, suppose that $\mathsf{EHP}+\Phi(x)\vdash \psi(x)$ with $\sigma$ being a witness to the inductive hypothesis, so that for all pcas $\mathcal{A}$ and all non-empty sets $\mathcal{X}$ and all $f$ in $\mathbb{F}(\mathcal{X})$ one has $\sigma: \Phi(f)\Rightarrow \psi(f)$. Let $t(y)$ be a term in the signature of Heyting prealgebras. Then by taking $f=t(g)$, one has $\sigma: \Phi(t(g))\Rightarrow \psi(t(g))$ for all pcas $\mathcal{A}$ and all non-empty sets $\mathcal{X}$ and all $g$ in $\mathbb{F}(\mathcal{X})$.
\end{proof}


\section{Appendix: McCarty's Theorem}\label{sec:appendixmccarty}

The purpose of this brief appendix is to present a self-contained proof of Proposition~\ref{prop:itsuniform} and Theorem~\ref{thm:zfr-valued}. For references to McCarty's original work, see the discussion prior to the statements of these results in \S\ref{sec:setheory}. We begin with the proof of Proposition~\ref{prop:itsuniform}:

\begin{proof}
So it suffices to find $\rho, \sigma,\tau, \iota$ (for ``reflexivity'', ``symmetry,'' ``transitivity'', and ``indiscernability'') such that for all $a,b,c,a^{\prime}, b^{\prime}$ from $V_{\kappa}^{\mathcal{A}}$ we have $\rho: \top \leadsto \|a=a\|$,  $\sigma: \|a=b\|\leadsto \|b=a\|$, $\tau: \|a=b\|\wedge \|b=c\| \leadsto \|a=c\|$, $\iota: \|a=a^{\prime}\| \wedge \|b=b^{\prime} \|\wedge \|a\in b\|\leadsto \|a^{\prime}\in b^{\prime}\|$. 

For $\rho$, by the recursion theorem \cite[Proposition 1.3.4 p. 12]{Oosten2008aa}, choose an index $j$ such that $jn=pn(pjj)$. We show that $pjj\in \|a=a\|$ for all $a\in V^{\mathcal{A}}_{\kappa}$ by induction on rank, so that we can set $\rho{n}=pjj$. Suppose it holds for all sets in $V^{\mathcal{A}}_\kappa$ of lower rank than $a$. Suppose that $\langle n,c\rangle \in a$. Then $c$ has lower rank than $a$, so $pjj \in \|c=c\|$. Then we have $pn(pjj)\in \|c\in a\|$, and so $jn=pn(pjj)\in \|c\in a\|$. Since $\langle n,c\rangle \in a$ was arbitrary, we thus have  $j\in p_0\|a=a\|$. A similar argument shows $j\in p_1\|a=a\|$ and so $pjj\in \|a=a\|$. 

Symmetry is trivial since it is witnessed by $\sigma(p{e_0}{e_1})= p{e_1}{e_0}$, which exists by Proposition~\ref{prop:onpcas}.

For transitivity, by Proposition~\ref{prop:onpcas}, choose $\wp$ from $\mathcal{A}$ such that $\wp uvwxy = u(v(w(xy)))$. By the  recursion theorem \cite[Proposition 1.3.4 p. 12]{Oosten2008aa}, choose $\tau$ such that
\begin{align}
(p_0\tau) (p (p e_0 e_1) (p e_0^{\ast} e_1^{\ast})) n & = p(\wp p_0  e_0^{\ast}  p_0  e_0 n) (\tau(p (p_1 (e_0n))  (\wp p_1  e_0^{\ast}  p_0  e_0 n))) \notag \\
(p_1\tau) (p (p e_0 e_1) (p e_0^{\ast} e_1^{\ast})) n & = p (\wp p_0  e_1  p_0  e_1^{\ast} n )(\tau(p (p_1 (e_1^{\ast} n)) (\wp p_1  e_1  p_0  e_1^{\ast} n))) \notag
\end{align}
We argue by induction on rank. So suppose that $p{e_0}{e_1} \in \|a=b\|$ and $p{e_0^{\ast}}{e_1^{\ast}} \in \|b=c\|$. We must show that $(p_0\tau) (p (e_0 e_1) p(e_0^{\ast} e_1^{\ast})) \in p_0 \|a=c\|$ and $(p_1\tau) p( (e_0 e_1) p(e_0^{\ast} e_1^{\ast})) \in p_1 \|a=c\|$. We focus on the first since the other is similar. For this we must show that if $\langle n,d\rangle \in a$ then $(p_0\tau) (p(e_0 e_1) p(e_0^{\ast} e_1^{\ast})) n \in \|d\in c\|$. So suppose that $\langle n,d\rangle\in a$. Since $e_0 \in p_0 \|a=b\|$, we have that $e_0n\in \|d\in b\|$. Let ${e_0}n=p{n_0}{n_1}$. Since $p{n_0}{n_1} \in \|d\in b\|$ there is $d^{\prime}$ with $\langle n_0, d^{\prime}\rangle \in b$ and $n_1\in \|d=d^{\prime}\|$. Since  $\langle n_0, d^{\prime}\rangle \in b$ and $e_0^{\ast}\in p_0 \|b=c\|$, we have $e_0^{\ast}n_0\in \|d^{\prime}\in c\|$. Then $e_0^{\ast}n_0=p{m_0}{m_1}$ and there is $d^{\prime\prime}$ with $\langle m_0, d^{\prime\prime}\rangle\in c$ and $m_1\in \| d^{\prime}=d^{\prime\prime}\|$. By induction hypothesis $\tau(p{n_1}{m_1})\in \|d=d^{\prime\prime}\|$. Then $\exists \; d^{\prime\prime}\; (\langle m_0, d^{\prime\prime}\rangle \in c \; \wedge \; \tau(p{n_1}{m_1})\in \|d=d^{\prime\prime}\|)$. Hence $p{m_0}\tau(p{n_1}{m_1}) \in \|d\in c\|$. Now, note that
$m_0 = p_0(e_0^{\ast} n_0) = p_0(e_0^{\ast}(p_0 (e_0 n)))=\wp p_0  e_0^{\ast}  p_0  e_0 n$, and $n_1  = p_1(e_0{n})$ and $m_1 = p_1(e_0^{\ast} n_0) = p_1(e_0^{\ast}(p_0(e_0 n))) = \wp p_1  e_0^{\ast}  p_0  e_0 n$. From this and what was said earlier in the paragraph, we are then done with the verification that $\tau$ is the witnesses to transitivity.

Now it remains to define $\iota$. So apply Proposition~\ref{prop:onpcas} to obtain $\iota_0$ and $\iota$ such that $\iota_0 e_0 e_1 e_1^{\prime} = \tau(p (\sigma(p{e_0}{e_1}) e_1^{\prime}))$, 
$\iota e_0 e_1 e_0^{\ast} e_1^{\ast} e_0^{\prime} e_1^{\prime} = p (p_0 (e_0^{\ast}e_0^{\prime}))  (\tau (p (\iota_0 e_0 e_1 e_1^{\prime}) ((p_1 (e_0^{\ast} e_0^{\prime})))))$. So suppose that $p{e_0}{e_1} \in \|a=a^{\prime}\|$ and $p{e_0^{\ast}}{e_1^{\ast}} \in \|b=b^{\prime}\|$ and $p{e_0^{\prime}}{e_1^{\prime}} \in \|a\in b\|$. It suffices to show that $\iota e_0 e_1 e_0^{\ast} e_1^{\ast} e_0^{\prime} e_1^{\prime} \in \|a^{\prime}\in b^{\prime}\|$. Since $p{e_0^{\prime}}{e_1^{\prime}} \in \|a\in b\|$, there is $c$ such that $\langle e_0^{\prime}, c\rangle \in b$ and $e_1^{\prime}\in \|a=c\|$. Since $p{e_0}{e_1}\in \|a=a^{\prime}\|$ we have  $\tau(p(\sigma(p{e_0}{e_1})) e_1^{\prime})\in \|a^{\prime}=c\|$, or what is the same $\iota_0 e_0 e_1 e_1^{\prime}\in \|a^{\prime}=c\|$. Now, since $e_0^{\ast}\in p_0 \|b=b^{\prime}\|$ and  $\langle e_0^{\prime}, c\rangle \in b$, it follows that ${e_0^{\ast}}e_0^{\prime}\in \|c\in b^{\prime}\|$. If we write $e_0^{\ast}e_0^{\prime}=p{\ell_0}{\ell_1}$ then it follows that there is $c^{\prime}$ with $\langle \ell_0, c^{\prime}\rangle \in b^{\prime}$ and $\ell_1\in \|c=c^{\prime}\|$. Since  $\iota_0 e_0 e_1 e_1^{\prime}\in \|a^{\prime}=c\|$ and $\ell_1\in \|c=c^{\prime}\|$ we have that $\tau (p (\iota_0 e_0 e_1 e_1^{\prime}) (\ell_1)) \in \|a^{\prime}=c^{\prime}\|$. Since $\ell_i = p_i (e_0^{\ast} e_0^{\prime})$, we thus have  $\exists \; c^{\prime} \; \langle (p_0 (e_0^{\ast} e_0^{\prime}), c^{\prime}\rangle \in b^{\prime} \; \wedge \; 
\tau (p (\iota_0 e_0 e_1 e_1^{\prime}) (p_1 (e_0^{\ast} e_0^{\prime}))) \in \|a^{\prime}=c^{\prime}\|$, which is to say that $\iota e_0 e_1 e_0^{\ast} e_1^{\ast} e_0^{\prime} e_1^{\prime} \in \|a^{\prime}\in b^{\prime}\|$.
\end{proof}

Before going onto the proof of Theorem~\ref{thm:zfr-valued}, we need only one small preliminary proposition, which is from \cite{McCarty1984aa} Lemma 6.2 p. 92.
\begin{prop}\label{lem:mccartylem}
\begin{enumerate}
\item[]
\item[] (i) For all $b \in V_\alpha^{\mathcal{A}}$ and all $c\in V_{\kappa}^{\mathcal{A}}$, if $\|c \in b\|\equiv \top$ then $c \in V_\beta^{\mathcal{A}}$ for some $\beta < \alpha$.
\item[] (ii) If $a \in V_\alpha^{\mathcal{A}}$ and $\| a = b\|\equiv \top$, then $b \in V_\alpha^{\mathcal{A}}$.
\item[] (iii) If $\beta<\alpha$ then $V_{\beta}^{\mathcal{A}}\subseteq V_{\alpha}^{\mathcal{A}}$.
\item[]
\end{enumerate}
\end{prop} 
\begin{proof}
The proof of (i)-(iii) is by simultaneous induction on $\alpha$. The zero and limit steps are trivial. So suppose the result holds for $\alpha$; we show that it holds for $\alpha+1$. 

For (i), suppose that $b\in V_{\alpha+1}^{\mathcal{A}}$. Then $b\subseteq \omega\times V_{\alpha}^{\mathcal{A}}$. Suppose that $\|c\in b\|\equiv \top$. Choose $p e_0 e_1 \in \|c\in b\|$, so that there is $d$ with $\langle e_0,d\rangle \in b$ and $e_1\in \|c=d\|$. Then $d\in V_{\alpha}^{\mathcal{A}}$ and so by the induction hypothesis for~(ii) we have that $c\in V_{\alpha}^{\mathcal{A}}$. 

For (ii), suppose that $a\in V_{\alpha+1}^{\mathcal{A}}$ and that $\|a=b\|\equiv \top$ with witness $j\in \|a=b\|$. Suppose that $\langle e,d\rangle \in b$. It suffices to show that $d\in V_{\alpha}^{\mathcal{A}}$ since then we would have $b\in V_{\alpha+1}^{\mathcal{A}}$. Recall that in the proof of Proposition~\ref{prop:itsuniform}, we showed that there is an index $i_0$ such that $i_0 \in \|d=d\|$ for all $a\in V_{\kappa}^{\mathcal{A}}$. Then $\langle e,d\rangle \in b$ implies $p e i_0 \in \|d\in b\|$. Then by appeal to $V_{\kappa}^{\mathcal{A}}$ being a uniform $P(\mathcal{A})$-valued structure and substitution therein (Proposition~\ref{prop:uniform valuation} in conjunction with Proposition~\ref{prop:itsuniform}), choose $k_0$ such that for all $u,v,w$ one has $k_0\in \| (u\in v \; \wedge \; w=v)\Rightarrow u\in w\|$. Then $k_0 p(p e i_0) j\in \|d\in a\|$. Then by the induction hypothesis for (i) and (iii), we have that $d\in V_{\alpha}^{\mathcal{A}}$.

For (iii), it suffices by induction hypothesis to show that $V_{\alpha}^{\mathcal{A}}\subseteq V_{\alpha+1}^{\mathcal{A}}$. So suppose that $b\in V_{\alpha}^{\mathcal{A}}$, and suppose that $\langle e,c\rangle \in b$. Then $ p e i_0\in \|c\in b\|$ and so by~(i) there is $\beta<\alpha$ such that $c\in V_{\beta}^{\mathcal{A}}$ and so by induction hypothesis we have $c\in V_{\alpha}^{\mathcal{A}}$. Hence $b\subseteq \mathcal{A}\times V_{\alpha}^{\mathcal{A}}$ which is just to say that $b$ is an element of $V_{\alpha+1}^{\mathcal{A}}$.
\end{proof}

Here, finally, is the proof of Theorem~\ref{thm:zfr-valued}, which of course resembles the proof that the axioms of set theory are valid on Boolean-valued models from e.g. \cite[pp. 37 ff]{Bell1985aa}:

\begin{proof}
By Proposition~\ref{prop:soundness}, it suffices to show that the axioms of $\mathsf{IZF}$ are valid. Recall that in the proof of Proposition~\ref{prop:itsuniform}, we showed that there is an index $i_0$ such that $i_0 \in \|a=a\|$ for all $a\in V_{\kappa}^{\mathcal{A}}$. This index $i_0$ is fixed throughout this proof.

To verify extensionality, by Proposition~\ref{prop:onpcas} choose $f,h$ such that $fen = e p n i_0$ and $h p e_0 e_1 = p (fe_0) (fe_1)$. Suppose $p e_0 e_1 \in \bigcap_{c\in V_{\kappa}^{\mathcal{A}}} (\|c\in a\Rightarrow c\in b\| \wedge \|c\in b \Rightarrow c\in a\|)$. We must show that $h p e_0 e_1\in \|a=b\|$, or what is the same that $f e_0\in p_0\|a=b\|$ and $f e_1\in p_1 \|a=b\|$. This is equivalent to showing that $\langle n,c\rangle \in a$ implies $f e_0 n = e_0 p n i_0  \in \|c\in b\|$ and $\langle n,c\rangle \in b$ implies $f e_1 n=e_1 p n i_0  \in \|c\in a\|$. We verify the first since the proof of the other is identical. So suppose that $\langle n,c\rangle \in a$. Then taking $d=c$ we trivially have $\exists \; d \; (\langle n,d\rangle \in a \; \wedge \; i_0 \in \|c=d\|)$. Then $p n i_0 \in \|c\in a\|$. Since $e_0\in p_0\|c\in a\Rightarrow c\in b\|$ we have that $e_0 p n i_0 \in \|c\in b\|$, so that we are now done verifying extensionality.

For pairing, suppose that $a,b$ are members of $V_{\alpha}^{\mathcal{A}}$. Fix an element $e_0$ of $\mathcal{A}$. Then $c=\{\langle e_0, a\rangle\cup \langle e_0, b\rangle\}$ is a member of $V_{\alpha+1}^{\mathcal{A}}$. By Proposition~\ref{prop:onpcas}, consider $h$ such that $hn  =p (p e_0 i_0)(p e_0 i_0)$. Then we may verify that $h: \top \leadsto \|a\in c\| \wedge \|b\in c\|$. For, by setting $d=a$ and $d^{\prime}=b$ we have $\exists \; d \; (\langle e_0, d\rangle \in c \; \wedge \; i_0 \in \|d=a\|)$ and $\exists \; d^{\prime} \; (\langle e_0, d^{\prime}\rangle \in c \; \wedge \; i_0 \in \|d=b\|)$. This finishes the verification of pairing. 

For the union axiom, let $a$ in $V^{\mathcal{A}}_{\kappa}$ be given and then set  $u = \{ \langle n,c \rangle : n \in \|\exists \; x \; (c \in x\ \wedge\ x \in a)\|\}$. Then $u$ is in $V_\kappa^{\mathcal{A}}$ by Proposition~\ref{lem:mccartylem}. By Proposition~\ref{prop:onpcas}, choose $f$ such that $fn= p n i_0$. Then we show that $f$ is the witness to the reduction $\|\exists \; x \; (c\in x \; \wedge \; x\in a)\| \leq \|c\in u\|$. So suppose that $n$ is in $\|\exists \; x \; (c\in x \; \wedge \; x\in a)\|$. Then by taking $d=c$ we have that $\exists \; d \; (\langle n,d\rangle\in u \; \wedge \; i_0\in \|c=d\|)$. Then $fn=pni_0 \in \|c\in u\|$. 

For power set, let $a$ in $V^{\mathcal{A}}_{\kappa}$ be given and set $P^{\mathcal{A}}(a) = \{\langle n,c\rangle: n\in \|c\subseteq a\|\}$. By Proposition~\ref{lem:mccartylem}, choose $\beta<\kappa$ such that $b\in V^{\mathcal{A}}_{\kappa}$ and $\|b\in a\|\equiv \top$ implies $b\in V^{\mathcal{A}}_{\beta}$. Supposing that $\langle n,c\rangle \in P^{\mathcal{A}}(a)$, we first show that $c\in V^{\mathcal{A}}_{\beta+1}$. So suppose that $\langle e,b\rangle \in c$. Then taking $d=b$ we have $\exists \;d \; \langle e,d\rangle \in c \; \wedge \; i_0\in \|b=d\|$, so that $p{e}i_0 \in \|b\in c\|$. Then by hypothesis on $n$, we have that $n p e i_0 \in \|b\in a\|$. Hence $\|b\in a\|\equiv \top$ and thus $b\in V^{\mathcal{A}}_{\beta}$. Thus indeed $c\in V^{\mathcal{A}}_{\beta+1}$ for all $\langle n,c\rangle \in P^{\mathcal{A}}(a)$. Hence $P^{\mathcal{A}}(a)\in V^{\mathcal{A}}_{\beta+2}$. Now consider, just as in the verification of the union axiom, the element $f$ such that  $fn= p n i_0$. Then we show that $f$ is the witness to the reduction $\|c\subseteq a\| \leq \|c\in P^{\mathcal{A}}(a)\|$. So suppose that $n\in \|c\subseteq a\|$. Then taking $d=c$ we have that $\exists \; d \; (\langle n,d\rangle \in P^{\mathcal{A}}(a) \; \wedge \; i_0\in \|c=d\|)$. Hence $fn= p n i_0 \in \| c\subseteq P^{\mathcal{A}}(a)\|$.

For separation, let $a$ in $V^{\mathcal{A}}_{\kappa}$ and a formula $\varphi(x)$ in the signature be given, which perhaps has parameters from $V^{\mathcal{A}}_{\kappa}$. Define $b= \{ \langle  e_0,d \rangle\ : \ e_0 \in \|d \in a\ \wedge\ \varphi(d)\| \}$. Again by similar appeal to Proposition~\ref{lem:mccartylem}, $b$ is an element of~$V_\kappa^{\mathcal{A}}$. Then by appeal to $V_{\kappa}^{\mathcal{A}}$ being a uniform $P(\mathcal{A})$-valued structure and substitution therein (Proposition~\ref{prop:uniform valuation} in conjunction with Proposition~\ref{prop:itsuniform}), let $g$ be an element  witnessing the substitution $\| d\in a \; \wedge \; \varphi(d) \; \wedge \; c=d\|\leq \|c\in a \; \wedge \; \varphi(c)\|$. Suppose that $p e_0 e_1 \in \|c\in b\|$. Then $\exists \; d \; (\langle e_0, d\rangle\in b \; \wedge \; e_1\in \|c=d\|)$. Then by definition of $b$, we have $e_0\in \|d\in a \; \wedge \; \varphi(d)\|$. Then $g (p e_0 e_1)\in \|c\in a  \; \wedge \; \varphi(c)\|$. Hence $g$ is also a witness to the reduction $\|c\in b\|\leq \|c\in a \; \wedge \; \varphi(c)\|$. Conversely, suppose that $e_0\in \|c\in a \; \wedge \; \varphi(c)\|$. Then by taking $d=c$, we have that $\exists \; d \; (\langle e_0, d\rangle \in b \; \wedge \; i_0\in \|c=d\|)$. Then $p{e_0}{i_0} \in \|c\in b\|$. Hence the element $f e_0=p e_0  i_0$ is the witness to the reduction $\|c\in a \; \wedge \; \varphi(c)\|\leq \|c\in b\|$.

For collection, suppose that $a$ from $V^{\mathcal{A}}_{\kappa}$ is given and that $\varphi(x,y)$ is a formula in the signature, perhaps with parameters from $V^{\mathcal{A}}_{\kappa}$. Choose $\alpha$ such that $\|c \in a\|\equiv \top$ implies $c \in V_\alpha^{\mathcal{A}}$, and set $Z = \{ c \in V_{\alpha}^{\mathcal{A}}: \|c \in a\|\equiv \top\}$. For each triple $(c,e,k)$ such that $c \in Z$, $e \in \|\forall x \in a \; \exists \; y \; \varphi(x,y) \|$, $k \in \|c \in a \|$, define $\gamma (c, e, k)$ to be the least ordinal $\gamma<\kappa$ such that there is $d \in V_\gamma^{\mathcal{A}}$ with $ek\in \|\varphi(c,d)\|$. Then since $\kappa$ is strongly inaccessible and satisfies $\kappa>\left|\mathcal{A}\right|$, choose $\beta<\kappa$ such that $\beta$ is strictly greater than all such  $\gamma (c, e, k)$. Let $b=\mathcal{A} \times V_\beta^{\mathcal{A}}$ which is trivially an element of $V_{\kappa}^{\mathcal{A}}$. Fix an element $e_0$ of $\mathcal{A}$. Our desired $e^\prime$ is given by $e^{\prime} e k = p (p e_0 i_0) (ek)$, which exists by Proposition~\ref{prop:onpcas}. For suppose $e\in \| \forall x \in a \ \exists \; y \; \varphi(x,y)\|$. We must establish that $e^\prime e \in \| \forall \; x \in a \; \exists \; y\in b\; \varphi(x,y)\|$. So suppose $c$ is fixed and $k \in \|c\in a\|$. By construction of $\beta$, choose $d\in V_{\beta}^{\mathcal{A}}$ with $ek\in \|\varphi(c,d)\|$. Then $p e_0 i_0\in \|d\in b\|$ and so $p (p e_0 i_0) (ek) \in \|\exists \; y\in b \; \varphi(c,y)\|$.

For the axiom of infinity, recall the elements $\widetilde{n}$ from $\mathcal{A}$ from~(\ref{eqn:currynumerals}) and the elements $\overline{n}$ and $\overline{\omega}$ from $V_{\kappa}^{\mathcal{A}}$ from~(\ref{eqn:mccartynumbers}). It suffices to show that there are $e^{\prime},e\in \mathcal{A}$ such that $e^{\prime} \in  \|\overline{0} \in \overline{\omega} \|$ and $e \in \|c\in \overline{\omega}\Rightarrow \exists \; y \; (c\in y \; \wedge y\in \overline{\omega})\|$. For $e^{\prime}$, simply take $e^{\prime}=p\widetilde{0}{i_0}$. Then we have $\langle \widetilde{0}, \overline{0}\rangle \in \overline{\omega}$ and $i_0\in \|\overline{0}=\overline{0}\|$ and hence $e^{\prime}\in  \|\overline{0} \in \overline{\omega} \| $. Now we work on $e$. By appeal to $V_{\kappa}^{\mathcal{A}}$ being a uniform $P(\mathcal{A})$-valued structure and substitution therein (Proposition~\ref{prop:uniform valuation} in conjunction with Proposition~\ref{prop:itsuniform}), choose $k_0$ such that for all $u,v,w$ one has $k_0\in \| (u\in v \; \wedge \; u=w)\Rightarrow w\in v\|$. Using primitive recursion on Curry numerals \cite[Proposition 1.3.5 p. 12]{Oosten2008aa}, we choose $e$ such that $e p\widetilde{n}e_1= p({k_0}p(p\widetilde{n}{i_0})e_1)( p(\widetilde{n+1})i_0)$. Suppose that $p e_0 e_1 \in \|c\in \overline{\omega}\|$. Then there is $d$ such that $\langle e_0, d\rangle \in \overline{\omega}$ and $e_1\in \|d=c\|$. Then $\langle e_0, d\rangle = \langle \widetilde{n}, \overline{n}\rangle$ for some $n<\omega$ and hence $e_1\in \|\overline{n}=c\|$. Since $p\widetilde{n} i_0 \in \|\overline{n}\in \overline{n+1}\|$, we have $k_0(p(p(\widetilde{n} i_0))e_1)\in \|c\in \overline{n+1}\|$. Since $p(\widetilde{n+1})i_0\in \|\overline{n+1}\in \overline{\omega}\|$, we thus have that $e p\widetilde{n}e_1=e pe_0e_1\in \|\exists \; y \; (c\in y \; \wedge y\in \overline{\omega})\|$.

For the the induction schema, we need to show that there is $e$ such that $e\in \|[\forall \; x \; (\forall \; y\in x \; \varphi(y))\Rightarrow \varphi(x)]\Rightarrow (\forall \; x \; \varphi(x))\|$. Let $k_0$ such that for all $a,b$ in $V_{\kappa}^{\mathcal{A}}$ we have $k_0\in \| \varphi(b) \Rightarrow (b\in a \Rightarrow \varphi(b))\|$. By Proposition~\ref{prop:onpcas}, choose $f$ such that $fem = m k_0 (e m)$. By the recursion theorem \cite[Proposition 1.3.4 p. 12]{Oosten2008aa}, choose an $e$ such that $em = fem = m k_0 (e m)$. Suppose that 
\begin{equation}\label{eqn:thisisM}
m\in \|\forall \; x \; (\forall \; y\in x \; \varphi(y))\Rightarrow \varphi(x)\|
\end{equation}
We show by induction on $\alpha$ that $a\in V_{\alpha}^{\mathcal{A}} \mbox{ implies } em\in \|\varphi(a)\|$. Suppose it holds for all $\beta<\alpha$. Suppose that $a\in V_{\alpha}^{\mathcal{A}}$. Then we claim ${k_0}(em)\in \|\forall \; y\in a \; \varphi(y)\|$. For this claim, it suffices to show that for all $b\in V_{\kappa}^{\mathcal{A}}$ we have  ${k_0}em\in \|b\in a\Rightarrow \varphi(b)\|$. So suppose that $n\in \|b\in a\|$. Then by the second part of Proposition~\ref{lem:mccartylem}, we have that $b\in V_{\beta}^{\mathcal{A}}$ for some $\beta<\alpha$. Then by induction hypothesis we have $em\in \|\varphi(b)\|$. Then by definition of $k_0$ we have that ${k_0}(em)\in \|b\in a\Rightarrow \varphi(b)\|$, which is what we wanted to show. So we've succeeded in showing the claim. Then by the hypothesis on~$m$ recorded in equation~(\ref{eqn:thisisM}), we have  $m{k_0}(em)\in \|\varphi(a)\|$, which by the choice of~$e$ implies that $em\in \|\varphi(a)\|$, which is what we wanted to show.
\end{proof}

\section{Acknowledgements}

This paper has been bettered by comments from and conversations with the following people, whom we warmly thank: Marianna Antonutti, Jeff Barrett, Benno van den Berg, Tim Carlson, Walter Dean, Michael Ernst, Rohan French, Leon Horsten, Bob Lubarsky, Jaap van Oosten, and Kai Wehmeier. Thanks also to the anonymous referees for very helpful comments.

\bibliography{rin-walsh.bib}

\end{document}